\documentclass[10pt]{article}

\usepackage{amsmath,amssymb,amsthm,amsfonts,verbatim}
\usepackage{enumerate}
\usepackage{microtype}
\usepackage{url}
\usepackage[all,arc]{xy}
\usepackage{mathrsfs}
\usepackage{listings, color} 
\usepackage[pdftex]{graphicx}
\usepackage[section]{placeins}
\usepackage{soul}
\usepackage{tikz}
\usepackage{ esint }
\usepackage{graphicx}
\usepackage{hyperref}
\usepackage[makeroom]{cancel}
\usepackage{caption}
\usepackage[text={6.75in,9.5in},centering,letterpaper]{geometry}

\tikzset{node distance=2cm, auto}
 \allowdisplaybreaks
 \numberwithin{equation}{section} 

\newtheoremstyle{quest}{\topsep}{\topsep}{}{}{\bfseries}{}{ }{\thmname{#1}\thmnote{ #3}.}
\theoremstyle{quest}

\theoremstyle{plain}
\theoremstyle{definition}
\newtheorem{theorem}{Theorem}[section]
\newtheorem{corollary}[theorem]{Corollary}
\newtheorem{proposition}[theorem]{Proposition}

\newtheorem{lemma}[theorem]{Lemma}

\newtheorem{hypothesis}{Hypothesis}
\newtheorem{definition}[theorem]{Definition}

\newtheorem{remark}[theorem]{Remark}

\usepackage{listings}
\usepackage{color}
 
\definecolor{dkgreen}{rgb}{0,0.6,0}
\definecolor{gray}{rgb}{0.5,0.5,0.5}
\definecolor{mauve}{rgb}{0.58,0,0.82}
 
\title{\vspace{-50pt} Mixing in Reaction-Diffusion Systems: \\ Large Phase Offsets}
\author{ \Large Sameer Iyer \footnote{\url{sameer_iyer@brown.edu}. Division of Applied Mathematics, Brown University, 182 George Street, Providence, RI 02912, USA. Supported by NSF grants DMS 1611695 and DMS 1209437. } \hspace{10 mm} Bjorn Sandstede \footnote{\url{bjorn_sandstede@brown.edu}. Division of Applied Mathematics, Brown University, 182 George Street, Providence, RI 02912, USA. Partially supported by NSF grant DMS 1408742.}}
\date{October 20, 2016}

\DeclareMathOperator{\supp}{\text{supp}}

\newcommand{\p}{\ensuremath{\partial}}
\newcommand{\n}{\ensuremath{\nonumber}}

\newcommand{\ud}{\,\mathrm{d}}

\makeatletter
\def\@Aboxed#1&#2\ENDDNE{%
  \settowidth\@tempdima{$\displaystyle#1{}$}%
  \addtolength\@tempdima{\fboxsep}%
  \addtolength\@tempdima{\fboxrule}%
  \global\@tempdima=\@tempdima
  \kern\@tempdima
  &
  \kern-\@tempdima
  \fcolorbox{red}{yellow}{$\displaystyle #1#2$}
}
\makeatother

\begin{document}
\maketitle
\vspace{-30pt}

\begin{abstract}
We consider Reaction-Diffusion systems on $\mathbb{R}$, and prove diffusive mixing of asymptotic states $u_0(kx - \phi_{\pm}, k)$, where $u_0$ is a periodic wave. Our analysis is the first to treat arbitrarily large phase-offsets $\phi_d = \phi_{+}- \phi_{-}$, so long as this offset proceeds in a sufficiently regular manner. The offset $\phi_d$ completely determines the size of the asymptotic profiles, placing our analysis in the large data setting. In addition, the present result is a global stability result, in the sense that the class of initial data considered are not near the asymptotic profile in any sense. We prove global existence, decay, and asymptotic self-similarity of the associated wavenumber equation. We develop a functional framework to handle the linearized operator around large Burgers profiles via the exact integrability of the underlying Burgers flow. This framework enables us to prove a crucial, new mean-zero coercivity estimate, which we then combine with a nonlinear energy method.  
\end{abstract}

\section{Introduction}

We consider Reaction-Diffusion systems posed on the spatially extended domain, $\mathbb{R}$: 
\begin{equation} \label{eqn.rxn.diff.}
\partial_tu - D \partial_{xx} u = f(u).
\end{equation}

Here, $t \ge 1$, $x \in \mathbb{R}$, $u: \mathbb{R} \rightarrow \mathbb{R}^d$, and $D \in \mathbb{R}^{d \times d}$ is a symmetric, positive definite reaction matrix, and $f$ is a smooth nonlinearity. We will consider situations when such systems possess periodic, traveling wave solutions of the form $u_0(kx-\omega t; k)$ for wavenumbers $k$ in a specific range, $(k_1, k_2)$. Here $\omega = \omega(k)$ is called the nonlinear dispersion relation. We are interested in the stability of these periodic waves under the following perturbation: consider an initial condition of the form
\begin{equation} \label{initial.data}
u(1,x) = u_0(k x + \phi_0(x); k) + v_0(x), \hspace{3 mm} \phi_0(x) \rightarrow \phi_{\pm} \text{ as } x \rightarrow \pm \infty. 
\end{equation}

The function $\phi_0(x)$ represents initial phase off-set, which has magnitude: 
\begin{equation} \label{defn.phi.d}
\phi_d = \phi_+ - \phi_-. 
\end{equation}

The wavenumber is defined to be the derivative of the phase, and so (\ref{defn.phi.d}) implies that the wavenumber is given a localized $L^1$ initial perturbation, with $L^1$ norm controlling $\phi_d$. Indeed, in \cite{Sandstede}, solutions with initial data of the type (\ref{initial.data}) are proven to mix diffusively, supposing that $\phi_d$ is sufficiently small. Similar hypotheses were used in \cite{Zumbrun1}, \cite{Zumbrun2}, \cite{Zumbrun3} (all requiring the phase-offset to be small initially) to establish nonlinear stability of the wave-trains in question. In the present paper, we establish the diffusive mixing of wavenumber perturbations while allowing the initial phase off-set, $\phi_d$, to be arbitrarily large, so long as the phase variation from $\phi_-$ at $x = -\infty$ to $\phi_+$ at $x = \infty$ occurs in a sufficiently regular manner (see (\ref{as.small}) - (\ref{shape.condition}) for the precise requirements we impose). We emphasize that the size of the phase offset, (\ref{defn.phi.d}), determines the size of the asymptotic phase profile that we will show convergence to, and therefore \textit{this is a large-data asymptotics problem.} 

To state our theorem, we define the space $L^1_\xi(1)$, which is the natural counterpart to the spatial $\mathcal{C}^1$ norm in frequency variables, $\xi$, (where the majority of our analysis takes place). Define the norms: 
\begin{equation}
||\hat{g}||_{L^1_\xi(1)} := ||\hat{g}(\xi)(1+|\xi|)||_{L^1}, \hspace{3 mm} ||g||_{H^2(2)} := ||(1+z^2) g(z)||_{H^2} = ||(1 + |\xi|^2) \hat{g}(\xi)||_{H^2}. 
\end{equation}

In order to state our theorem, we must also introduce the following explicit profiles: The function $e_\ast$ will denote the Gaussian error function: 
\begin{equation}
e_\ast(z) = \frac{1}{\sqrt{4\pi}} \int_{-\infty}^z e^{-\frac{\theta^2}{4}} d\theta. 
\end{equation}

We define the following self-similar Burgers profile:
\begin{equation} \label{fA}
f_A^\ast(z) := \frac{Ae'_\ast(z)}{1 + Ae_\ast(z)}, \hspace{3 mm} \log(1+A) = \phi_d. 
\end{equation}

For our analysis, it is important to understand the size of $f_A^\ast$ and variants thereof (defined by $\bar{b}$ in (\ref{barb.1})) in various norms. These estimates are given in subsection \ref{explicit}, see specifically (\ref{fAlarge.1}), (\ref{fAlarge}), and (\ref{barb.small.ck}).

\subsection{Main Results}

We will seek $2\pi$-periodic traveling wave solutions, $u(t,x) = u_0(k_0 x - \omega_0 t)$, to the system (\ref{eqn.rxn.diff.}). Denoting the characteristic coordinates by $\bar{\theta} = k_0 x - \omega_0 t$, this is equivalent to solutions, $u = u_0(\bar{\theta})$, to the following problem: 
\begin{align} \label{WT.1}
k^2 D\partial_{\bar{\theta}}^2 u + \omega \partial_{\bar{\theta}}u + f(u) = 0, \text{ for } k = k_0, \omega = \omega_0. 
\end{align}

The linear operator $\mathcal{L}_0$ is obtained upon linearizing the above equation around $u_0$: 
\begin{align} \label{WT.3}
\mathcal{L}_0 = \mathcal{L}(k_0) = k_0^2 D\partial_{\bar{\theta}}^2 + \omega_0 \partial_{\bar{\theta}} + f'(u_0(\theta)). 
\end{align}

$\mathcal{L}_0$ is closed and densely defined on $L^2_{per}(0, 2\pi)$, with domain $D(\mathcal{L}_0) = H^2_{per}(0, 2\pi)$. We will assume $0$ is a simple eigenvalue of $\mathcal{L}_0$ on $L^2_{per}(0,2\pi)$. This means that the the null space is spanned by $\partial_{\bar{\theta}} u_0$, which can be seen by a straightforward application of $\partial_{\bar{\theta}}$ to the system (\ref{WT.1}). For each fixed $\xi \in \mathbb{R}$, we can seek solutions to the linearized problem $\partial_t v = \mathcal{L}_0 v$ in the form: 
\begin{align}
u(t, \bar{\theta}) = e^{\lambda(\xi) t + i \xi \theta/ k_0} \tilde{v}(\theta, \xi). 
\end{align}

Inserting such a solution into (\ref{WT.3}) gives the eigenvalue problem: 
\begin{equation}
\tilde{L}(\xi) \tilde{v} = \lambda(\xi) \tilde{v}, 
\end{equation} 

where:
\begin{align}\nonumber
\tilde{\mathcal{L}}(\xi)\tilde{v} &= e^{-i\xi \bar{\theta}/k_0} \mathcal{L} \Big( e^{i\xi \bar{\theta}/k_0} \tilde{v}(\bar{\theta}, \xi) \Big) \\ \label{WT.2}
& = k_0^2 D  (\partial_{\bar{\theta}}+ i\xi/k_0)^2 \tilde{v} + \omega (\partial_{\bar{\theta}}+ i\xi/k_0) \tilde{v} + f'(u_0(\bar{\theta})) \tilde{v}
\end{align}

For each $\xi$, the operator $\tilde{L}(\xi)$ has discrete spectrum on $L^2_{per}(0, 2\pi)$ because the resolvent is compact. Hence, the eigenvalues of $\tilde{L}(\xi)$ are then labeled as $\lambda_j(\xi)$, and are assumed to be ordered such that $Re \lambda_{j+1}(0) \le Re \lambda_j(0)$. We will also assume that $\lambda_1(0)$ is the right-most element in the spectrum for $\xi = 0$. With the appropriate notation now fixed, let us state the hypothesis we make: 

\begin{hypothesis} \label{Hyp.1} Equation (\ref{eqn.rxn.diff.}) admits spectrally stable wave-train solutions, $u(t,x) = u_0(\bar{\theta})$, with $\bar{\theta} = k_0x - \omega_0 t$ for appropriate $k_0 \neq 0, \omega_0 \in \mathbb{R}$. $u_0$ is $2\pi$ periodic. This means the following: the linearized operator of (\ref{eqn.rxn.diff.}) around $u_0$ has a simple eigenvalue at $\lambda = 0$. The linear dispersion relation, denoted $\lambda(\xi)$, satisfies $\lambda(0) = 0$, is dissipative, so $\lambda''(0) < 0$, and there exist constants $\sigma_0, \xi_0, \alpha_0 > 0$ such that $Re \lambda(\xi) < -\sigma_0$ for $|\xi| > \xi_0$, and for the region $|\xi| \le \xi_0$, we have $Re \lambda(\xi) < -\alpha_0 \xi^2$. All other eigenvalues $\lambda_j(\xi)$, satisfy $Re \lambda_j(\xi) \le - \sigma_0$ for all $\xi \in [-\frac{k}{2}, \frac{k}{2}]$. 
\end{hypothesis}

Fix now one wave-train satisfying Hypothesis \ref{Hyp.1}, and define the following parameters: 
\begin{align} \n
&\omega_0 = \omega(k_0), \hspace{2 mm} c_p = \frac{\omega_0}{ k_0}, \hspace{2 mm} c_g = \omega'(k_0), \hspace{2 mm} \alpha = -\frac{\lambda''(0)}{2}, \\
& \hspace{15 mm} \beta = -\frac{1}{2}\omega''(k_0), \hspace{3 mm} \bar{\theta} = k_0 x - \omega_0 (t-1). 
\end{align}

For the linear dispersion relation, we then know that: 
\begin{align}
\lambda(\xi) = i(c_p - c_g)\xi - \alpha \xi^2 + o(\xi^3). 
\end{align}

Having fixed $k_0$, we will immediately take the convention that we call this parameter $k$ from now on. In order to state our main result, we must introduce the Bloch transform:  
\begin{align} \label{Bloch.1}
\mathcal{J}(f) := \tilde{f}(\xi, \nu) := \sum_{j \in \mathbb{Z}} e^{ij\nu} \hat{f}(\xi + jk), \hspace{5 mm} \xi \in [-\frac{k}{2}, \frac{k}{2}] \subset \mathbb{R}, \hspace{5 mm} \nu \in [0, 2\pi], 
\end{align}

Relevant properties of the Bloch transform will be discussed Section \ref{Sec.Mod}. Our main result is:
\begin{theorem} \label{thm.main} Fix any $0 < \phi_d < \infty$, and pick any $\rho_\ast > 0$, possibly large. Let $u_0(\cdot; k)$ be a spectrally stable wave-train with the dispersion relation $\omega = \omega(k)$ such that $\omega''(k) = \beta \neq 0$. Consider initial data in the form specified in (\ref{initial.data}), satisfying:  
\begin{equation} \label{given.size}
||\partial_x \phi_0||_{H^2(2)} + ||v_0||_{H^2(2)} + \phi_d  < \infty. \footnote{By the weighted embedding $H^2(2) \hookrightarrow L^1$, it is clear that $\phi_d \le ||\partial_x \phi_0||_{H^2(2)}$, so it is impossible to make the $H^2(2)$ norm of $\partial_x \phi(x)$ small while retaining large phase-offsets $\phi_d$.}
\end{equation}

Then there exists $\delta > 0$, sufficiently small, depending only on $\phi_d$ and not on $\rho_\ast$ such that\footnote{Note that by estimate (\ref{barb.small.ck}), conditions (\ref{as.small}) - (\ref{shape.condition}) define a large class of data. (\ref{shape.condition}) enforces that the initial datum $\p_x \phi_0$ must be comparable (in a weak sense) to $\frac{1}{\sqrt{T_0}}f_A^\ast(\frac{\cdot}{\sqrt{T_0}})$, which in particular implies that $\p_x \phi_0$ cannot be taken close to $f_A^\ast$ itself in any norm. Heuristically, $\frac{1}{\sqrt{T_0}}f_A^\ast(\frac{\cdot}{\sqrt{T_0}})$ defines a stable manifold of algebraic decay, to which we will relate the evolution of $\p_x \phi$.}:
\begin{align} \label{as.small}
&||\mathcal{F}\{\partial_x \phi_0 \}||_{L^1_\xi(1)} + ||(\partial_\nu + i\xi)^j \tilde{v}_0||_{L^2([-\frac{k}{2}, \frac{k}{2}] \times \mathbb{T} )} \le \delta, \text{ for $j = 0,1$ and if } \\ \label{shape.condition}
&|| \frac{1}{\sqrt{T_0}} f_A^\ast\Big( \frac{\cdot}{\sqrt{T_0}} \Big)  -\partial_x \phi_0(\cdot)||_{H^2(2)} \le \rho_\ast, \text{ for } T_0 = \Big(\frac{\phi_d}{\delta}\Big)^2,
\end{align}

the solution to (\ref{eqn.rxn.diff.}) exists globally, and can be written as the modulated wave-train: 
\begin{align}
u(t,x) = u_0(kx-\omega (t-1) + \phi_A^\ast(\frac{x-c_g(t-1))}{\sqrt{T_0}};k) + v(t,x), 
\end{align}

where:
\begin{align} \label{pA}
\phi_A^\ast(z) = \frac{\alpha}{\beta} \log \{1 + A e_\ast(\frac{x}{\sqrt{\alpha t}}) \}, \hspace{3 mm} \log\{1 + A \} = \phi_d,
\end{align}

and 
\begin{equation}
||v(t,\cdot)||_{L^\infty} \le C(\phi_d, \sigma) t^{-\frac{1}{2}+\sigma}, \text{ for any } \sigma > 0. 
\end{equation}

\end{theorem}

\begin{remark}[Notation] Throughout the analysis, we use $C(\phi_d)$, and $C(\tilde{\phi_d})$ to mean constants which depend poorly on large $\phi_d, \tilde{\phi}_d$, respectively, without renaming these constants between lines. Similarly, we use the notation $\lesssim$ to mean $\le C$, where $C$ is some hard constant independent of the parameters at play. We will depict function spaces in frequency coordinates explicitly, such as $||\cdot||_{L^1_\xi}$. Thus, $||\cdot||_{H^1}$ means the usual, spatial $H^1$ norm. In the case of $H^2(2)$, the Fourier transform is an isomorphism between $H^2(2) \rightarrow H^2_\xi(2)$ via Plancherel, and so there is no ambiguity in the notation $H^2(2)$. Finally, we take our time variable to start at $t = 1$ as opposed to $t = 0$ which notationally allows us to replace $1+t$ by $t$ in many estimates. 
\end{remark}

\subsection{Outline of Proof:}

We briefly recall the main ideas in \cite{Sandstede}. The first step is to linearize the equation (\ref{eqn.rxn.diff.}) around the wave-train, $u_0(\cdot; k)$. The linearized operator possesses essential spectrum up to the imaginary axis. Two equations are subsequently extracted which then govern the dynamics of the perturbation: the first is the equation for the local wavenumber, which contains the spectral modes approaching the imaginary axis, and the second equation only has spectral modes in $\{\text{Re}(z) < \lambda_0 < 0 \}$. As solutions to the second, stable, equation decay exponentially, the first equation for the localized wavenumber governs the dynamics of the perturbation. Note that the local wavenumber is defined to be the derivative of the phase: 
\begin{equation}
u^c \approx \partial_x \phi(t,x),
\end{equation}

and so $||u^c||_{L^1} = \phi_d$. To ease notation, let us also now set the convention for the introduction only that $\theta \approx x$. $u^c$ obeys a Burgers type-equation:
\begin{equation} \label{eqn.crit.0}
u^c_t - \alpha u^c_{xx} = \beta u^c u^c_{x} + \text{ Asymptotically ``Irrelevant" Terms}.
\end{equation}

As we will state below in Lemma \ref{L.suff}, our main result, Theorem \ref{thm.main}, will hold provided we can show that: 
\begin{equation} \label{seek.1}
||\sqrt{t}u^c(t, \sqrt{t}z) - f^\ast_A(z)||_{H^2(2)} \le C(\phi_d, \sigma) t^{-\frac{1}{2}+\sigma} \text{ for some } \sigma > 0. 
\end{equation}

The ``irrelevant" terms in (\ref{eqn.crit.0}) are in the sense of \cite{BKL}, which formally are expected not to contribute to the long time asymptotics of $u^c$. This is seen through the Renormalization Group (RG) iteration scheme, whose purpose is to extract the asymptotic behavior that we seek from (\ref{seek.1}), and whose basic notions we now set. 

\subsection*{The Renormalization Group:}

First, through the selection of an appropriate time-scale, called $L$, which is usually large, one discretizes the desired convergence from (\ref{seek.1}) into: 
\begin{equation} \label{seek.2}
||L^n u^c (L^{2n}, L^n z) - f_A^\ast(z)||_{H^2(2)}  \lesssim L^{-n+\sigma}, \hspace{3 mm} \text{ for any } \sigma > 0. 
\end{equation}

Here $L^{2n} = t$, and $z$ is the self-similar variable $z = \frac{x}{\sqrt{t}}$. Clearly, (\ref{seek.2}) is the discretized version of (\ref{seek.1}). This motivates the consideration of the renormalized solution, as well as the renormalization map, $R_L$:
\begin{equation} \label{RGDI}
u^{n,c}(t,z) := L^n u^c(L^{2n}t, L^nz), \hspace{5 mm} u^{n,c}(1,z) := R_L \{ u^{n-1,c}(L^2, z) \},
\end{equation}

where the map $R_L$ is defined on functions $f: \mathbb{R} \rightarrow \mathbb{R}$ via: 
\begin{equation}
R_L f(z) = Lf(Lz). 
\end{equation}

According to (\ref{seek.2}), the object of study are the sequence of initial datum, $\Big\{u^{n,c}(1,z) \Big\}_{n \ge 0}$. According to (\ref{RGDI}), the way to obtain $u^{n,c}(1,z)$ is iterative: first flow forward the equation governing $u^{n-1,c}$ until time $L^2$, and then apply the renormalization map $R_L$. 

In the case when $||u^c(1, \cdot)||_{H^2(2)} \le \delta$ (which corresponds to $\phi_d \le \delta$ through Sobolev embeddings) the size of the asymptotic fixed point $f_A^\ast$ is also order $\delta$. For small fixed points, the RG iteration introduced in (\cite{BKL}) proves to be an effective mechanism to pick out the anticipated convergence. Indeed, the relevant techniques for doing so in the present context were introduced in \cite{Sandstede}. 

Our main difficulty is that for large phase-offsets, $\phi_d$, the anticipated limit point of the RG iteration is large in the function space $H^2(2)$. Let us emphasize that the largeness of the limit point, $f_A^\ast$ is \textit{intrinsic, and is essentially independent of the functional framework with which we choose to work.} This can be seen easily via the representation (\ref{fA}), and the calculations (\ref{fAlarge.1}).  

\subsection*{Intermediate Regime, Energy Estimates:}

Our approach is as follows: the initial wavenumber offset ($\partial_x \phi$ in (\ref{given.size})) is large in $H^2(2)$, but small in $L^1_\xi(1)$. We find an exact Burgers solution, $\bar{b}$, which is similarly small in $L^1_\xi(1)$, large in $H^2(2)$, and which converges asymptotically to the anticipated fixed point $f_A^\ast$. That is: 
\begin{equation}
||u^{n,c}(1,\cdot) - f^\ast_A(\cdot)||_{H^2(2)} \le ||u^{n,c}(1,\cdot) - \bar{b}^{(n)}(1,\cdot)||_{H^2(2)} + ||\bar{b}^{(n)}(1, \cdot) - f^\ast_A(\cdot)||_{H^2(2)}. 
\end{equation}

Necessarily, $\bar{b}$, together with the renormalized versions $\bar{b}^{(n)}$, must satisfy $\int_{\mathbb{R}} \bar{b}^{(n)} = \int u^{n,c} = \phi_d$. We then study the comparison, $a^{(n)} := u^{n,c} - \bar{b}^{(n)}$, which has zero mean:
\begin{align}
a^{(n)} := u^{n,c} - \bar{b}^{n}, \hspace{5 mm} \int_{\mathbb{R}} a^{(n)} = 0.
\end{align}

The mean-zero component $a$ solves a linearized Burgers equation: 
\begin{align} \label{alpha.lin.1}
&\Big( \partial_t - \partial_{xx} \Big) a^{(n)} = \bar{b}^{(n)}_x a^{(n)} + \bar{b}^{(n)} a^{(n)}_x + a^{(n)} a^{(n)}_x + \mathcal{O}(L^{-n}), \\
&\text{Initial Datum: }g^{(n)} := a^{(n)}(1,\cdot).
\end{align}

The $L^1_\xi(1)$ smallness of $a$ persists for an ``intermediate" time range after which $u^c, \bar{b}$ grow to size $\phi_d$. In this intermediate time range, we show that the iterates $a^{(n)}$ are driven down to size $\epsilon$: $a^{(n)} \in B_\epsilon(0) \subset H^2(2)$. This is achieved via a nonlinear energy method which effectively captures the interaction between the persisting $L^1_\xi(1)$ smallness, the Burgers type nonlinearities, and the weights required in controlling the $H^2(2)$ Sobolev norm. Importantly, the intermediate time range is propagated long enough to make the higher order, irrelevant terms in (\ref{alpha.lin.1}) become sufficiently irrelevant. 

More specifically, our energy estimates are applied to the quantity $\gamma$, defined as in: 
\begin{equation}
a^{(n)} = \bar{a}^{(n)} + \gamma^{(n)}, \text{ where } \bar{a}^{(n)} = e^{\partial_{xx}(t-1)} g^{(n)}(z).
\end{equation}

Thus, $a^{(n)}$ is separated out into the linear flow, which retains the mean-zero property, and $\gamma$, which represents the local, nonlinear deviation. The estimates that we obtain are summarized:
\begin{align} \label{LCE.1}
&\text{Linear Contractive Estimate: } ||R_L\{\bar{a}^{(n)}(L^2,z) \}||_{H^2(2)} \lesssim \frac{1}{L}||g^{(n)}||_{H^2(2)}, \\ \nonumber
&\text{Nonlinear Deviation Estimate: }||\gamma^{(n)}(L^2, z)||_{H^2(2)} \lesssim J(L) \Big( \delta_w ||g^{(n)}||_{H^2(2)} + L^{-2n}C(\phi_d) \\ \label{LCE.2}
& \hspace{30 mm} + L^{-n(1-p)} \delta_w ||\tilde{g}^{(n,s)}||_{B_{L^n}(2,2,2)} \Big).
\end{align}

Here, $\delta_w$ is a small parameter, and $J(L), C(\phi_d)$ are functions which depend poorly on their arguments, and are independent of $n$. The Bloch norms appearing on the right-hand side above are introduced in (\ref{norm.B}). The estimates (\ref{LCE.1}) - (\ref{LCE.2}) can be iterated finitely many times during this intermediate time range, say $n \in 1,...,N_0$ for some large but finite $N_0 < \infty$, in order to obtain: 
\begin{align} \label{eps.intro}
||g^{N_0}||_{H^2(2)} \le \epsilon, \text{ for some } \epsilon \text{ small.}
\end{align}

\subsection*{Asymptotic Regime, Coercivity Estimates:}

Once $a^{(n)}$ has reached $B_\epsilon(0) \subset H^2(2)$, we lose any smallness of $u^c, \bar{b}$. Our energy method then becomes ineffective, as the linearized terms $\bar{b}_x a + \bar{b} a_x$ in (\ref{alpha.lin.1}) are no longer small, which then may obstruct the required coercivity. That is, the following quantity:
\begin{align} \nonumber
\int \Big( \partial_t - \partial_{xx} \Big) a^{(n)} \cdot a^{(n)} &- \int \Big( \bar{b}^{(n)}_x a^{(n)} + \bar{b}^{(n)} a^{(n)}_x \Big) \cdot a^{(n)}
\end{align} 

cannot be bounded below by:
\begin{align} \label{fail.1}
\frac{\partial_t}{2} \int \Big|a^{(n)}\Big|^2 + \int \Big|a^{(n)}_x\Big|^2 \text{ for } n > N_0. 
\end{align}

To combat this, we develop a functional framework in Section \ref{s.ff} to study the semigroup associated with the linearized operator: 
\begin{equation}
S_{\bar{b}} = \partial_{xx} + \partial_x \Big( \bar{b} \cdot \Big), \text{ and } \Phi_{\bar{b}} \text{ the corresponding flow map.}
\end{equation}

This framework is built on the nonlinear Cole-Hopf map, $\mathcal{N}$, defined in (\ref{defn.CH}), which takes localized $L^1$ solutions to Burgers equation to non-integrable solutions to the heat-equation. Linearizing this map around $\bar{b}$ enables us to study $S_{\bar{b}}$ without requiring any smallness on the Burgers solution $\bar{b}$. The crucial coercivity estimate that we prove (see Proposition \ref{prop.Burg.1}) is: 
\begin{equation} \label{contraction.Sb.intro}
||R_L \Phi_{\bar{b}}(L^2-1) g||_{H^2(2)} \le \frac{C(\phi_d)}{L} ||g||_{H^2(2)} \hspace{3 mm} \text{ if } \hspace{3 mm}\int_{\mathbb{R}} g = 0. 
\end{equation}

This coercivity estimate is then used in place of (\ref{fail.1}) by writing the $a^{(n)}$-endpoint using Duhamel: 
\begin{align}
a^{(n)}(L^2, z) &= \Phi_{\bar{b}^{(n)}}(L^2-1)  g^{(n)}  + \text{ Quadratic and Higher Terms},
\end{align}

Applying $R_L$ to iterate forward then gives: 
\begin{align} \nonumber
||g^{(n+1)}||_{H^2(2)} &= ||R_L \alpha^{(n)}(L^2, z)||_{H^2(2)} \le || R_L \Phi_{\bar{b}^{(n)}}(L^2-1)  g^{(n)}||_{H^2(2)} + \text{ Quadratic} \\ \label{linear.decay}
& \lesssim \frac{1}{L} ||g^{(n)}||_{H^2(2)} + \text{	Quadratic}. 
\end{align}

Once linear decay (contraction) has been recovered above (\ref{linear.decay}), we are able to send $n \rightarrow \infty$ to obtain (\ref{seek.1}) and consequently prove our main theorem. 

\section{Modulation Equations} \label{Sec.Mod}

In this section, we will extract the equations for the phase (the derivative of which corresponds to the wavenumber equation) and the stable component of the perturbation. This procedure has been carried out in \cite{Sandstede}, and does not change for our setup. We will briefly outline the main steps for the purposes of self-containment, but we refer the reader to \cite{Sandstede} for details and proofs. 

Let us introduce the characteristic coordinates $\bar{\theta} = \bar{\theta}(t,x) = kx - \omega t$. For each fixed $t$, one can obtain $x$ from $\bar{\theta}$ and vice-versa. Rewriting the original system, (\ref{eqn.rxn.diff.}) in $(t, \bar{\theta})$ coordinates yields:
\begin{align} \label{eqn.bar}
\partial_t u = k^2 D\partial_{\bar{\theta}}^2 u + \omega \partial_{\bar{\theta}} u + f(u).
\end{align}

It seems natural to seek a solution to the above system, (\ref{eqn.bar}), of the type: 
\begin{align}
u(t,x) &= u_0(\bar{\theta}(t,x) + \phi(t,x); k + \partial_x \phi(t,x)) + v(t,x) \\
& = u_0(\bar{\theta} + \phi(t,\bar{\theta}); k + k\partial_{\bar{\theta}} \phi(t, \bar{\theta}) + v(t,\bar{\theta}). 
\end{align}

However, as explained in \cite{Sandstede}, Remark 3.1, one runs into technical issues when relating such dynamics back to those of the original wave train. To avert this, we introduce a nonlinear change of coordinates: 
\begin{align}
\bar{\theta} = \theta - \phi(\theta, t),
\end{align}

and consider $\theta$ as our new variable. Such a transformation is invertible due to $\partial_\theta \phi(\theta, t)$ being small uniformly in $\theta$ for all $t$, which we shall prove. In $(t,\theta)$ coordinates, we seek an ansatz of the form: 
\begin{align} \label{A.1}
u(t, \bar{\theta}) = u_0(\theta; k(1 + \partial_\theta \phi(t, \theta))) + w(t, \theta). 
\end{align}

Denote by: 
\begin{align}
u^\phi_0 := u_0(\theta; k(1 + \partial_\theta \phi)).
\end{align}

Upon inserting the ansatz (\ref{A.1}) into (\ref{eqn.bar}), we obtain: 
\begin{align} \nonumber
-\frac{\partial_t \phi}{1-\partial_\theta \phi} \partial_\theta u_0^\theta &- k \Big(\partial_\theta^2 \phi \frac{\partial_t \phi}{1-\partial_\theta \phi} - \partial_\theta \partial_t \phi \Big)\partial_k u^\phi_0 + \partial_t w  - \frac{\partial_t \phi}{1 - \partial_\theta \phi} \partial_t w \\ \nonumber
& = k^2D\Big[\Big(\frac{1}{1-\partial_\theta \phi} \partial_\theta + \frac{k \partial_\theta^2 \phi}{1-\partial_\theta \phi} \partial_k \Big)^2u^\phi_0 + \Big(\frac{1}{1-\partial_\theta \phi} \partial_\theta \Big)^2 w \Big] \\ \nonumber
& + \omega \frac{1}{1-\partial_\theta \phi} \Big(\partial_\theta u^\phi_0 + k \partial_\theta^2 \phi \partial_k u^\phi_0 + \partial_\theta w \Big) \\ \label{sys.res}
& - \Big(k^2 D \partial_\theta^2 u_0- \omega \partial_\theta u_0 + f(u_0) \Big) + f(u^\phi_0 + w). 
\end{align}

In order to analyze the resulting system, we need to use the Fourier-Bloch transformation, whose basics we now state. The ultimate goal is to split the relatively complicated system (\ref{sys.res}) into two equations, one for the ``critical" component, whose Fourier support approaches the origin, and one for the ``stable" component, whose Fourier support is bounded away from the origin, and is therefore expected to decay exponentially. Recall the definition given in (\ref{Bloch.1}), from which the following identity follows immediately: 
\begin{align} \label{Bloch.deriv}
\mathcal{J}(\p_\theta w)(\nu, \xi) =  \Big(\p_\nu + i \xi \Big) \tilde{w}(\nu, \xi). 
\end{align}

Let us define the four-parameter Bloch-wave space, $B_{L}(s,k,p)$, via: 
\begin{equation} \label{norm.B}
||\tilde{f}||_{B_{L}(s,k,p)}^2 := \sum_{\alpha \le k, \beta \le p} || \partial_\xi^\alpha \partial_\nu^\beta \tilde{f}\xi^s ||^2_{L^2((-\frac{Lk}{2}, \frac{Lk}{2}), L^2(\mathbb{T}) )}. 
\end{equation}

Clearly, the index, $s$, in the Bloch-norm represents spatial regularity. Note that in this notation, the assumption (\ref{as.small}) translates into $||\tilde{v_0}||_{B_1(1,0,1)} \le \delta$. In what follows for our energy estimates, we switch between spatial and Fourier/ Bloch space via Parseval's identity. First, we recall that $\lambda_1(\xi) \subset \mathbb{C}$ is the analytic curve defining the spectrum of the linearization around our given wavetrain which approaches the imaginary axis, $Re(z) = 0$ (see \cite{Sandstede}, pages 3544 - 3545, for the details of these spectral considerations). We then define the following projection: 
\begin{equation}
Q^c(\xi) := \frac{1}{2\pi i} \int_{\Gamma} \Big( \lambda - \tilde{\mathcal{L}}(\xi) \Big)^{-1} d\lambda,
\end{equation}

where $\Gamma \subset \mathbb{C}$ is a closed curve surrounding $\lambda_1(\xi)$, and $\tilde{\mathcal{L}}$ has been defined in (\ref{WT.2}). Note that it is possible to find a contour $\Gamma$ due to Hypothesis \ref{Hyp.1}.  Next, fix a normalized cutoff function:
\begin{align}
\chi(\xi) = 
\begin{cases} 
       1 \text{ for } |\xi| \le 1, \\
        0 \text{ for } |\xi| \ge 2.
   \end{cases}
\end{align}

Define the following Fourier mode-filters:
\begin{equation} \label{mf}
p_{mf}^c(\xi) =  Q^c(\xi) \chi\Big( \frac{8\xi}{l_1} \Big), \hspace{3 mm} p^s_{mf}(\xi) = 1 - p^c_{mf}(\xi), 
\end{equation}

where $l_1$ is an arbitrary parameter. We will introduce our unknowns of interest: 
\begin{definition} \label{defn.unknowns} The unknowns of interest are defined via:
\begin{align} \label{UNK.n}
\hat{u}^c = \mathcal{J}(\partial_\theta \phi), \hspace{3 mm} \tilde{u}^s = \mathcal{J}(w - S_1 \partial_\theta \phi), 
\end{align}
where $S_1$ is a Fourier multiplier with frequency support: 
\begin{align} \label{S1.prop.1}
\supp S_1 \subset \{|\xi| \in \frac{l_1}{8}, \frac{l_1}{4} \}.
\end{align}
The precise formulation of $S_1$ is given in \cite{Sandstede}, pp. 3558-3559, but the only relevant property for our purpose is (\ref{S1.prop.1}).
\end{definition}

By the arguments given in \cite{Sandstede}, in the case of the critical wavenumber unknown, $u^c$, the Bloch-wave transform coincides with the Fourier-transform, so we simply work with $\hat{u}^c$, whereas for the stable unknown, we must work with the full Bloch-transform, denoted by $\tilde{u}^s$. We will also define the following dispersion relation: 
\begin{equation} \label{disp.}
\lambda(\xi) = \chi\Big(\frac{\xi l_1}{4}\Big)\Big( |\xi|^2 +  o(\xi^3) \Big). 
\end{equation}

With this choice of unknown, we have: 
\begin{lemma} The ansatz $u$ given in (\ref{A.1}) satisfies equation (\ref{eqn.rxn.diff.}) if and only if $\hat{u}^c, \tilde{u}^s$ defined through (\ref{UNK.n}) satisfy the following equations: 
\begin{align} \label{main.eqn.1}
&\Big( \partial_t - \lambda(\xi) \Big) \hat{u}^c = \hat{N}^c(\hat{u}^c, \tilde{u}^s) := p^c_{mf}(\xi) i \beta \xi \Big(\hat{u}^c \ast \hat{u}^c\Big) + \hat{N}(\hat{u}^c, \tilde{u}^s), \\ \label{prove.1}
& \hspace{10 mm} \int_{\mathbb{R}} u^c(1,\cdot) = \phi_d, \\ \label{prove.2}
& \hspace{10 mm} ||\hat{u}^c(1,\xi)||_{L^1_\xi(1)} + ||\tilde{u}^s(1,\nu, \xi)||_{B_1(1,0,1)} \lesssim \delta, \\ \label{prove.3}
& \hspace{10 mm} ||u^c(1,\cdot) - \frac{1}{\sqrt{T_0}} f_A^\ast\Big( \frac{\cdot}{\sqrt{T_0}} \Big)||_{H^2(2)} \lesssim \rho_\ast. 
\end{align}

The equation for $\tilde{u}^s$ is given in Bloch representation:
\begin{align} \label{main.eqn.2}
\partial_t \tilde{u}^{n,s} - \Lambda \tilde{u}^{s} &= \tilde{N}^{(s)}(\hat{u}^c, \tilde{u}^s).
\end{align}

The critical and stable nonlinearities, $N$, and $\tilde{N}^{s}$ are related in the following manner: 
\begin{align} \label{eqn.N.1}
&\hat{N}^c(\hat{u}^c, \tilde{u}^s) = (\partial_\nu + i\frac{\xi}{k}) p_{mf}^c(\xi) \tilde{\mathcal{N}}(\hat{u}^c, \tilde{u}^s), \\
&\tilde{N}^s(\hat{u}^c, \tilde{u}^s) = p_{mf}^s(\xi) \tilde{\mathcal{N}}(\hat{u}^c, \tilde{u}^s), 
\end{align}

where $\tilde{\mathcal{N}}$ is a smooth nonlinear map from $H^{m+2} \times H^{m+2} \rightarrow H^2$. More specifically, the nonlinearity, $\hat{N}$, excludes the Burgers terms from $\hat{N}^c$, and is therefore comprised of three components. These terms are ``irrelevant" to the asymptotics, and are given in detail here: 
\begin{equation}
\hat{N}(\hat{u}^c, \tilde{u}^s) = \hat{N}_1 + \hat{N}_2 + \hat{N}_3,
\end{equation}

where (in Fourier/ Bloch):
\begin{align} \label{nl.1}
\hat{N}_1(\hat{u}^c, \tilde{u}^s) &= p^c_{mf}(\xi) h_1(\xi) \Big( \hat{u^c} \ast \hat{u^c} \Big), \hspace{3 mm} h_1(\xi) = o(\xi^2) \\
\hat{N}_2(\hat{u}^c, \tilde{u}^s) &= p^c_{mf}(\xi) h_2(\xi) \Big( \hat{u^c} \ast (\tilde{u}^{s}_\nu + \tilde{u}^s_{\nu \nu} ) \Big), \hspace{3 mm} h_2(\xi) = o(\xi) \\ \label{nl.4}
\hat{N}_3(\hat{u}^c, \tilde{u}^s) &= p^c_{mf}(\xi) h_3(\xi) p(\hat{u}^c, \tilde{u}^s), \hspace{3 mm} h_3(\xi) = o(\xi). 
\end{align}

$p(\cdot, \cdot)$ above is a polynomial, containing only terms of cubic or higher degree in spatial coordinates, so in Fourier space involves three or more convolutions. The linear operator $\Lambda$ has the property:
\begin{align} \label{stable.spec}
\text{spectrum}(\Lambda) \subset \{ Re(z) < -\sigma_0 \}.
\end{align}

Moreover, $\Lambda$ is second order, and is a relatively bounded perturbation of $\p_{xx}$. Furthermore, the unknowns satisfy the following conditions on their Fourier supports: 
\begin{align}
&\supp ( \hat{u}^c) \subset \{ 0 \le |\xi| \le \frac{l_1}{4} \}, \\
& \supp( \tilde{u}^s) \subset \{ |\xi| \ge \frac{l_1}{8} \} \times \mathbb{T}.
\end{align}

\end{lemma}

\begin{proof}
This is derived in detail in \cite{Sandstede}, and so we omit repeating the derivation. However, we do prove the conditions in (\ref{prove.1}) - (\ref{prove.3}). As remarked in (\ref{initial.data}), our initial data consists of: 
\begin{align}
u(1,x) = u_0(kx + \phi_0(x); k) + v_0(x) = u_0(\bar{\theta} + \phi_0(\bar{\theta}); k) + v_0(x) = u_0(\theta; k) + v_0(x).  
\end{align}

In order to match (\ref{A.1}), we will take: 
\begin{align}
w(1,\theta) = v_0(\theta) + u_0(\theta, k) - u_0(\theta, k(1+\partial_\theta \phi_0(\theta))). 
\end{align}

From here, it is clear that $\phi(1,\theta) = \phi_0(\theta) \xrightarrow{\theta \rightarrow \pm \infty} \phi_{\pm}$, and that for $j = 0, 1$: 
\begin{align}
||\mathcal{F}\Big(\partial_\theta \phi_0\Big)||_{L^1_\xi} + ||\Big(\partial_\nu + i \xi \Big)^j\tilde{w}(1,\cdot, \cdot)||_{L^2([-\frac{k}{2}, \frac{k}{2}] \times \mathbb{T})} \lesssim \delta.
\end{align}

The estimate on $\tilde{w}(1,\cdot, \cdot)$ follows upon Taylor expanding $u_0(\theta, \cdot)$, and applying the estimates in (\ref{as.small}). One now has $\psi_0 := \partial_\theta \phi_0$, and $w_0 = w(1,\cdot)$. Then:  
\begin{align}
(\hat{u}^c, \tilde{u}^s)|_{t = 1} = \mathcal{J} \mathcal{S}^{-1}(\psi_0, w_0),
\end{align}

and since $\mathcal{S}^{-1}$ is a Fourier multiplier with compact support, we have: 
\begin{align}
||\hat{u}^c|_{t = 1}||_{L^1_\xi(1)} + ||\tilde{u}^s|_{t = 1}||_{B_1(1,0,1)} \lesssim \delta. 
\end{align}

We have: $\tilde{u}^c = \mathcal{J}(\psi_0)$, from which it follows that:
\begin{align}
\hat{u}^c(\xi =0) = \hat{\psi_0}(\xi = 0) = \int \psi_0 d\theta = \int \partial_\theta \phi_0 d\theta = \phi_d.
\end{align} 

Finally, by definition of $u^c$, we have: 
\begin{align} \nonumber
|| u^c(1,\cdot) - \frac{1}{\sqrt{T_0}} f_A^\ast\Big( \frac{\cdot}{\sqrt{T_0}} \Big)||_{H^2(2)} = || \hat{\psi_0}(1,\cdot) - \mathcal{F} \{ \frac{1}{\sqrt{T_0}} f_A^\ast\Big( \frac{\cdot}{\sqrt{T_0}} \Big) \}||_{H^2(2)} \le \rho_\ast. 
\end{align}

\end{proof}

\begin{remark} The dispersion relation, $\lambda(\xi)$, is cut off, but done so in such a way that the Laplacian is preserved on the cut-off critical modes, as we write in (\ref{disp.}). 
\end{remark}

\begin{remark} Due to the fact that the Bloch transform agrees with the Fourier transform for $\hat{u}^c$, we may omit the factor of $\p_\nu$ in (\ref{eqn.N.1}) and simply retain the derivative $i\xi$.
\end{remark}

A corollary to the above lemma is:  
\begin{corollary} Suppose $\hat{u}^c$ satisfies the system (\ref{main.eqn.1}) - (\ref{prove.1}), with $\hat{N}_i$ for $i = 1,2,3$ defined in (\ref{nl.1}) - (\ref{nl.4}). Then $\hat{u}^c$ exhibits mean-preserving flow, that is: 
\begin{equation} \label{mean.cons}
\int u^c(t, \theta) d\theta = \phi_d, \text{ for all } t \ge 1. 
\end{equation}
\end{corollary}
\begin{proof}
This follows from taking an integration of (\ref{main.eqn.1}), and using that each term on the right-hand side contains at least one derivative, that is one factor of $\xi$, according to (\ref{nl.1}) - (\ref{nl.4}). 
\end{proof}

Next, we will state the folllowing: 
\begin{lemma} \label{L.suff} Let $u_0(\cdot; k)$ be a spectrally stable wave-train with the dispersion relation $\omega = \omega(k)$. Consider initial data which are perturbations of the form in (\ref{initial.data}).  Suppose the corresponding unknowns $[\hat{u}^c, \tilde{u}^s]$, as defined in Definition (\ref{defn.unknowns}), are shown to exist globally, satisfying the following asymptotic behavior: 
\begin{align} \label{ASY.1}
||\sqrt{t} u^c(t, \sqrt{t}z) - f^\ast_A(z)||_{H^2(2)} + ||\sqrt{t} \tilde{u}^s(t, \sqrt{t}z)||_{B_{\sqrt{t}}(2,2,2)} \le C(\phi_d) t^{-\frac{1}{2}+\sigma},
\end{align}
for some $\sigma > 0$, where the self-similar variable $z = \frac{\theta}{\sqrt{t}}$. Then the solution to (\ref{eqn.rxn.diff.}) exists globally, and can be written as the modulated wave-train: 
\begin{align}
u(t,x) = u_0(kx-\omega t + \phi_A^\ast(x-c_g(t-1));k) + v(t,x), 
\end{align}

where:
\begin{align} \label{pA1}
\phi_A^\ast(z) = \frac{\alpha}{\beta} \log \{1 + A e_\ast(\frac{x}{\sqrt{\alpha t}}) \}, \hspace{3 mm} \log\{1 + A \} = \phi_d, \hspace{3 mm} c_g = \omega'(k). 
\end{align}

and 
\begin{equation}
||v(t,\cdot)||_{L^\infty} \le C(\phi_d, \sigma) t^{-\frac{1}{2}+\sigma}, \text{ for some }\sigma > 0. 
\end{equation}
\end{lemma}
\begin{proof}
This exact claim is proven in \cite{Sandstede}, pp. 3562 - 3564. 
\end{proof}

Lemma \ref{L.suff} then asserts that the asymptotic convergence given in (\ref{ASY.1}) implies our main theorem, Theorem \ref{thm.main}. We may now turn to analyzing the system (\ref{main.eqn.1})  - (\ref{nl.4}), with the aim of proving (\ref{ASY.1}), which constitutes the entire focus of our analysis. 

We give now an existence result for the system of equations (\ref{main.eqn.1}) - (\ref{main.eqn.2}). This is a very crude bound; it is simply what is required in order to bootstrap towards decay and asymptotics, which is performed in the next sections. 
\begin{lemma} \label{thm.LE} Suppose the initial data satisfy the criteria (\ref{prove.2}). Then if $\delta$ is sufficiently small relative to $k$ and universal constants, the solution $\hat{u}^c, \tilde{u}^s$ to the system (\ref{main.eqn.1}) - (\ref{nl.4}) exists for all $t > 1$, and satisfies the uniform bound:
\begin{equation}
\sup_{t \ge 1} ||\hat{u}^c||_{L^1_\xi(1)} \le 1000 \delta. 
\end{equation}
\end{lemma}
\begin{proof}

The proof follows from a trivial energy estimate together with continuous induction. Let $T_\ast > 0$ be the maximal time at which $||\hat{u}^c||_{L^1_\xi} \le 1000 \delta$. We shall multiply the critical equation, (\ref{main.eqn.1}) by $\overline{\hat{u}^c}$ and take integrations over the domain $[-\frac{k}{2}, \frac{k}{2}] \times \mathbb{T}$:
\begin{align}
\frac{\partial_t}{2} \int \Big| \hat{u}^c \Big|^2 + \int \Big|o(\xi) \hat{u}^c \Big|^2 \lesssim 1000 C(k)  \delta \Big[ \int \Big| \hat{u}^c \Big|^2 + \int \Big| \tilde{u}^2_{\nu}, \tilde{u}^s_{\nu \nu} \Big|^2 \Big].
\end{align}

As the Fourier modes are restricted to $[-\frac{k}{2}, \frac{k}{2}]$, one can repeat the same energy identity to obtain: 
\begin{align}
\frac{\partial_t}{2} \int \Big| \langle \xi \rangle^2 \hat{u}^c \Big|^2 + \int \Big|\langle \xi \rangle^3 \hat{u}^c \Big|^2 \lesssim 1000 C(k) \delta \Big[ \int \Big| \hat{u}^c \Big|^2 + \int \Big| \tilde{u}^2_{\nu}, \tilde{u}^s_{\nu \nu} \Big|^2 \Big].
\end{align}

For the stable equation, (\ref{main.eqn.2}), one can use that $\Lambda$ is a relatively bounded perturbation of $\partial_{\theta \theta}$ in Bloch wave space to obtain the energy identities: 
\begin{align}
\partial_t \int \Big| \Big( \partial_\nu + i \xi \Big)^j \tilde{u}^s \Big|^2 + \int \Big| \Big( \partial_\nu + i \xi \Big)^j \tilde{u}^s \Big|^2 + \int \Big| \Big( \partial_\nu + i \xi \Big)^{1+j} \tilde{u}^s \Big|^2 \le 0,
\end{align}

so long as $1 \le t \le T_\ast$. A standard Gronwall argument together with the embedding $L^2_\xi(2) \hookrightarrow L^1_\xi(1)$ shows that no such $T_\ast$ can exist upon taking $\delta$ small relative to universal constants and $k$.

\end{proof}

\section{Preparation for Asymptotics} \label{prep.section}

Our aim for the remainder of this paper is to prove the following: 
\begin{theorem} \label{thm.main.2} Suppose $\beta = -\omega''(k) \neq 0$. Given any $\phi_d, \rho_\ast < \infty$, there exists a $\delta > 0$ which depends on $\phi_d$, such that whenever:  
\begin{equation}
||\hat{u}^c(1,\cdot)||_{L^1_\xi(1)} + ||\tilde{u}^s(1,\cdot, \cdot)||_{B_1(1,0,1)} \le C\delta, 
\end{equation}

the solution to system (\ref{main.eqn.1}) - (\ref{main.eqn.2}) exists globally and satisfies the following asymptotics:
\begin{align} \label{main.RG.est}
|| \sqrt{t} u^c(t, \sqrt{t}z) - f_A^\ast(z)||_{H^2(2)} + ||\sqrt{t}\tilde{u}^s(t, \sqrt{t}z)||_{B_{\sqrt{t}}(2,2,2)} \le C(\phi_d) t^{-\frac{1}{2}+\sigma},
\end{align}

for some $\sigma > 0$. 
\end{theorem}

As discussed after Lemma \ref{L.suff}, this implies our main result. It is our task in this section to introduce the various profiles, parameters, and unknowns which we will use to obtain the convergence in (\ref{main.RG.est}). The main result of this section is the splitting in (\ref{Split.1}), the corresponding convergence in (\ref{grow}), and the reduction of our goal to (\ref{goal}). 

\begin{remark}[Notational Convention] We will now rename the spatial variable $\theta \in \mathbb{R}$ to $x$ in order to ease the notation. This convention will be in effect for the entirety of the remainder of this paper. 
\end{remark}

\subsection{Explicit Profiles} \label{explicit}

The purpose of this subsection is to split the desired dynamics (estimate (\ref{main.RG.est}) above) into two pieces which are analyzed separately. This is the splitting seen below in (\ref{Split.1}). Recall the definition of $f^\ast_A$ given in (\ref{fA}). Define the self-similar Burgers profile:
\begin{equation} \label{ssB1}
\Big(\partial_t - \alpha\partial_{xx} \Big) b = \beta bb_x, \hspace{3 mm} b(1,x) = f_A^\ast(x), \hspace{3 mm}  \log(1+A) = \phi_d
\end{equation}

This profile is given explicitly by: 
\begin{align}
b(t,x) = \frac{1}{\sqrt{t}}f^\ast_A(\frac{x}{\sqrt{t}}) = \frac{1}{\sqrt{t}}f_A^\ast(z). 
\end{align}

The profile $f_A^\ast$ is not small in any sense, despite our assumptions on $u^c(1,x)$. We now record the size of $f_A^\ast$ more precisely:
\begin{lemma} For $\log(1+A) = \phi_d$, 
\begin{align} \label{fAlarge.1}
||f_A^\ast||_{L^1} = \phi_d,  \text{ and } ||f_A^\ast||_{L^\infty} \gtrsim 1. 
\end{align}
We also have positivity: 
\begin{align} \label{sign.1}
f^\ast_A \ge 0. 
\end{align}
\end{lemma}
\begin{proof}
The $L^1$ estimate in (\ref{fAlarge.1}) follows from explicitly integrating:
\begin{align}
||f^\ast_A||_{L^1} = \int_\mathbb{R} f^\ast_A  = \int_\mathbb{R} \frac{A e_\ast'(z)}{1 + Ae_\ast(z)} = \int_{\mathbb{R}} \partial_z \{ \frac{1}{1+Ae_\ast(z)} \} = \log(1+A) = \phi_d. 
\end{align}

For the $L^\infty$ estimate, we evaluate explicitly: 
\begin{align}
f^\ast_A(1) = \frac{\frac{A}{\sqrt{4\pi}}}{1 + Ae_\ast(0)} \gtrsim 1. 
\end{align}

The sign condition is verified again via direct calculation: 
\begin{align}
f^\ast_A = \frac{1}{\sqrt{4\pi}} \frac{e^{-\frac{z^2}{4}}}{1 + Ae_\ast(z)} \ge 0. 
\end{align}

\end{proof}

\begin{corollary} The profile $f^\ast_A$ cannot be small in $H^2(2)$: 
\begin{equation} \label{fAlarge}
||f^\ast_A||_{H^2(2)} \ge \phi_d
\end{equation}
\end{corollary}
\begin{proof}
This follows from the embedding $H^2(2) \hookrightarrow L^1$. 
\end{proof}

Taking Burgers to start with the initial data $f_A^\ast$ means that we know $b(t,x)$ explicitly as the corresponding self-similar solution: 
\begin{equation}
t^{\frac{1}{2}}b(t,x) = t^{\frac{1}{2}}b(t,t^{\frac{1}{2}}z) = f_A^\ast(z), \hspace{3 mm} z = \frac{x}{\sqrt{t}}. 
\end{equation}

Define also an auxiliary Burgers profile: 
\begin{equation}
\bar{b}_t - \bar{b}_{xx} = \bar{b}\bar{b}_x, \hspace{3 mm} \bar{b}(1,x) = b(T_0, x). 
\end{equation}

By uniqueness, this then means 
\begin{equation} \label{barb.1}
\bar{b}(t,x) = b(t-1+T_0,x) = \frac{1}{\sqrt{t+T_0}} f_A^\ast(\frac{x}{\sqrt{t+T_0}}).
\end{equation}

Our critical solution will be compared to $\bar{b}$ via: 
\begin{align} \nonumber
||t^{\frac{1}{2}}u^c(t, t^{\frac{1}{2}}z) &- f_A^\ast(z)||_{H^2(2)} \\ \label{Split.1}
& \le ||t^{\frac{1}{2}}u^c(t, t^{\frac{1}{2}}z) - t^{\frac{1}{2}}\bar{b}(t, t^{\frac{1}{2}}z)||_{H^2(2)}  + ||t^{\frac{1}{2}}\bar{b}(t, t^{\frac{1}{2}}z) - f_A^\ast(z)||_{H^2(2)} \\ \nonumber
& = (\ref{Split.1}.1) + (\ref{Split.1}.2).
\end{align}

\begin{lemma} The following asymptotic convergence holds true in $H^2(2)$:
\begin{equation} \label{grow}
||\sqrt{t}\bar{b}(t, \sqrt{t}z) - f_A^\ast(z)||_{H^2(2)} \le C(\phi_d, T_0) t^{-\frac{1}{2}}. 
\end{equation}
\end{lemma}
\begin{proof}

We can clearly restrict ourselves to studying $t \ge 10$, because for $t \le 10$ the desired estimate is clear. 
\begin{align} \nonumber
\Big|\sqrt{t}\bar{b}(t, \sqrt{t}z) - f_A^\ast(z)\Big| &\le \Big|\frac{\sqrt{t}}{\sqrt{t+T_0}}f_A^\ast(\frac{x}{\sqrt{t+T_0}}) - f_A^\ast(\frac{x}{\sqrt{t+T_0}})\Big| + \Big|f_A^\ast(\frac{x}{\sqrt{t+T_0}}) - f_A^\ast(\frac{x}{\sqrt{t}})\Big| \\
&\le \Big|1 - \frac{\sqrt{t}}{\sqrt{t+T_0}} \Big| \Big|f_A^\ast(\frac{x}{\sqrt{t+T_0}})\Big| + \Big|f_A^\ast(\frac{x}{\sqrt{t+T_0}}) - f_A^\ast(\frac{x}{\sqrt{t}})\Big|.
\end{align}

For the first inequality, we directly estimate: 
\begin{align}
\Big|\frac{\sqrt{t+T_0} - \sqrt{t}}{\sqrt{t+T_0}} \Big| \le \frac{C(T_0)}{\sqrt{t}}. 
\end{align}

For the second inequality, we first save notation by labeling the denominators:
\begin{equation}
D_1 := 1+e_\ast(\frac{x}{\sqrt{t+T_0}}), \hspace{3 mm} D_2 := 1+e_\ast(\frac{x}{\sqrt{t}}).
\end{equation}

Then, 
\begin{align}
\Big|f_A^\ast(\frac{x}{\sqrt{t+T_0}}) - f_A^\ast(\frac{x}{\sqrt{t}})\Big| &\le \Big| \frac{e^{-\frac{x^2}{t+T_0}}}{D_1} - \frac{e^{-\frac{x^2}{t}}}{D_1} \Big| + \Big|\frac{e^{-\frac{x^2}{t}}}{D_1} -  \frac{e^{-\frac{x^2}{t}}}{D_2}\Big|.
\end{align}

For the first term, let $N$ be a large number to be selected. Then:
\begin{align} \nonumber
\Big| \frac{e^{-\frac{x^2}{t+T_0}}}{D_1} &- \frac{e^{-\frac{x^2}{t}}}{D_1} \Big| \le e^{-z^2} \Big| e^{\frac{x^2}{t} - \frac{x^2}{t+T_0}} - 1\Big| = e^{-z^2}\Big[\sum_{n \ge 1} \frac{1}{n!} \Big( z^2 \frac{T_0}{t+T_0} \Big)^n  \Big] \\ \nonumber
& = e^{-z^2} \Big[ -\frac{1}{\sqrt{t}} + \frac{1}{\sqrt{t}} \sum_{n \ge 0} \frac{1}{n!} \Big( z^2 \frac{T_0}{t+T_0} t^{\frac{1}{2n}} \Big)^n  \Big] \\ \nonumber
& = e^{-z^2} \Big[ -\frac{1}{\sqrt{t}} + \frac{1}{\sqrt{t}} \sum_{n=0}^{N} \frac{1}{n!} \Big( z^2 \frac{T_0}{t+T_0} t^{\frac{1}{2n}} \Big)^n + \frac{1}{\sqrt{t}}\sum_{n = N+1}^\infty \frac{1}{n!}\Big( z^2 \frac{T_0}{t+T_0} t^{\frac{1}{2n}} \Big)^n  \Big].
\end{align}

The first two terms above are easily controlled by $\frac{N}{\sqrt{t}} z^{2N}e^{-z^2}$, and so we must control the tail of the sum. For this, we note that: 
\begin{align}
\varphi(t) := \frac{T_0}{t+T_0} t^{\frac{1}{2n}}, \hspace{3 mm} n \ge 1,
\end{align}

has derivative: 
\begin{equation}
\varphi'(t) = \frac{1}{2n} t^{\frac{1}{2n}-1} \frac{T_0}{t+T_0} - t^{\frac{1}{2n}} \frac{T_0}{(t+T_0)^2},
\end{equation}

and so obtains a maximum at:
\begin{equation}
t_\ast = \Big( \frac{\frac{1}{2n}}{1 - \frac{1}{2n}} \Big)T_0, \hspace{3 mm} \varphi(t_\ast) = \frac{T_0}{T_0 + \frac{T_0}{2n}(1-\frac{1}{2n}^{-1})} \Big( \frac{T_0}{2n}(1-\frac{1}{2n})^{-1} \Big)^{\frac{1}{2n}}. 
\end{equation}

Taking $N$ large relative to $T_0$, it is clear that $\sup_{t \ge 10} \phi(t) < \rho < 1$ independent of $n$ because $t_\ast < 1$ for $N$ large. The maximum then obtains at the endpoint, $t = 10$:
\begin{equation}
\varphi(10) = \frac{T_0}{10 + T_0} 10^{\frac{1}{2n}} \le \rho < 1, \hspace{3 mm} \text{$n$ selected sufficiently large relative to $T_0$.}
\end{equation}
 
Thus, this sum can be controlled by: 
\begin{align}
\frac{e^{-z^2}}{\sqrt{t}} \Big[ \sum_{n \ge N+1} \frac{1}{n!} \Big( z^2 \frac{T_0}{t+T_0} t^{\frac{1}{2n}} \Big)^{n} \Big] \le \frac{e^{-z^2}}{\sqrt{t}} \Big[ e^{\rho z^2} - \sum_{n =0}^{N+1} \frac{1}{n!} \Big( z^2 \rho \Big)^n \Big].
\end{align}

Summarizing the first term, we have: 
\begin{align}
\Big| \frac{e^{-\frac{x^2}{t+T_0}}}{D_1} &- \frac{e^{-\frac{x^2}{t}}}{D_1} \Big| \le C(T_0)z^N \frac{e^{-(1-\rho)z^2}}{\sqrt{t}}. 
\end{align}

For the second term, we combine the fraction:
\begin{align}
\Big|\frac{e^{-\frac{x^2}{t}}}{D_1} -  \frac{e^{-\frac{x^2}{t}}}{D_2}\Big| = \frac{1}{D_1 D_2} \Big| \Big( D_2 - D_1 \Big) e^{-\frac{x^2}{t}}\Big|. 
\end{align}

To estimate the difference between $D_i$, we use the integral definition: 
\begin{align}
\Big| D_1 - D_2 \Big| \le \int_{\frac{x}{\sqrt{t+T_0}}}^{\frac{x}{\sqrt{t}}} e^{-z^2} dz \le \Big| \frac{x}{\sqrt{t}} - \frac{x}{\sqrt{t+T_0}}\Big| \le \frac{x}{\sqrt{t}}\Big| 1 - \frac{\sqrt{t}}{\sqrt{t+T_0}} \Big| \le z\frac{T_0}{\sqrt{t}}.
\end{align}

Therefore, combining all inequalities: 
\begin{equation}
\Big|f_A^\ast(\frac{x}{\sqrt{t+T_0}}) - f_A^\ast(\frac{x}{\sqrt{t}})\Big| \le \Big| \le  \frac{C(T_0)}{\sqrt{t}}z^Ne^{-z(1-\rho)^2},
\end{equation}

where $N$ is a constant depending on $T_0$. This then proves the desired claim. 

\end{proof}

According to (\ref{Split.1}), it remains to analyze (\ref{Split.1}.1). Let us discuss the profile $\bar{b}(t,x)$, for which the first task is to select $T_0$. 

\begin{equation} \label{T0.sel}
T_0 = T_0(\phi_d, \delta) = \Big( \frac{\phi_d}{\delta} \Big)^2.
\end{equation}

\begin{lemma} Suppose $\delta$ is selected sufficiently small relative to $\phi_d$. Then, the profile $\bar{b}$ obeys the following estimates: 
\begin{align} \label{barb.small.ck}
 &||\bar{b}(1,x)||_{\mathcal{C}^k} + ||\hat{\bar{b}}(1,\cdot)||_{L^1_\xi(k)} \lesssim \delta \text{ for } k \ge 0, \\ \label{barbint}
 &\int \bar{b}(t, z) dz = \phi_d. 
\end{align}
\end{lemma}
\begin{proof}
This follows from the definition (\ref{barb.1}):
\begin{align}
|\bar{b}(1,\cdot)| \le |\frac{1}{\sqrt{T_0}}| |f^\ast_A(\frac{x}{\sqrt{T_0}})| \le \frac{1}{\sqrt{T_0}} ||f^\ast_A||_{L^\infty} \le \delta \frac{\delta}{\phi_d^2} ||f^\ast_A||_{L^\infty} \lesssim \delta. 
\end{align}

Differentiating (\ref{barb.1}) enables us to extend to higher order spatial derivatives. The $L^1_\xi$ estimates in (\ref{barb.small.ck}) work similarly. By directly performing the Fourier transform, we have:
\begin{align}
\hat{\bar{b}}(1,\xi) = \hat{f^\ast_A}(\sqrt{T_0} \xi),
\end{align}

and so:
\begin{equation}
||\hat{\bar{b}}(1,\xi)||_{L^1_\xi(k)} = ||\hat{f^\ast_A}(\sqrt{T_0} \xi)||_{L^1_\xi(k)} = \frac{1}{(\sqrt{T_0})^{k+1}} ||f^\ast_A||_{L^1_\xi} \lesssim \delta. 
\end{equation}

\end{proof}

Fix any parameter $L > 0$. We will need to record the behavior of $\bar{b}$ under the following rescaling map: 
\begin{align} \label{explicit.bn}
\bar{b}^{(n)}(t,z) := L^n \bar{b} (L^{2n}t, L^n z) = \frac{L^n}{\sqrt{L^{2n}t + T_0}} f_A^\ast \Big( \frac{L^nz }{\sqrt{L^{2n}t + T_0}} \Big), \hspace{5 mm} t \in [1,L^2].
\end{align}

We now quantify the size of the $\bar{b}^n$:
\begin{lemma} The sequence of iterates, (\ref{explicit.bn}), satisfies the following bounds:  
\begin{align} \label{criteria.2}
\sup_{t \in [1, L^2]} ||z^{k+\frac{1}{2}} \partial_z^k \bar{b}^{(n)}||_{L^2} + ||z^{k+1} \partial_z^k \bar{b}^{(n)}||_{L^\infty} \le C(\phi_d). 
\end{align}
The majorizing quantity above is independent of $n, L$ and small $\delta$.
\end{lemma}
\begin{proof}
This follows by direct computation using the explicit form in (\ref{explicit.bn}).
\end{proof}

\begin{remark} The quantity $\phi_d$ now enters our analysis through $\bar{b}^{(n)}$ in the above lemma. 
\end{remark}

\begin{remark}[Transience of Smallness] \label{remk.tr} Let us turn to the sequence of iterates $\bar{b}^{(n)}(1,z)$ defined by (\ref{explicit.bn}). For $n = 0$, $\bar{b}^{(n)}(1,z) = \bar{b}(1,z)$, which satisfies estimate (\ref{barb.small.ck}). We may now ask, does each iterate in $n$ satisfy the same smallness in $\mathcal{C}^1$ and $L^1_\xi$ based norms? The answer to this is no: one argues as follows, using the $H^2(2) \hookrightarrow L^\infty$ embedding:
\begin{align} \label{transient}
||\bar{b}^{(n)}(1,z) - f^\ast_A(z)||_{L^\infty} \lesssim ||\bar{b}^{(n)}(1,z) - f^\ast_A(z)||_{H^2(2)} \xrightarrow{n \rightarrow \infty} 0, 
\end{align}

according to the convergence in (\ref{grow}). Combined with the $L^\infty$ lower bound in (\ref{fAlarge.1}), one deduces for $n$ sufficiently large, we can no longer say that $||\bar{b}^{(n)}(1,z)||_{L^\infty}$ is small. 
\end{remark}

We will need to cutoff the Fourier-modes of $\bar{b}$, which upon referring to the definitions in (\ref{mf}), we do in the following way: 
\begin{align}
\hat{\bar{b}}^c = p_{mf}^c(\xi) \hat{\bar{b}}, \hspace{5 mm} \hat{\bar{b}}^s = p_{mf}^s(\xi) \hat{\bar{b}}, \hspace{5 mm} \bar{b}^c + \bar{b}^s = \bar{b}. 
\end{align}

We then have a splitting of the system:
\begin{align}
\Big(\partial_t - \lambda_1(\xi) \Big) \hat{\bar{b}}^c &=  p_{mf}^c(\xi) i \beta \xi \Big( \hat{\bar{b}} \ast \hat{\bar{b}} \Big)= p_{mf}^c i \beta \xi \Big( \hat{\bar{b}}^c \ast \hat{\bar{b}}^c \Big) + p_{mf}^c(\xi)i\beta \xi N_{b}^c(\bar{b}^c, \bar{b}^s), \\
\Big( \partial_t - \lambda_1(\xi) \Big) \hat{\bar{b}}^s &=  p_{mf}^s(\xi) i \beta \xi \Big( \bar{b} \ast \bar{b} \Big) = p_{mf}^s(\xi)i\beta \xi N_{b}^s(\bar{b}^c, \bar{b}^s).
\end{align}

Here, 
\begin{align} \label{defn.nbc}
N^c_b := \hat{\bar{b}}^s \ast \hat{\bar{b}}^c, \hspace{3 mm} N^s_b := \Big( \hat{\bar{b}}^s \ast \hat{\bar{b}}^s \Big) + \Big( \hat{\bar{b}}^s \ast \hat{\bar{b}}^c \Big). 
\end{align} 

The stable component, $\bar{b}^s$, decays to zero exponentially. Note that there is no quadratic self-interaction of $\bar{b}^s$ in the critical equation for $\bar{b}^c$. 

\subsection{The Renormalization Group}

Let us briefly discuss the idea of the Renormalization Group (RG), which was brought to the study of nonlinear evolution equations by \cite{BKL}. We refer the reader to \cite{BKL} and \cite{Sandstede} for more thorough treatments than the one we give here. The idea of the (RG) is to discretize the dynamics described by (\ref{Split.1}.1). The reader should keep in mind our goal, which is to obtain asymptotics of the type (\ref{Split.1}.1). Given this, let us define our unknown via: 
\begin{align} \label{defn.alpha}
a(t,x) = u^c(t,x) - \bar{b}^c(t,x). 
\end{align}

One fixes a time-scale, denoted by $L$. Then, by setting $t = L^n$, and as usual $z = \frac{x}{\sqrt{t}}$, the estimate in (\ref{Split.1}.1) reads: 
\begin{align} \nonumber
||t^{\frac{1}{2}} u^c(t, t^{\frac{1}{2}}z) - t^{\frac{1}{2}} \bar{b}(t, t^{\frac{1}{2}}z)||_{H^2(2)} &= ||t^{\frac{1}{2}} \alpha(t, t^{\frac{1}{2}}z)||_{H^2(2)} \\ \label{RGN}
& = ||L^n \alpha(L^{2n}, L^n z)||_{H^2(2)}. 
\end{align}

Motivated by (\ref{RGN}), one defines a new unknown: 
\begin{align}
a^{(n)}(t,x) := L^n a(L^{2n}t, L^n z). 
\end{align}

The quantities of interest then become the sequence of initial data, $a^{(n)}(1,z)$, of the unknowns $a^{(n)}$. By direct computation, one checks the iterative relation: 
\begin{align} \label{it.RG.R}
a^{(n)}(1, z) = L a^{(n-1)}(L^2, Lz). 
\end{align}

Motivated by (\ref{it.RG.R}), one defines the renormalization map: 
\begin{align} \label{defm.RGM.1}
R_Lf(z) := Lf(Lz),
\end{align}

so that $a^{(n)}(1,z) = R_L a^{(n-1)}(L^2, z)$. From here, the group operation is clear: it is the composition of the flow-forward map for a time scale of $L^2$ with the map $R_L$. $R_L$ has the Fourier representation: 
\begin{align} \label{defn.RGM}
\hat{R}_Lf(\xi) = \hat{f}(\frac{\xi}{L}).  
\end{align}

Due to the (\ref{defn.RGM}), we will set the following abuse of notation, which will allow for ease of several formulas: 
\begin{align}
\hat{R}_{L^{-1}} \hat{f}(\xi) = \hat{f}(\frac{\xi}{L}) = \mathcal{F} \Big[ R_L f \Big](\xi) =\mathcal{F} \Big[ L f(L \cdot) \Big](\xi).
\end{align}

From here, we have: 
\begin{lemma} Suppose $\int_\mathbb{R} f = 0$, equivalently, $\hat{f}(0) = 0$. Consider the heat semigroup $u = e^{\Delta (t-1)}f$. Then for any $0 < L < \infty$, there exists some constant $0 < C < \infty$, independent of $L$, such that:
\begin{align} \label{contr.1.1}
||R_L u(L^2, \cdot)||_{H^2(2)} \le \frac{C}{L} ||f||_{H^2(2)}.
\end{align}
\end{lemma}
\begin{proof}
This is estimate (4.5) on page 3561 in \cite{Sandstede}. 
\end{proof}

We will also record here basic facts which follow by direct calculation from (\ref{defn.RGM}), which will be recalled later: 
\begin{lemma}[Basic Properties of $R_L$] Fix any $0 < L < \infty$. Then:
\begin{align} \label{defn.b.b}
&||Rf||_{H^2(2)} \le L^\frac{5}{2} ||f||_{H^2(2)}, \hspace{20 mm} \text{(Boundedness)}\\ \label{commb1}
&\partial_x R_L = L R_L \partial_x  \hspace{38 mm} \text{(Commutativity)}.
\end{align}
\end{lemma}

Within the RG framework, the desired convergence (\ref{Split.1}.1) becomes equivalent to the convergence: 
\begin{align} \label{goal}
||a^{(n)}(1,z)||_{H^2(2)} \lesssim L^{-n(1-\sigma)}, \text{ for some } \sigma > 0. 
\end{align}

This will be the convergence that we prove in the forthcoming analysis. 

\subsection{Presentation of Unknowns}

Our goal for the rest of this paper is to extract the convergence in (\ref{Split.1}.1). First, let us record that the definition (\ref{defn.alpha}) preserves the Fourier support property: 
\begin{equation}
\text{supp } \hat{a} \subset (-\frac{l_1}{4}, 0].
\end{equation}

\begin{lemma}  $a$ satisfies the following properties:
\begin{align}  \label{as.w}
\int a(1,\cdot) =  0, \hspace{3 mm}  ||\hat{a}(1,\cdot)||_{L^1_\xi(1)} \lesssim \delta. 
\end{align}
\end{lemma}
\begin{proof}
First, according to (\ref{prove.1}) and (\ref{barb.small.ck}),
\begin{align}
\int a(1,\cdot)  = \int u^c(1,\cdot) - \int \bar{b}(1,\cdot) = \phi_d - \phi_d = 0.
\end{align}

The second estimate in (\ref{as.w}) follows by (\ref{prove.2}) and the first estimate in (\ref{barb.small.ck}):
\begin{equation}
||\hat{a}(1,\cdot)||_{L^1_\xi(1)} \le ||\hat{u}^c(1,\cdot)||_{L^1_\xi} + ||\bar{b}(1,\cdot)||_{L^1_\xi} \lesssim \delta.
\end{equation} 
\end{proof}

As suggested by (\ref{goal}) for the critical equation, for any $p > 0$, define the renormalized unknowns:
\begin{align} \label{RNO.1}
&\hat{a}^{(n)} = \hat{R}_{L^{-n}} \hat{a}(L^{2n}t, \xi), \\ \label{RNO.2}
&\tilde{u}^{(n,s)}(t,z) = L^{n(1-p)} \tilde{u}^s(L^{2n}t, \frac{\xi}{L^n}, \nu) = L^{n(1-p)} \hat{R}_{L^{-n}} \tilde{u}^s(L^{2n}t, \xi, \nu). 
\end{align}

The factor of $L^{-np}$ is a blow-up factor for the stable unknown, which helps in controlling the nonlinearities in (\ref{BUS.1}). Written in spatial scale, we have:
\begin{align} \label{scaling}
&a^{(n)}(t,z) = R_{L^n} a (L^{2n}t, \cdot) = L^n a(L^{2n}t, L^nz) = u^{n,c} - \bar{b}^{(n)}.
\end{align}

The unknowns in (\ref{RNO.1}) - (\ref{RNO.2}), the renormalized system on the time interval $[1, L^2]$: 
\begin{align} \nonumber
 &\partial_t\hat{a}^{(n)} -  \lambda_1^{(n)} \hat{a}^{(n)} =p^c_{mf}(\xi)i\beta \xi \Big( \Big( \hat{a}^{(n)} \ast \hat{a}^{(n)} \Big) + 2\Big( \hat{a}^{(n)} \ast \bar{b}^{c,n} \Big) \Big) + \hat{N}^{(n)}(u^{n,c}, u^{s, n}) \\ \label{eqn.c.alpha}
 & \hspace{30 mm} +  L^{-n} o(\xi^3) \hat{\bar{b}}^{(n)} + p_{mf}^c i \beta \xi N_b^c \\  \label{eqn.stable}
&\partial_t \tilde{u}^{n,s} - \Lambda^{(n)} \tilde{u}^{n,s} = L^{n(3-p)}\tilde{N}^{(n,s)},  \\ \label{eqn.stable.1}
&g^{(n)}(z) := a^{(n)}(1,z) = La^{(n-1)}(L^2, Lz), \hspace{3 mm} \tilde{u}^{n,s}(1,\xi, \nu) = \tilde{g}^{(n,s)}
\end{align}

We have denoted $\lambda^{(n)}, N^{(n)}$ above to be the rescaled versions, that is: 
\begin{align}
&\lambda^{(n)}(\xi) := L^{2n} \lambda\Big( \frac{\xi}{L^{2n}} \Big) =  \chi\Big(\frac{\xi l_1}{4L^n}\Big)\{ |\xi|^2 +  L^{-n} o(\xi^3) \}, \\
&N^{(n)} := \sum_{i = 1}^3 N^{(n)}_i, \\ \label{nl.resc}
&N^{(n)}_i := L^{3n} N_i \Big( \hat{R}_{L^{n}}\hat{u}^{n,c}(L^{-2n}t, \cdot), L^{-n(1-p)}\hat{R}_{L^{n}} \tilde{u}^{(n,s)}(L^{-2n}t, \cdot, \nu) \Big) (L^{2n}t, L^nz), \\
&\tilde{N}^{(n,s)} = \hat{R}_{L^{-n}}\tilde{p}^s_{mf} \tilde{N}^s(\hat{R}_{L^n} \hat{u}^{n,c}, L^{-n(1-p)} \hat{R}_{L^{n}} \tilde{u}^{n,s}).
\end{align}

The expression in (\ref{nl.resc}) appears to be fairly bulky; we write it more explicitly given the forms of (\ref{nl.1}) - (\ref{nl.4}):
\begin{align} \label{drn.1}
&\hat{N}^{(n)}_1(\hat{u}^{n,c}, \tilde{u}^{(n,s)})(t, \xi) = L^{-n}  o(\xi^2) \Big( \hat{u}^{n,c} \ast \hat{u}^{n,c} \Big), \\ \label{drn.2}
&\hat{N}^{(n)}_2(\hat{u}^{n,c}, \tilde{u}^{(n,s)})(t, \xi, \nu) = L^{-n(1-p)} o(\xi) \hat{u}^{n,c} \ast \Big( \tilde{u}^{n,s}_\nu + \tilde{u}^{n,s}_{\nu \nu} \Big),\\ \label{drn.3}
&\hat{N}^{(n)}_3(\hat{u}^{n,c}, \tilde{u}^{(n,s)})(t,\xi, \nu) = L^{-n}o(\xi)  p(\hat{u}^{n,c}, \tilde{u}^{(n,s)}).
\end{align}

Via the mean-zero feature of $a$ in (\ref{as.w}), the linear flow becomes a contraction on the space $H^2(2)$, according to (\ref{contr.1.1}). Motivated by this, the linear and nonlinear propagations are separated out: 
\begin{equation} \label{separ.1}
a^{(n)} = \bar{a}^{(n)} + \gamma^{(n)},
\end{equation}

where,
\begin{equation} \label{eqn.bar.a}
\partial_t \hat{\bar{a}}^{(n)} - \lambda^{(n)}(\xi) \hat{\bar{a}}^{(n)} = 0, \hspace{5 mm} \bar{a}^{(n)}(1,z) = a^{(n)}(1,z) = g^{(n)}(z). 
\end{equation}

The equation (\ref{eqn.bar.a}) implies:
\begin{equation} \label{almost.heat}
\text{supp } \hat{\gamma}^{(n)} \subset \{ \xi: |\xi| \le -L^{n}\frac{l_1}{4} \}.
\end{equation}

The linear part, $\lambda^{(n)}$, is $o(\xi^2)$ for $\xi \sim 0$, and so (\ref{almost.heat}) behaves like the heat-equation on the critical modes, and differs from the heat equation on exponentially decaying modes. We may thus inherit the following heat equation bounds for $\bar{a}$: 
\begin{align}
&\sup_{t \in [1,L^2]} ||\bar{a}^{(n)}||_{L^2} + \int_1^{L^2} ||\bar{a}^{(n)}_{z}||_{L^2} ds \le ||g^{(n)}||_{L^2}, \\
& ||\hat{\bar{a}}^{(n)}||_{L^1_\xi} \le ||\hat{g}^{(n)}||_{L^1_\xi}. 
\end{align}

Analogous bounds hold true for higher derivatives. We record now the following linearized system for $\gamma^{(n)}$, which follows from subtracting (\ref{eqn.bar.a}) from (\ref{eqn.c.alpha}):
\begin{align} \nonumber
\Big(\partial_t -  \lambda^{(n)}(\xi) \Big) \hat{\gamma}^{(n)} &= \hat{\mathcal{M}}^{(n)}(\hat{\bar{a}}^{(n)}, \hat{\gamma}^{(n)}, \hat{\bar{b}}^{c,n})+ \hat{N}^{(n)}(\hat{\bar{a}}^{(n)} + \hat{\gamma}^{(n)} + \hat{\bar{b}}^{(n)}, \tilde{u}^{s,n})  \\ \label{sys.1}
&  + \hat{\mathcal{B}}^{(n)}, \hspace{5 mm} \gamma^{n}(1,z) = 0. 
\end{align}

We have introduced the following notation: 
\begin{align} \label{DEF.M}
\hat{\mathcal{M}}^{(n)} &:= p^c_{mf}(\xi)i\beta \xi \Big((\hat{\bar{a}}^{(n)} + \hat{\gamma}^{(n)} ) \ast ( \hat{\bar{a}}^{(n)} + \hat{\gamma}^{(n)} ) + 2(\hat{\bar{a}}^{(n)} + \hat{\gamma}^{(n)} ) \ast \bar{b}^{c,n} \Big), \\ \label{DEF.B}
\hat{\mathcal{B}}^{(n)} &:= L^{-n} o(\xi^3) \hat{\bar{b}}^{(n)} + p_{mf}^c i \beta \xi N_b^c.
\end{align}

The terms $N^c_b$ are defined in (\ref{defn.nbc}). The estimate we demand from $\gamma^{(n)}$ is stated in Theorem \ref{thm.1} below. To understand the statement of this estimate, we must first introduce the various parameters that we will be using, and to this we now turn:

\subsection{Relevant Parameters}

Due to the complexity of the analysis which follows (several parameters are needed), we start by displaying a schematic of the various parameters that arise, and their interdependencies. The reader should refer back to this discussion when reading the lemmata which follow.

\vspace{3 mm}

\hspace{35 mm} \begin{tikzpicture} 
  \node (A) {$\phi_d$};
  \node (B) [right of=A] {$L$};
   \node (C) [below of=B] {$\epsilon$};
  \node (D) [right of=B] {$\delta_w$};
  \node (E) [below of = D] {$N_0$}; 
  \node (F) [right of = E] {$\delta, T_0$};
  \node (G) [below of = E] {$\rho_\ast$};
  \draw[->] (A) to node {} (B);
  \draw[->] (B) to node {} (C);
  \draw[->] (A) to node {} (C);
  \draw[->] (B) to node  {} (D);
   \draw[->] (C) to node [swap] { } (E);
   \draw[->] (E) to node {} (F); 
   \draw[->] (D) to node [swap] {} (F); 
   \draw[->] (G) to node {} (E);
\end{tikzpicture}
    
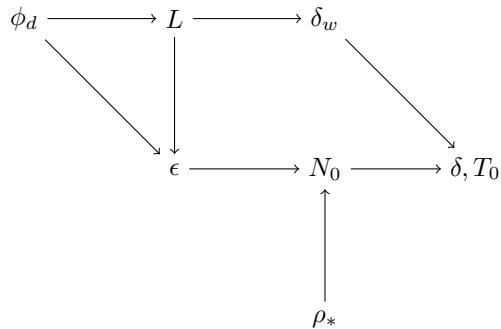
\captionof{figure}{Parameter Dependencies}
   \label{PDpic}
   
 \vspace{2 mm}

Let us now explain these dependences on a heuristic level. The precise selections of these parameters are made in Proposition \ref{Prop.sel}.

\vspace{3 mm}

(1) The parameter $\phi_d$ is given, and is arbitrarily large. It is the mean of the profile $f^\ast_A$ (see (\ref{fAlarge.1})), and the mean of the wavenumber unknown, $u^c$, (see (\ref{prove.1})). The parameter $\rho_\ast$ is also prescribed, and can be large.

\vspace{3 mm}

(2) The parameter $L$ determines the appropriate spatial scale in the Renormalization Group analysis that we will perform. It is selected only based on $\phi_d$, in such a way that: 
\begin{equation} \label{P2}
\frac{C(\phi_d)}{L} \ll 1,
\end{equation}

where $C(\phi_d)$ is some function of $\phi_d$ that can be made explicit, which is universal and independent of the other parameters in our analysis (like $C(\phi_d) = e^{\phi_d}$ or $C(\phi_d) = e^{e^{\phi_d}}$, for instance). 

\vspace{3 mm}

(3) The parameter $\epsilon$ is selected based on $\phi_d$ and $L$. This is done in such a way that: 
\begin{align} \label{P3}
\epsilon J(L) C(\phi_d) \ll 1,
\end{align}

for functions $C(\phi_d)$ and $J(L)$ which again are universal and can be computed explicitly.

\vspace{3 mm}

(4) The parameter $N_0$ is subsequently selected based on $\phi_d, \rho_\ast L, \epsilon$. The rule for this selection is essentially of the form: 
\begin{equation} \label{P4}
\frac{C(\phi_d) J(L)}{L^{N_0}} \le \epsilon,
\end{equation}

where $C(\phi_d), J(L)$ are again some functions which can be computed explicitly, based only on universal constants and functions. 

\vspace{3 mm}

(5) $\delta_w$ is selected as follows: given $\phi_d < \infty$, and an $L$ which is selected according to (\ref{P2}), we select $\delta_w$ according to:
\begin{align} \label{P5}
J(L) C(\phi_d) \delta_w \ll 1,
\end{align}

Again, $J(L), C(\phi_d)$ are some universal, explicitly determined functions.

\vspace{3 mm}

\begin{remark} $N_0$ and $\delta_w$ are selected independently. That is, the selection of $N_0$ does not affect the selection of $\delta_w$, and vice-versa. 
\end{remark}

(6) The parameter $\delta$ will be selected in estimate (\ref{delta.deltaw}), based on $\delta_w$ and $N_0$ which are determined by the previous steps. Specifically, the selection is made via: 
\begin{align} \label{P6}
\delta = L^{-2N_0} \delta_w. 
\end{align} 

Making a selection in this manner serves the following purpose: $\delta_w$ represents the smallness that is required in order for our energy estimates to close (Steps 1 through 4). $N_0$ represents the number of times that these energy estimates are iterated. By making the selection (\ref{P6}), since $\delta$ represents the size of the datum (in $L^1_\xi(1)$) at the zeroth iterate, we ensure that all $N_0$ iterates remain less than $\delta_w$. Finally, once this has been done, we can set: $T_0 = T_0(\phi_d, \delta) = \Big( \frac{\phi_d}{\delta} \Big)^2$.

\subsection{Assumptions on Iterative Initial Data}

For the energy estimates in Section \ref{RGA.Section} to be valid, we need to assume the following bounds on the uniform norm of the initial data and on the Burger's flow: 
\begin{equation} \label{assumption.1}
||\hat{g}^{(n)}||_{L^1_\xi(1)} \le \delta_w, \hspace{5 mm} \sup_{t \in [1, L^2]} ||\hat{\bar{b}}^{(n)}(t)||_{L^1_\xi(1)} \le \delta_w, \text{ for } n \le N_0.
\end{equation}

For now, we refrain from assigning a size to $\delta_w$ relative to other parameters. This is done in the statement of each lemma below. Once $N_0$ and $\delta$ have been fixed in Proposition \ref{Prop.sel}, we will apply Lemma \ref{thm.LE}, in order to conclude that $|| \hat{u}^{c}||_{L^1_\xi(1)} \lesssim \delta$ on the time interval $[1, L^{2N_0}]$. From this, we have: 
\begin{equation} \label{continue}
||\hat{u}^{n,c}||_{L^1_\xi(1)} \le L^{N_0} || \hat{u}^{c}||_{L^1_\xi(1)} \lesssim L^{N_0} \delta = \delta_w \text{ for } n \le N_0. 
\end{equation}

By using analogous estimates for the Burger's flow, we have that (\ref{barb.small.ck}) implies (\ref{assumption.1}). Let us now state: 
\begin{lemma} For each $n \ge 0$, the mean-zero feature is preserved:
\begin{equation}
\int g^{(n)}(z) dz = \int \alpha^{(n)}(1,z) dz = 0. 
\end{equation}
\end{lemma}
\begin{proof}
We have: 
\begin{align} \nonumber
\int \alpha^{(n)}(1,z) \ud z &\stackrel{(\ref{scaling})}{=} \int u^{n,c}(1,z) dz - \bar{b}^{(n)}(1,z) dz \\ \nonumber
& \stackrel{(\ref{explicit.bn})}{=} \int L^nu^c(L^{2n}, L^n z) \ud z- \int L^n \bar{b}(L^{2n}, L^n z) \ud z  \\ \nonumber
& = \int u^c(L^{2n}, y) \ud y - \int \bar{b}(L^{2n}, y) \ud y \\ \nonumber
& \stackrel{(\ref{mean.cons})}{=} \int u^c(1, y) \ud y - \int \bar{b}(1, y) \ud y \\ 
& = \phi_d - \phi_d = 0. 
\end{align}
\end{proof}

\begin{remark}[Notation] As $\delta < \delta_w$, we feel free in Steps 1 - 4 below to replace $\delta$ with $\delta_w$. Also, should a quantity appear on the right-hand side of an inequality which is clearly dominated by the same quantity on the left-hand side, we feel free to omit it. 
\end{remark}

The goal is to show the estimate in (\ref{goal}) for $\alpha^{(n)}$, which we will do by controlling $\gamma^{(n)}$ in the manner indicated below in (\ref{present.1}). We now turn to this. 

\section{Energy Estimates for $\gamma^{(n)}$, $0 \le n \le N_0$} \label{RGA.Section}

The goal of this section is to prove:

\begin{theorem} \label{thm.1} Fix any $\phi_d < \infty$. Let $0 < p < 1$ sufficiently small relative to universal constants. Then there exists a universal functions $J_1(L), C_1(\phi_d)$, such that if: 
\begin{align} \label{crit.0}
&\delta_w J_1(L) C_1(\phi_d) \ll 1, \text{ and } \\ \label{crit.-1}
&||\hat{g}^{(n)}, \tilde{g}^{(n,s)}||_{L^1_\xi(1)} +\sup_{t \in [1, L^2]} ||\hat{\bar{b}}^{(n)}(t)||_{L^1_\xi(1)} \le \delta_w, 
\end{align}
then, we have the following bounds for $\gamma^{(n)}$, for each $L \ge 1$:
\begin{equation} \label{present.1}
||\gamma^{(n)}(L^2, z)||_{H^2(2)} \lesssim J_E(L) \Big( \delta_w ||g^{(n)}||_{H^2(2)} + L^{-2n}C_E(\phi_d) + L^{-n(1-p)} \delta_w ||\tilde{g}^{(n,s)}||_{B_{L^n}(2,2,2)} \Big).
\end{equation}

Here $J_E(L)$ and $C_E(\phi_d)$ are universal functions of their arguments, which can be found explicitly. 
\end{theorem}

\begin{remark}[Notational Conventions] Our labeling convention for stating the upcoming results will be as follows: $C_i(\phi_d), J_i(L)$ for $i \in \mathbb{N}$ will denote the universal functions appearing as assumptions or criteria that must be met (in the present case, estimate (\ref{crit.0})). $C_{E,i}(\phi_d), J_{E,i}(\phi_d)$ will denote the ``output" universal function that is obtained as a result of the result being stated (in the present case, estimate (\ref{present.1})). An unspecified $C(\phi_d)$ or $J(L)$ simply stands for a universal function of $\phi_d, L$ respectively, which has come up during a calculation, which may change definitions during the course of that particular calculation, and is not significant enough at that moment to warrant a particular subscript. In other words, it is a temporary placeholder. 

Regarding integrations, when unspecified, we mean: 
\begin{align}
\int f = \int_{[-\frac{k}{2}, \frac{k}{2}] \times \mathbb{T}} f(\xi, \nu) \ud \xi \ud \nu.
\end{align}

Furthermore, we remark that all critical profiles are independent of $\hat{u}^c$, and so the $\mathbb{T}$ integration will typically simply evaluate to $2\pi$. Nevertheless, we retain the integration over $\mathbb{T}$ in order to remain technically correct when treating the critical-stable nonlinear interaction. 
\end{remark}

The proof proceeds in the following steps: 

\begin{itemize}
\item[(Step 1)] Estimate of $||\gamma^{(n)}||_{L^2}:$ This is achieved by applying the multiplier $\gamma^{(n)}$ to the system (\ref{sys.1}) and integrating by parts. 

\item[(Step 2)] Estimate of $||\gamma_z^{(n)}||_{L^2}:$ This is achieved by applying the multiplier $\gamma^{(n)}_z$ to the differentiated system $\partial_z$ (\ref{sys.1}) and integrating by parts. 

\item[(Step 3)] Estimate of $||\gamma_{zz}^{(n)}||_{L^2}:$ This is achieved by applying the multiplier $\gamma^{(n)}_{zz}$ to the twice differentiated system $\partial_z^2$ (\ref{sys.1}) and integrating by parts.

\item[(Step 4)] Estimate of $||\{\gamma_{z}^{(n)}, \gamma_{zz}^{(n)} \} z^2 ||_{L^2}:$ This is achieved by applying weighted multipliers, $z^m \partial_z^j \gamma^{(n)}$ for $m = 1,2$, $j = 0,1,2$ and repeating the previous steps.  
\end{itemize}

Upon proving Theorem \ref{thm.1}, we perform two additional steps which concludes the analysis for $0 \le n \le N_0$: 
\begin{itemize}
\item[(Step 5)] Stable Estimates for $\tilde{u}^{n,s}$: This is achieved directly by using the exponential decay of the linear operator, $\Lambda^{(n)}$ which follows from (\ref{stable.spec}).

\item[(Step 6)] Iteration for $0 \le n \le N_0$: The above steps are iterated up to $n = N_0$, at which point the smallness conditions (\ref{crit.-1}) are violated, and we must switch to a new set of estimates which are non-perturbative. 
\end{itemize}

\subsection*{Step 1: Estimate of $\gamma^{(n)}$:}

We now prove the steps detailed above. The reader should refer to the definition of the Bloch norm given in (\ref{norm.B}). We feel free to omit the parameters $\alpha, \beta$ from the system for the sake of presentation. First, multiply (\ref{sys.1}) by $\gamma^{(n)}$ on the spatial side, which we then immediately transfer to frequency side via Parseval. On the left-hand side, we have the positive quantities: 
\begin{equation} \label{m.pos.1}
\int \Big( \partial_t - \lambda^{(n)}(\xi) \Big) \hat{\gamma}^{(n)} \cdot \overline{\hat{\gamma}^{(n)}} \ge  \partial_t \int |\hat{\gamma}^{(n)}|^2 + \int \chi\Big(\frac{l_1}{4}\xi\Big) |\xi \hat{\gamma}^{(n)}|^2 \ge \partial_t \int |\hat{\gamma}^{(n)}|^2 + \int  |\xi \hat{\gamma}^{(n)}|^2
\end{equation}

where we have used that $\chi(\frac{l_1}{4}\xi) = 1$ on the support of $\hat{\gamma}^{(n)}$. On the right-hand side, we treat the terms from (\ref{DEF.M}), starting with:
\begin{align}\nonumber
&\Big|\int p^c_{mf}(\xi) i \beta \xi \Big( \hat{\bar{a}}^{(n)} \ast \hat{\bar{a}}^{(n)} \Big) \cdot \overline{\hat{\gamma}}^{(n)}\Big| = \Big|- \int p^{c}_{mf}(\xi) \beta \Big( \hat{\bar{a}}^{(n)} \ast \hat{\bar{a}}^{(n)}\Big) \overline{i\xi\hat{\gamma}^{(n)}} \Big| \\ \label{m.1.0}
&\hspace{20 mm} \le ||\hat{\bar{a}}^{(n)} \ast \hat{\bar{a}}^{(n)}||_{L^2}||p_{mf}^c i \xi \hat{\gamma}^{(n)}||_{L^2}  \\ \label{m.1}
&\hspace{20 mm} \le ||\hat{\bar{a}}^{(n)}||_{L^1_\xi} ||\hat{\bar{a}}^{(n)}||_{L^2}||p_{mf}^c i \xi \hat{\gamma}^{(n)}||_{L^2} \lesssim \delta_w^2 ||\hat{\bar{a}}^{(n)}||_{L^2}^2.
\end{align}

For (\ref{m.1.0}), we have used Young's inequality for convolution. In (\ref{m.1}), we have used the Young's Inequality for products to absorb the $||p^c_{mf} i \xi \hat{\gamma}^{(n)}||_{L^2}^2$ term into (\ref{m.pos.1}), and we have then used that $||\hat{\bar{a}}^{(n)}||_{L^1_\xi} \le ||\hat{g}^{(n)}||_{L^1_\xi}$, which is assumed to be of size $\delta_w$. Similarly:
\begin{align}\nonumber
& \Big| \int p^c_{mf}i \beta \xi \Big( \hat{\bar{a}}^{(n)} \ast \hat{\gamma}^{(n)} \Big) \cdot \overline{\hat{\gamma}^{(n)}} \Big| = \Big| \int p_{mf}^c \beta \Big( \hat{\bar{a}}^{(n)} \ast \hat{\gamma}^{(n)}\Big) \cdot \overline{i\xi \hat{\gamma}}^{(n)} \Big| \\
& \hspace{20 mm} \le ||\hat{\bar{a}}^{(n)}||_{L^1_\xi}||\hat{\gamma}^{(n)}||_{L^2}||p_{mf}^c i\xi \hat{\gamma}^{(n)}||_{L^2} \lesssim \delta_w^2 ||\hat{\gamma}^{(n)}||_{L^2}^2 \\ \nonumber
& \Big| \int p^c_{mf}i \beta \xi \Big( \hat{\gamma}^{(n)} \ast \hat{\gamma}^{(n)} \Big) \cdot \overline{\hat{\gamma}^{(n)}} \Big| \le ||\hat{\gamma}^{(n)}||_{L^2} ||p_{mf} i \beta \xi \hat{\gamma}^{(n)} \ast \hat{\gamma}^{(n)}||_{L^2} \\ 
& \hspace{20 mm} \le ||\hat{\gamma}^{(n)}||_{L^2} ||\hat{\gamma}^{(n)}||_{L^1_\xi} ||p_{mf}^c \xi \hat{\gamma}^{(n)}||_{L^2} \le \delta_w ||\hat{\gamma}^{(n)}||_{L^2}^2.
\end{align}

Next, we have the linearizations around $\bar{b}$, where we recall (\ref{crit.-1}):
\begin{align} \nonumber
\Big| \int p_{mf}^c i \beta \xi \Big( (\hat{\bar{a}}^{(n)} + \hat{\gamma}^{(n)}) \ast \hat{\bar{b}}^{c,n} \Big) \cdot \hat{\gamma}^{(n)} \Big| &\le ||\hat{\bar{b}}^{n,c}||_{L^1_\xi} ||\hat{\bar{a}}^{(n)} + \hat{\gamma}^{(n)}||_{L^2} ||\hat{\gamma}^{(n)}||_{L^2} \\
&\le \delta_w ||\hat{\bar{a}}^{(n)}||_{L^2}^2 + \delta_w ||\hat{\gamma}^{(n)}||_{L^2}^2. 
\end{align}

Finally, we have the terms in the irrelevant components of the nonlinearity. Referring to (\ref{drn.1}), the term $N_1^{(n)}$ is estimated by noting that $h_1(\xi) = o(|\xi|^2)$, and so contributes an additional $L^{-n}$: 
\begin{align} \nonumber
\Big| \int N_1^{(n)} \cdot \gamma^{(n)} \Big| &= L^{-n} \Big| \int p_{mf}^c(\xi) o(|\xi|^2) \Big( \hat{u}^{n,c} \ast \hat{u}^{n,c}\Big) \cdot \overline{\hat{\gamma}}^{(n)}  \Big| \\ \nonumber 
&= L^{-n} \Big| \int p_{mf}^c(\xi) o(\xi) \Big( \hat{u}^{n,c} \ast \hat{u}^{n,c}\Big) \overline{o(\xi) \hat{\gamma}^{(n)}} \Big| \\ \nonumber
&\le L^{-n} ||o(\xi) \Big( \hat{u}^{n,c} \ast \hat{u}^{n,c} \Big)||_{L^2} ||o(\xi) \hat{\gamma}^{(n)}||_{L^2} = L^{-n} ||\hat{u}^{n,c}||_{L^1_\xi}||\xi \hat{u}^{n,c}||_{L^2} ||\xi \hat{\gamma}^{(n)}||_{L^2} \\
&\le \delta_w L^{-n} ||\xi \Big(\hat{\gamma}^{(n)} + \hat{\bar{a}}^{(n)} + \hat{\bar{b}}^{(n)} \Big)||_{L^2} ||\xi \hat{\gamma}^{(n)}||_{L^2}
\end{align}

Referring to (\ref{drn.2}), to control $N_2^{(n)}$, we have $h_2(\xi) = o(\xi)$, which does not contribute an additional $L^{-n}$. For this reason we use the blowup scaling of $u^{n,s}$ defined in (\ref{RNO.2}) (see (\ref{drn.2}), which shows an extra factor of $L^{-n(1-p)}$). We treat only the $\tilde{u}^{n,s}_\nu$ term from (\ref{drn.2}), with the $\tilde{u}^{n,s}_{\nu \nu}$ term being treated analogously: 
\begin{align} \nonumber
&\Big| \int N_2^{(n)} \cdot \gamma^{(n)} \Big| = L^{-n(1-p)}\Big| \int o(\xi) \Big( \hat{u}^{n,c} \ast \tilde{u}^{n,s}_\nu \Big) \cdot \overline{\hat{\gamma}^{(n)}}  \Big| \\ \nonumber
& \hspace{25 mm} =  L^{-n(1-p)} \Big| \int \Big( \hat{u}^{n,c} \ast \tilde{u}^{n,s}_\nu \Big) \cdot \overline{\overline{o(\xi)}\hat{\gamma}^{(n)}} \Big| \\ \nonumber
& \hspace{25 mm} \le  L^{-n(1-p)} ||\hat{u}^{n,c} \ast \tilde{u}^{n,s}_\nu||_{L^2} ||o(\xi)\hat{\gamma}^{(n)}||_{L^2}  \\ \nonumber
& \hspace{25 mm} \le  L^{-n(1-p)}  ||\hat{u}^{n,c}||_{L^1_\xi} ||\tilde{u}^{n,s}||_{B_{L^n}(0,0,1)}||\xi \hat{\gamma}^{(n)}||_{L^2} \\ \label{BUS.1}
& \hspace{25 mm} \le L^{-2n(1-p)} \delta_w^2 ||\tilde{u}^{n,s}||_{B_{L^n}(0,0,0)}^2 + \frac{1}{100}||\xi \hat{\gamma}^{(n)}||_{L^2}^2. 
\end{align}

Note that the mode filter $p_{mf}^c(\xi)$ which is implicitly in the $o(\xi)$ multiplier above allows us to put the Bloch norm on the stable unknown. This argument proceeds as follows (we simply show the $n = 1$ case):
\begin{align} \nonumber
p_{mf}^c(\xi) \hat{u}^c \ast \tilde{u}^{s} &= \int \hat{u}^c(\xi-\eta) \Big[ p_{mf}^c(\eta) +  p_{mf}^c(\xi) - p_{mf}^c(\eta) \Big] \tilde{u}^s(\eta) d\eta \\
&= \int \hat{u}^c(\xi - \eta) \Big[p_{mf}^c(\xi) - p_{mf}^c(\eta) \Big] \tilde{u}^s(\eta) d\eta + \hat{u}^c \ast \Big( p_{mf}^c \tilde{u}^{s} \Big). 
\end{align}

The latter term above can be estimated by the Bloch norm, as the modes of $\tilde{u}^s$ are directly restricted by the multiplier $p_{mf}^c$, and so we must only consider the first integral. By the Fourier support of $\hat{u}^c$, we have $-\frac{l_1}{4} \le \xi - \eta \le 0$, and by the Fourier support of $\tilde{u}^s$, we have $-\infty \le \eta \le -\frac{l_1}{8}$. From here it is obvious that taking $L$ sufficiently large relative to $\frac{l_1}{4}$, we can insert an indicator function for free: 
\begin{align}
\int \hat{u}^c(\xi-\eta) \Big[p_{mf}^c(\xi) - p_{mf}^c(\eta) \Big] 1_{\{\eta \ge -Lk\}} \tilde{u}^s(\eta) d\eta. 
\end{align}

This can now be estimated via the $B_{L}(0,0,0)$ norm of $\tilde{u}^s$. Referring to (\ref{drn.3}), we note that the final nonlinearity, $N_3^{(n)}$ contains the cubic and higher nonlinear terms, and so we may control in the same manner as $N_1^{(n)}, N_2^{(n)}$. We must now estimate the irrelevant linear term from $\mathcal{B}^n$: 
\begin{align} \nonumber
 L^{-n} \int o(\xi^3) \hat{\bar{b}}^{(n)} \overline{\hat{\gamma}^{(n)}} & \lesssim L^{-2n}||\xi^2 \hat{\bar{b}}^{(n)}||_{L^2}^2 + \frac{1}{100} ||\xi \hat{\gamma}^{(n)}||_{L^2}^2. 
\end{align}

Finally, we must control the nonlinearity arising from Burgers self-interaction (see the definition in (\ref{defn.nbc})), for which the rapid decay of $\bar{b}^{n,s}$ dominates:
\begin{align}
\Big| \int p_{mf}^c i \beta \xi N_b^c \cdot \overline{\hat{\gamma}^{(n)}}  \Big| = \Big| \int p_{mf}^c i \beta \xi \Big( \bar{b}^{n,c} \ast \bar{b}^{n,s} \Big) \cdot  \overline{\hat{\gamma}^{(n)}} \Big| \le L^{-n} \delta_w \Big( ||\hat{\bar{b}}^{n,c}||_{L^2}^2 + ||\xi \hat{\gamma}^{(n)}||_{L^2}^2 \Big). 
\end{align}

Putting everything together then gives: 
\begin{align} \nonumber
&\partial_t \int |\hat{\gamma}^{(n)}|^2 \ud z + \int |\xi \gamma^{(n)}_z|^2 \ud z \\
& \hspace{5 mm} \lesssim \delta_w^2 ||\hat{\gamma}^{(n)}||_{L^2}^2 + \delta_w^2 ||\bar{a}^{(n)}||_{H^1}^2 + L^{-2n} ||\bar{b}^{(n)}||_{H^2}^2 + \delta_w^2 L^{-2n(1-p)}||\tilde{u}^{n,s}||_{B_{L^n}(0,0,0)}^2. 
\end{align}

An application of Gronwall's Lemma then yields: 
\begin{align} \nonumber
&\sup_{t \in [1,L^2]} \int |\hat{\gamma}^{(n)}|^2 \ud z + \int_{1}^{L^2} \int |\xi \hat{\gamma}^{(n)}|^2 \ud z \ud t \\
\label{z.1} & \hspace{5 mm} \lesssim \delta_w^2 L^2 ||g^{(n)}||_{L^2}^2 +  L^{-2n} \phi_d \delta_w + \delta_w^2 L^{-2n(1-p)} \int_1^{L^2} ||\tilde{u}^{n,s}(t)||_{B_{L^n}(0,0,0)}^2 \ud t. 
\end{align}

We have used the ability to select $\delta_w$ much smaller than $L^2$, by selecting $J_1(L)$ appropriately in (\ref{crit.0}) thereby making the exponential an order-1 constant: 
\begin{equation}
\exp\{L^2\delta_w^2\} \lesssim 1, \text{ for } \delta_w \text{ sufficiently small. }
\end{equation}

\subsection*{Step 1b: Weighted $L^2(2)$ Estimate of $\gamma^{(n)}$}

The corresponding weighted estimate for $\gamma^{(n)}$ in $L^2(2)$ can be given by successively applying the multipliers $\gamma^{(n)}z, \gamma^{(n)}z^2$. We omit these details as they are largely identical to the calculation just performed. The delicate matter in applying these multipliers is to absorb only those weights given in (\ref{criteria.2}) into $\bar{b}^{(n)}$ profiles. As such, we display the delicate terms below. First, the Burgers contribution from $\mathcal{M}$, as shown in (\ref{DEF.M}):
\begin{align} \nonumber
\int i\xi \Big( \hat{\bar{a}}^{(n)} + \hat{\gamma}^{(n)} \Big) \ast \hat{\bar{b}}^{c,n} \cdot \overline{\mathcal{F}\{ z^4 \gamma \}} &\le ||\hat{\bar{b}}^{c,n}||_{L^1_\xi(1)} ||\mathcal{F}\{z^2 ( \bar{a}^{(n)} + \gamma^{(n)}) \}||_{L^2} ||\mathcal{F}\{z^2 \gamma^{(n)} \}||_{L^2} \\
& + l.o.t(z).
\end{align}

The terms contained in $l.ot.(z)$ are lower-order in $z$ than $z^4$. The essential point (which will be in use without further mention) is that we can avoid placing weights on linearized $\bar{b}^{c,n}$ by using the nonlinearity. Next, we approach the slightly more delicate $\mathcal{N}_1^{(n)}$:
\begin{align} \nonumber
L^{-n} \Big|\int &p_{mf}^c o(\xi^2) \Big( \hat{u}^{n,c} \ast \hat{u}^{n,c} \Big) \cdot \overline{\mathcal{F}\{ z^4 \gamma^{(n)}\}} \Big|= L^{-n}\Big| \int o(\xi) \Big(\hat{u}^{n,c} \ast \hat{u}^{n,c} \Big) \cdot \overline{z^4 \gamma^{(n)}_z}  \Big| + l.o.t(m) \\
& = L^{-n} \Big| \int o(\xi) \Big( \hat{a}^{n,c} + \hat{\bar{b}}^{n,c} \Big) \ast \Big( \hat{a}^{n,c} + \hat{\bar{b}}^{n,c} \Big) \cdot \overline{\mathcal{F}\{z^4 \gamma^{(n)}_z\}}  \Big| + l.o.t(m).
\end{align}

As $\hat{a}^{n,c} = \gamma^{(n)} + \hat{\bar{a}}^{(n)}$ can accept the weights of $z^2$, the most delicate term is the Burgers quadratic interaction: 
\begin{align}
 L^{-n} \Big| \int o(\xi) \Big( \hat{\bar{b}}^{n,c} \ast \hat{\bar{b}}^{n,c} \Big)  \cdot \overline{\mathcal{F}\{z^4 \gamma^{(n)}_z\}}  \Big| \le L^{-n} ||z^{\frac{1}{2}} \bar{b}^{n,c}||_{L^\infty} ||z^{\frac{3}{2}} \bar{b}^{n,c}||_{L^2} ||z^2 \gamma^{(n)}_z||_{L^2}.
\end{align}

A comparison to (\ref{criteria.2}) then shows that $||z^{\frac{1}{2}} \bar{b}^{n,c}||_{L^\infty} ||z^{\frac{3}{2}} \bar{b}^{n,c}||_{L^2} \le C(\phi_d)$. The next Burgers involvement occurs through $\mathcal{B}^{n}$: 
\begin{align}
L^{-n}\int o(\xi^3) \hat{\bar{b}}^{n,c} \overline{\mathcal{F}\{\gamma^{(n)} z^4 \} } \le L^{-n}||z^2 \bar{b}^{n,c}_{zz}||_{L^2}||z^2 \gamma^{(n)}_z||_{L^2} + l.o.t(m). 
\end{align}

For the second term in $\mathcal{B}^{(n)}$ we must interpolate:
\begin{align} \nonumber
\int p_{mf}^c i\xi &\Big(\hat{\bar{b}}^{n,c} \ast \hat{\bar{b}}^{n,s} \Big) \cdot \overline{\mathcal{F}\{ z^4 \gamma^{(n)} \}} \le ||z^{\frac{3}{2}} \bar{b}^{n,c}_z  ||_{L^2} ||z^{\frac{1}{2}} \bar{b}^{n,s}||_{L^\infty} || z^2 \gamma^{(n)}||_{L^2} \\ \nonumber
& + ||z^{\frac{3}{2}} \bar{b}^{n,s}_z||_{L^2}||z^{\frac{1}{2}}\bar{b}^{n,c}||_{L^\infty} ||z^2 \gamma^{(n)}||_{L^2} + l.o.t(m) \\
& \hspace{40 mm} \le \rho(\delta_w) C(\phi_d) ||z^2 \gamma^{(n)}||_{L^2}. 
\end{align}

Above, we have used that the weight of $z^{\frac{1}{2}}$ is sub-critical for the $L^\infty$ norm (it can accept a full $z$, see (\ref{criteria.2})), and therefore we retain some smallness in $\delta$. Thus, $\rho(\delta_w)$ is a function which can be made small by making $\delta_w$ small. 

\subsection*{Step 2: Estimate of $\gamma^{(n)}_z$:}

The next step is to differentiate the equation in $z$ on the spatial side, or multiply equation (\ref{sys.1}) by a factor of $i \xi$. We then apply the multiplier $i \xi \hat{\gamma}^{(n)}$  to this equation. Note that we refrain from giving details regarding the Fourier-support of each term, as this has been done in the previous step and the arguments are identical. First, the positive terms: 
\begin{equation} \label{pos.2.2}
\int \Big(\partial_t - \lambda(\xi) \Big) i\xi \hat{\gamma}^{(n)} \cdot \overline{i\xi \hat{\gamma}^{(n)}} \ge \partial_t \int |\xi \gamma^{(n)}|^2 + \int |\xi^2 \hat{\gamma}^{(n)}|^2
\end{equation}

For the marginal nonlinearity, performing an identical calculation to the first step yields: 
\begin{align} \label{abs}
\Big| \int \mathcal{M}^{(n)} (\hat{a}^{(n)}, \hat{\gamma}^{(n)}, \bar{b}^{c,n}) i \xi \cdot \overline{i \xi \hat{\gamma}^{(n)}} \Big| \lesssim \delta_w^2 ||\bar{a}^{(n)}||_{H^1}^2 + \delta_w^2 ||\gamma^{(n)}||_{H^1}^2 + \frac{1}{100} |||\xi|^2 \hat{\gamma}^{(n)}||_{L^2}^2. 
\end{align}

The latter term in (\ref{abs}) can be absorbed into (\ref{pos.2.2}). Now, we address the irrelevant components of the nonlinearity, after integrating by parts once: 
\begin{align} \nonumber
\Big| \int \hat{N}_1^{(n)} \cdot \overline{|\xi|^2 \hat{\gamma}^{(n)}} \Big| &= L^{-n} \Big| \int o(|\xi|^2) \Big( \hat{u}^{n,c} \ast \hat{u}^{n,c} \Big) \cdot \overline{|\xi|^2 \hat{\gamma}}^{(n)}  \Big| = L^{-n} \Big| \int o(\xi^2) \Big( \hat{u}^{n,c} \ast \hat{u}^{n,c} \Big) \overline{|\xi|^2 \hat{\gamma}^{(n)}} \Big| \\ \nonumber
&\le L^{-n} ||\hat{u}^{n,c}||_{L^1_\xi(1)} ||u^{n,c}||_{H^2} |||\xi|^2\hat{\gamma}^{(n)}||_{L^2}
\end{align}

To control $N_2$, we have $h_2(\xi) = o(\xi)$, which does not contribute an additional $L^{-n}$. For this reason we use the blowup scaling of $u^{n,s}$ defined in (\ref{scaling}). Again, we treat the $\tilde{u}^{n,s}_\nu$ term, with the $u^{n,s}_{\nu \nu}$ term being treated analogously:
\begin{align} \nonumber
\Big| \int \hat{N}_2^{(n)} \cdot \overline{|\xi|^2 \hat{\gamma}^{(n)}} \Big| &= L^{-n(1-p)} \Big| \int o(\xi) \Big( \hat{u}^{n,c} \ast \tilde{u}^{n,s}_\nu \Big) \cdot \overline{|\xi|^2\hat{\gamma}^{(n)}}  \Big| \\ 
& \le L^{-n(1-p)} ||\tilde{u}^{n,s}||_{B_{L^n}(1,0,1)}||\hat{u}^{n,c}||_{L^1_\xi}|||\xi|^2 \hat{\gamma}^{(n)}||_{L^2}. 
\end{align}

For the final nonlinearity, $N_3$, which contains the cubic and higher nonlinear terms, we may control in the same manner as $N_2$. Finally, we must estimate the irrelevant linear term: 
\begin{align} \nonumber
 L^{-n} \int o(\xi^3) \hat{\bar{b}}^{(n)} \overline{\hat{\gamma^{(n)}_{zz}}} &\le L^{-n} |||\xi|^2\hat{\bar{b}}^{(n)}||_{L^2} |||\xi|^2 \hat{\gamma}^{(n)}||_{L^2} \\
&\lesssim L^{-2n}||(i\xi)^3 \hat{\bar{b}}^{(n)}||_{L^2}^2 + \frac{1}{100} |||\xi|^2 \hat{\gamma}^{(n)}||_{L^2}^2. 
\end{align}

The estimate of the Burgers self-interaction term, $N^c_b$ is given via: 
\begin{equation}
\Big| \int p_{mf}^c \beta \Big(-|\xi|^2 \Big) N^c_b \overline{i\xi \hat{\gamma}^{(n)}} \Big| \lesssim \delta_w \Big[ ||\xi \hat{\bar{b}}^{n,c}||_{L^2}^2 + |||\xi|^2 \hat{\gamma}^{(n)}||_{L^2}^2 \Big].
\end{equation}

Putting everything together then gives: 
\begin{align} \nonumber
\partial_t \int |\xi \hat{\gamma}^{(n)}|^2 + \int |\xi^2 \hat{\gamma}^{(n)}|^2 &\lesssim \delta_w^2 ||\xi \hat{\gamma}^{(n)}||_{L^2}^2 + \delta_w^2 ||\hat{\gamma}^{(n)}||_{L^2}^2 + \delta_w^2 ||\bar{a}^{(n)}||_{H^2}^2 + L^{-2n} ||\bar{b}^{(n, c)}||_{H^3}^2 \\ &+ \delta_w^2 L^{-2n(1-p)}||u^{n,s}||_{H^1}^2. 
\end{align}

In the same manner as in (\ref{z.1}), an application of Gronwall's Lemma then yields: 
\begin{align} \nonumber
\sup_{t \in [1,L^2]} \int |\xi \hat{\gamma}^{(n)}|^2 + \int_1^{L^2} \int |\xi^2 \hat{\gamma}^{(n)}|^2 \ud z \ud t &\lesssim \delta_w^2 L^2 ||g^{(n)}||_{H^1}^2 +  L^{-2n} \phi_d \delta \\  \label{dz.1}
&  + \delta_w^2 L^{-2n(1-p)} \int_1^{L^2} ||\tilde{u}^{n,s}(t)||_{B_{L^n}(1,0,0)}^2 \ud t. 
\end{align} 

\subsection*{Step 3: Estimate of $\gamma^{(n)}_{zz}$:}

We now take a further derivative spatial derivative, $\partial_z$, of our system, which corresponds to applying the Fourier multiplier $-|\xi|^2$ to (\ref{sys.1}). Once this is done, we multiply by $\gamma^{(n)}_{zz}$ and integrate by parts. This gives the positive terms: 
\begin{align}
\int \Big( \partial_t - \lambda(\xi) \Big) \Big(-|\xi|^2 \hat{\gamma}^{(n)} \Big) \cdot  \overline{-|\xi|^2 \hat{\gamma}^{(n)} } = \partial_t \int |\xi^2 \hat{\gamma}^{(n)}|^2 + \int |\xi^3 \hat{\gamma}^{(n)}|^2
\end{align}

The marginal nonlinearity is now treated, as usual, via an integration by parts: 
\begin{align}
\Big|\int \mathcal{M}^{(n)} \Big( -|\xi|^2 \Big) \cdot \Big(-|\xi|^2 \hat{\gamma}^{(n)} \Big) \Big| \le \delta_w^2 ||\xi^2 \{\hat{\bar{a}}^{(n)} + \hat{\gamma}^{(n)} \} ||_{L^2}^2 + \frac{1}{100} ||\xi^3 \hat{\gamma}^{(n)}||_{L^2}^2. 
\end{align}

Now, the first irrelevant nonlinearity: 
\begin{align} \nonumber
\Big| \int \xi^2 \hat{N}_1^{(n)} \cdot \overline{\xi^2 \hat{\gamma}^{(n)}} \Big| &=  L^{-n}\Big| \int o(\xi^3) \Big( \hat{u}^{n,c} \ast \hat{u}^{n,c} \Big) \cdot \overline{\hat{\gamma^{(n)}}_{zzz}}   \Big| \\ \nonumber
&\le L^{-n} ||\hat{u}^{n,c}||_{L^1_\xi(1)} ||u^{n,c}||_{H^3} ||\xi^3 \hat{\gamma}^{(n)}||_{L^2} \\
&\le L^{-n} \delta_w ||(i\xi)^3 \{\hat{\bar{b}}^{(n)} + \hat{\bar{a}}^{(n)} + \hat{\gamma}^{(n)}\}||_{L^2}||\xi^3 \hat{\gamma}^{(n)}||_{L^2}. 
\end{align}

Next, we have: 
\begin{align} \nonumber
\Big| \int \xi^2 \hat{N}^{(n)}_2 \cdot \overline{\xi^2 \hat{\gamma}^{(n)}} \Big| &= L^{-n(1-p)} \Big| \int o(\xi^2) \Big( \hat{u}^{n,c} \ast \tilde{u}^{n,s}_\nu \Big) \cdot \overline{\mathcal{F}\{\gamma^{(n)}_{zzz}\}}  \Big| \\ \nonumber
&\le L^{-n(1-p)} ||\hat{u}^{n,c}||_{L^1_\xi(1)} ||\tilde{u}^{n,s}||_{B_{L^n}(2,0,1)} ||\xi^3 \hat{\gamma}^{(n)}||_{L^2}.
\end{align}

The third nonlinearity $N_3$ is controlled in similar ways to $N_1$, $N_2$.  Last, we must estimate the irrelevant linear contribution: 
\begin{align}
\int (i \xi) L^{-n} o(\xi^3) \hat{\bar{b}}^{(n)} \} \cdot \overline{(i\xi)^3 \hat{\gamma}^{(n)}} \lesssim \frac{1}{100} ||\xi^3 \hat{\gamma}^{(n)}||_{L^2}^2 + \phi_d \delta.
\end{align}

Putting everything together then gives: 
\begin{align} \nonumber
\partial_t \int |\xi^2 \hat{\gamma}^{(n)}|^2 + \int |(i\xi)^3\hat{\gamma}^{(n)}|^2 &\lesssim \delta_w^2 \Big( |||\xi|^2 \hat{\gamma}^{(n)}||_{L^2}^2  + ||\gamma^{(n)}||_{H^1}^2 \Big) + \delta_w^2 ||\bar{a}^{(n)}||_{H^3}^2 + L^{-2n} ||\bar{b}^{(n)}||_{H^4}^2 \\ &+ \delta_w^2 L^{-2n}||\tilde{u}^{n,s}||_{B_{L^n}(2,0,0)}^2. 
\end{align}

Via an application of Gronwall's Lemma and applying the previously established bounds in (\ref{z.1}) and (\ref{dz.1}), we obtain: 
\begin{align} \label{dzz.1}
\sup_{t \in [1,L^2]} \int |\xi^2 \hat{\gamma}^{(n)}|^2 &+ \int_1^{L^2} \int |(i\xi)^3 \hat{\gamma}^{(n)}|^2 \ud z \ud t \\  \nonumber 
& \lesssim \delta_w^2 L^2 ||g^{(n)}||_{H^2}^2 +  L^{-2n} \phi_d \delta + \delta_w^2 L^{-2n} \int_1^{L^2} ||u^{n,s}(t)||_{B_{L^n}(2,0,0)}^2 \ud t. 
\end{align}

\subsection*{Step 4: Weighted $H^2(2)$ Estimates}

We have given the essential calculations once the weight $z^4 \gamma^{(n)}$ has been applied to the system in Step 1. The corresponding calculations for Steps 2 and 3 are nearly identical, and the aim to not violate criteria (\ref{criteria.2}) becomes easier with higher derivatives, as can be seen from (\ref{criteria.2}). Thus, we summarize: 
\begin{align} \label{good.1}
\sup_{t \in [1, L^2]} &\int \{| \gamma^{(n)}|^2, |\gamma^{(n)}_z|^2, |\gamma^{(n)}_{zz}|^2 \} z^{2m} + \int_1^{L^2} \int \{ |\gamma^{(n)}_z|^2, |\gamma^{(n)}_{zz}|^2, |\gamma^{(n)}_{zzz}|^2 \} z^{2m} \ud z \ud t \\ \nonumber &\lesssim J_E(L) \Big[ \delta_w^2 ||g^{(n)}||_{H^2(2)}^2 + L^{-2n}C_E(\phi_d) + \delta_w^2 L^{-2n(1-p)} \int_1^{L^2} ||\tilde{u}^{n,s}||_{B_{L^n}(2,0,0)} \ud t \Big].
\end{align}

Accumulating Steps 1-4 via (\ref{good.1}) results in Theorem \ref{thm.1}, up to defining the functions $J_1, C_1, J_E, C_E$ appropriately. 

\subsection*{Step 5: Stable Estimates}

We now address the stable component which arises in (\ref{good.1}). 

\begin{lemma} Let $\phi_d < \infty$ be prescribed. For any $L, \delta_w$, there exists universal functions $J_s(L), C_s(\phi_d)$, and $c_1 > 0$, such that for all $n \ge 0$: 
\begin{align} \nonumber
||\tilde{u}^{n,s}(L^2,\cdot, \cdot) ||_{B_{L^n}(2,2,2)} & \lesssim e^{-c_1 L^{2n+2} } ||\tilde{g}^{(n,s)} ||_{B_{L^n}(2,2,2)} + J_S(L) L^{-n} \delta_w \Big( ||g^{(n)}||_{H^2(2)} \\\label{quad}
&  + ||\tilde{g}^{(n,s)}||_{H^2(2)} \Big) + J_S(L) C_S(\phi_d) L^{-n}.
\end{align}
\end{lemma}

\begin{proof}

Via equation (\ref{eqn.stable}), by performing a basic energy estimate analogous to Lemma \ref{thm.LE}, and recalling (\ref{stable.spec}), we have for some functions $C(\phi_d), J(L)$: 
\begin{align} \nonumber
||\tilde{u}^{n,s}(L^2,\cdot, \cdot) ||_{B_{L^n}(2,2,2)} &\le e^{-c_1 L^{2n+2} } ||\tilde{g}^{(n,s)} ||_{B_{L^n}(2,2,2)} + J(L) L^{-n} \delta_w \Big( ||g^{(n)}||_{H^2(2)}  \\ \label{quad.1}
& + ||\tilde{g}^{(n,s)}||_{B_{L^n}(2,2,2)} \Big) + J(L) L^{-n}C(\phi_d),
\end{align}

so long as $||\hat{u}^{(n,c)}||_{L^1_\xi(1)}  \lesssim \delta_w$. Here the $\phi_d$ contribution comes in due to the $\bar{b}^{(n)}$ in estimating $||u^{n,c}||_{H^2(2)}$. We now pick $J_S, C_S$ appropriately, based on (\ref{quad.1}).

\end{proof}

\subsection*{Step 6: Iteration to $B_\epsilon(0) \subset H^2(2)$}

We recall the definition of $g^{(n)}$ given in (\ref{eqn.stable.1}). We are now ready to iterate the renormalization procedure.
\begin{proposition} Let $\phi_d$ be prescribed. Denote: 
\begin{align} \label{defn.rho.n}
\rho^{(n)} = ||g^{(n)}||_{H^2(2)} + ||\tilde{g}^{(n,s)}||_{B_{L^n}(2,2,2)}
\end{align}

There exists universal functions $J_2(L), J_{E,2}(L), C_2(\phi_d)$, and $C_{E,2}(\phi_d)$, such that if (\ref{crit.0}) - (\ref{crit.-1}) are satisfied, and if:
\begin{align} \label{crit.1}
&L^{\frac{5}{2}+1}e^{-c_1 L^2} \ll 1,\\ \label{crit.2}
& \delta_w J_2(L) C_2(\phi_d) \ll 1, \text{ and }\\
&||\hat{g}^{(n)}, \tilde{g}^{(n,s)}||_{L^1_\xi(1)} + ||\hat{\bar{b}}^{(n)}||_{L^1_\xi(1)} \le \delta_w,
\end{align} 
 
then:
\begin{align} \label{fullit.1}
\rho^{(n+1)} \le \frac{c_0}{L}  \rho^{(n)} + J_{E,2}(L) C_{E,2}(\phi_d) L^{-n},
\end{align}
for a constant $c_0$ which is universal, independent of $\phi_d, L, \delta_w$ and $n$. 
\end{proposition}
\begin{proof}

We start with:  
\begin{align} \nonumber
||g^{(n+1)}(z)||_{H^2(2)} &= ||R \{ a^{(n)}(L^2, z) \}||_{H^2(2)} \\ \nonumber
& \stackrel{(\ref{separ.1})}{\le} ||R \{ \bar{a}^{(n)}(L^2, z) \}||_{H^2(2)} + ||R \{ \gamma^{(n)}(L^2, z) \}||_{H^2(2)} \\ \nonumber
& \stackrel{(\ref{contr.1.1})}{\lesssim} \frac{C}{L}||g^{(n)}(z)||_{H^2(2)} + L^\frac{5}{2} ||\gamma^{(n)}(L^2,z)||_{H^2(2)} \\ \nonumber
&\stackrel{(\ref{present.1})}{\lesssim}  \frac{C}{L}||g^{(n)}(z)||_{H^2(2)} + L^\frac{5}{2} J_E(L)  \Big[ \delta_w ||g^{(n)}||_{H^2(2)} + L^{-2n} C_E(\phi_d) \\  \label{it.1}
& \hspace{50 mm} + L^{-n(1-p)} \delta_w ||\tilde{g}^{(n,s)}||_{B_{L^n}(2,2,2)}  \Big],
\end{align}

For the stable component, we appeal to estimate (\ref{quad}):
\begin{align} \nonumber
||\tilde{g}^{(n+1,s)}||_{B_{L^{n+1}}(2,2,2)} &\stackrel{(\ref{defn.b.b})}{\le} L^\frac{5}{2} ||\tilde{u}^{(n,s)}(L^2, \cdot)||_{B_{L^{n+1}(2,2,2)}}\\ \nonumber
& \stackrel{(\ref{quad})}{\le} L^\frac{5}{2} \Big[ e^{-L^{2n+2}} ||\tilde{g}^{(n,s)}||_{B_{L^n(2,2,2)}} + J_S(L) \delta_w L^{-n} ||g^{(n)}||_{H^2(2)} \\ \nonumber
& \hspace{10 mm} + J_S(L) L^{-n} \delta_w ||\tilde{g}^{(n,s)}||_{B_{L^n}(2,2,2)} \Big] + L^\frac{5}{2} J_S(L) C_S(\phi_d) L^{-n}  \\ \label{it.st}
&\stackrel{(\ref{crit.1})}{\le} \frac{1}{L} ||\tilde{g}^{(n,s)}||_{B_{L^n}(2,2,2)} +\frac{1}{L} ||g^{(n)}||_{H^2(2)} + J(L)L^{-n} C_S(\phi_d),
\end{align}

where $J(L) = L^\frac{5}{2} J_S(L)$. By selecting $J_2, C_2$ appropriately, we have the simplified iteration: 
\begin{align}
||g^{(n+1)}||_{H^2(2)} \le \frac{C}{L}||g^{(n)}||_{H^2(2)} + \frac{1}{L} ||\tilde{g}^{n,s}||_{B_{L^n}(2,2,2)} + J(L) L^{-n}C(\phi_d), \\
||\tilde{g}^{(n+1,s)}||_{B_{L^{n+1}}(2,2,2)} \le \frac{1}{L}||g^{(n-1)}||_{H^2(2)} + \frac{1}{L} ||\tilde{g}^{n,s}||_{B_{L^n}(2,2,2)} + J(L) L^{-n}C_s(\phi_d),
\end{align}
By adding the above two equations, and selecting $J_{E,2}, C_{E,2}$ appropriately, we obtain the full iterative estimate: 
\begin{align}
\rho^{(n+1)} \le \frac{c_0}{L}  \rho^{(n)} + J_{E,2}(L) C_{E,2}(\phi_d) L^{-n} . 
\end{align}

\end{proof}

The reader should now refer back to (\ref{prove.3}) for the definition of $\rho_\ast$. Iterating above then gives: 
\begin{lemma} Let $\phi_d$ be prescribed, and fix any $\sigma > 0$. Suppose criteria (\ref{crit.1}) - (\ref{crit.2}) are satisfied, and suppose in addition that:
\begin{align} \label{crit.3}
&L^{-\sigma} c_0 < 1 \text{ where } c_0 \text{ is defined in (\ref{fullit.1})},\\\label{crit.4}
&||\hat{g}^{(k)}, \tilde{g}^{(k,s)}||_{L^1_\xi(1)} + ||\hat{\bar{b}}^{(k)}||_{L^1_\xi(1)} \le \delta_w \text{ for all } 1 \le k \le n. 
\end{align}
Then, there exist universal functions $C_{E,3}(\phi_d), J_{E,3}(\phi_d)$ such that: 
\begin{align}
\rho^{(n)} \le \{ C_{E,3}(\phi_d) + J_{E,3}(L) + \rho_\ast \} L^{-(1-\sigma)n}.
\end{align}
\end{lemma} 
\begin{proof}

By iterating estimate (\ref{fullit.1}), and selecting $C_{E,3}, J_{E,3}$ appropriately, we obtain: 
\begin{align}   \nonumber
\rho^{(n)} &\le \Big( \frac{c_0}{L} \Big)^n \rho_\ast + J_{E,2}(L) n \Big( \frac{c_0}{L}\Big)^n C_{E,2} \\ \label{it.3}
&\le \{ C_{E,3}(\phi_d) J_{E,3}(L) + \rho_\ast \} L^{-(1-\sigma)n}.
\end{align}

\end{proof}

\begin{corollary} \label{cor.N0} Fix any $\epsilon > 0$, and let $\phi_d, \rho^\ast < \infty$ be prescribed. Suppose criteria (\ref{crit.1}) - (\ref{crit.2}) and (\ref{crit.3}) - (\ref{crit.4}) are satisfied. Suppose that the initial data satisfy:  
\begin{align} \label{crit.ID}
||\hat{g}^{(k)}, \tilde{g}^{(k,s)}||_{L^1_\xi(1)} + \sup_{t \in [1,L^2]} ||\hat{\bar{b}}^{(k)}||_{L^1_\xi(1)} \le \delta_w \text{ for all } 1 \le k \le N_0. 
\end{align}

So long as the following estimate is satisfied: 
\begin{align} \label{eps.1}
\frac{C_{E,3}(\phi_d) J_{E,3}(L) + \rho_\ast}{L^{N_0(1-\sigma)}} \le \epsilon,
\end{align}
then: 
\begin{align} \label{N0conc}
\rho^{(N_0)} \le \epsilon. 
\end{align}
\end{corollary}

\section{Functional Framework} \label{s.ff}

\subsection{Cole-Hopf Mapping}

In this section, we give the functional framework we exploit to overcome the large-data asymptotics at hand. The results which are most important to carry out the analysis of Section \ref{step.final} are Propositions \ref{prop.Burg.1}, \ref{prop.smooth}, and estimates (\ref{l2.new}), (\ref{check.2}). Suppose $b$ is a nonnegative solution to the Burgers equation: \footnote{For the sake of presentation, we shall drop the parameters $\alpha, \beta$ from \ref{fA} for this section. The Cole-Hopf mapping in (\ref{defn.CH}) and the corresponding analysis of this section can easily be scaled to fit the parameters $(\alpha,\beta)$. For all applications of the results from this section, the Burgers flow, $b$, will be taken to be a variant of $f_A^\ast$, as defined in (\ref{fA}). In particular, we will always apply it to positive profiles.} 
\begin{equation} \label{our.b}
b_t - b_{xx} =  bb_x = \frac{1}{2}\partial_x (b^2), \hspace{3 mm} ||b(1,x)||_{L^1} = ||b_0||_{L^1} = \phi_d < \infty, \hspace{3 mm} b(t,x) \ge 0.  
\end{equation}

We will also require the following hypothesis on $b$, which will be satisfied for each application of the estimates we establish in this section:
\begin{hypothesis} Given $\phi_d$ in (\ref{our.b}), let $0 < L < \infty$ be any prescribed constant. Recall the definition of the renormalization map $R_L$ from (\ref{defm.RGM.1}). We assume $b$ satisfies: 
\begin{align} \label{criteria}
\sup_L || z^k \partial_z^k R_L b(L^2,\cdot)||_{L^2} + ||z^{k+1} \partial_z^k R_L b(L^2, \cdot)||_{L^\infty}  \le C(\phi_d),
\end{align}
where the majorizing term in (\ref{criteria}) depends poorly on $\phi_d$, but is independent of $L$. 
\end{hypothesis} 

The Cole-Hopf map, defined via: 
\begin{equation} \label{defn.CH}
h(t,x) = \mathcal{N}(b) := \exp\{ \frac{1}{2} \int_{-\infty}^x b(t,y) dy \}, \hspace{3 mm} b(t,x) = \mathcal{N}^{-1}(h) =  2\frac{h_x}{h},
\end{equation}

takes $b$ to a solution, $h$, of the heat equation, $\partial_t h - \partial_{xx}h = 0$, with initial data resembling a ``front" with separation $\phi_d$. $\mathcal{N}^{-1}$ takes front solutions to the heat equation to a localized Burgers flow. Note that $h \ge 1$ by the positivity of $b \ge 0$, so the quotient in (\ref{defn.CH}) is well-defined. We start with the following immediate consequence of this relationship: 

\begin{lemma}\label{lemma.burgers.decay}[Uniform Decay] Suppose $b_0 \in L^1 \cap L^\infty$. Then for any $t \ge 1$, 
\begin{align} \label{BD.1}
&||b(t)||_{L^\infty} \lesssim \frac{||b_0||_{L^1} e^{||b_0||_{L^1}}}{t^{1/2}}, \\ \label{BD.2}
&||b_x(t)||_{L^1} \lesssim ||b_0||_{L^1} t^{-\frac{1}{2}}. 
\end{align}
\end{lemma}
\begin{proof}

As above, inverting Cole-Hopf yields $b(t,x) = 2\frac{h_x(t,x)}{h(t,x)}.$ $h(t,x)$ solves the heat equation with initial condition $h(1,x) = e^{\frac{1}{2}\int_{-\infty}^x b_0(y)dy}$. $\partial_x h$ also satisfies the heat equation, this time with initial condition $\partial_x h(1,x) = \frac{1}{2}b_0(x)e^{\frac{1}{2}\int_{-\infty}^x b_0(y) dy}$.  Then via standard heat-kernel estimates: 
\begin{equation}
||b(t)||_{L^\infty} \le 2\min |h| ||h_x(t)||_{L^\infty} \lesssim ||h_x(1)||_{L^1} t^{-\frac{1}{2}} \le ||b_0||_{L^1} \exp \{ ||b_0||_{L^1} \} t^{-\frac{1}{2}}. 
\end{equation}

The estimate (\ref{BD.2}) follows in a similar manner: by differentiating Cole-Hopf, we have $b_x = \frac{h_{xx}}{h} - \frac{h_x^2}{h^2}$, both of which are controlled by $t^{-\frac{1}{2}}||h_x(1)||_{L^1}$. By differentating Cole-Hopf, $||h_x(1,x)||_{L^1} \le ||b_0||_{L^1}C(\phi_d)$. 

\end{proof}

\subsection{The Operator $\Phi_b(t-1) $: Contractive Properties} \label{SSC}

Given our Burgers flow, $b$, define the linearized operator via: 
\begin{align} \label{defn.Sb}
S_{b}a := a_{xx} + a_x b + a b_x. 
\end{align}

The goal in this section is to obtain similar estimates for the solution operator to the flow $\partial_t - S_b$ as that of $\partial_t - \Delta$, without using any smallness of $b$. Let us introduce the preliminary notations: 
\begin{align}
&h(t,x) = \phi^H_t h_0 \iff \partial_t h = \partial_{xx} h, \hspace{3 mm} h(1,x) = h_0, \\
&b(t,x) = \phi^B_t b_0 \iff \partial_t b = \partial_{xx}b + bb_x, \hspace{3 mm} b(1,x) = b_0, \\
&a(t,x) = \Phi_b(t) a_0 \iff \partial_t a = S_b a, \hspace{3 mm} a(1,x) = a_0. 
\end{align}

The main result of Subsection \ref{SSC} is: 
\begin{proposition} \label{prop.Burg.1}Suppose $b \ge 0$ is any solution to Burgers equation, with mean $\phi_d$, and such that the criteria in (\ref{criteria}) is met. Let $0 < L < \infty$ be arbitrary. Then the semigroup $\Phi_b(L^2-1) $ contracts on $H^2(2)$ for mean-zero data: 
\begin{equation} \label{contraction.Sb}
||R_L \Phi_b(L^2-1) g||_{H^2(2)} \le \frac{C(\phi_d)}{L} ||g||_{H^2(2)} \hspace{3 mm} \text{ if } \hspace{3 mm}\int_{\mathbb{R}} g = 0. 
\end{equation}
The factor $C(\phi_d)$ depends poorly on $\phi_d$, but is independent of $L$. 
\end{proposition}

\begin{remark} This estimate should be compared to the estimate (\ref{contr.1.1}), which was for the heat semigroup, $\phi^H_t$. 
\end{remark}

We will now develop the machinery to prove Proposition \ref{prop.Burg.1}. On the nonlinear level, we have the following relationships between flows for each fixed $t \ge 1$ via the Cole-Hopf mapping: 

\vspace{3 mm}

\hspace{50 mm} \begin{tikzpicture}
  \node (A) {$L^1$};
  \node (B) [below of=A] {$L^1$};
   \node (C) [right of=A] {$L^\infty \cap \partial_x L^1$};
  \node (D) [right of=B] {$L^\infty \cap \partial_x L^1$};
  \draw[->] (A) to node [swap]{ $\phi_t^B$} (B);
  \draw[->] (A) to node {$\mathcal{N}$} (C);
  \draw[->] (B) to node [swap] {$\mathcal{N}$} (D);
   \draw[->] (C) to node {$\phi_t^H$} (D);
   
\end{tikzpicture}

The space $\partial_x L^1$ has the natural definition: 
\begin{equation}
||f||_{\partial_x L^1} := ||\partial_x f ||_{L^1}. 
\end{equation}

It is clear from (\ref{defn.CH}) that $\mathcal{N}$ maps elements from $L^1$ to $L^\infty \cap \partial_x L^1$. Linearizing the Cole-Hopf map around the fixed Burgers flow, $b$, then yields: 
\begin{align} \nonumber
d\mathcal{N}|_{b(t)} a_0 &= \frac{d}{d\epsilon}|_{\epsilon = 0} \mathcal{N}(b(t)+\epsilon a_0) \\  \label{dN} 
& = \frac{d}{d\epsilon}|_{\epsilon = 0} e^{\frac{1}{2}\int_{-\infty}^x b(t,y) + \epsilon a_0(y) dy}  = e^{\frac{1}{2}\int_{-\infty}^x b(t,y) \ud y}  \left( \frac{1}{2} \int_{-\infty}^x a_0(y) dy \right). 
\end{align}

This is a static calculation, valid for each $t \ge 1$. From here, we can write the inverse: 
\begin{equation} \label{dN.inv}
\Big(d\mathcal{N}|_{b(t)} \Big)^{-1} h = 2\partial_x \Big( e^{-\frac{1}{2}\int_{-\infty}^x b(t,y) dy} h \Big) = 2h_x e^{-\frac{1}{2}\int_{-\infty}^x b(t,y) dy} - h b e^{-\frac{1}{2}\int_{-\infty}^x b(t,y) dy}. 
\end{equation}

With these linearizations in hand, we obtain the following linearized version of the above diagram, where we identify tangent spaces of $L^1$ with $L^1$ and those of $L^\infty$ with $L^\infty$:

\vspace{3 mm}

\hspace{35  mm} \begin{tikzpicture}
  \node (A) {$L^1$};
  \node (B) [below of=A] {$L^1$};
   \node (C) [right of=A] {$L^\infty \cap \partial_x L^1$};
  \node (D) [right of=B] {$L^\infty \cap \partial_x L^1$};
  \draw[->] (A) to node [swap]{ $\Phi_b(t)  = d\phi_t^B$} (B);
  \draw[->] (A) to node {$d\mathcal{N}|_{b_0}$} (C);
  \draw[->] (B) to node [swap] {$d\mathcal{N}|_{b(t)}$} (D);
   \draw[->] (C) to node {$\phi_t^H$} (D);
   
\end{tikzpicture}

In particular, we have realized the solution operator, $\Phi_b(t-1) $ to the linear operator (\ref{defn.Sb}) in the above diagram. This is due to the following: 
\begin{lemma} The above diagram commutes, thereby giving the semigroup representation identity: 
\begin{equation} \label{rep.form}
\Phi_b(t-1) a_0 = \Big(d\mathcal{N}|_{b(t)} \Big)^{-1} \circ e^{\Delta(t-1)} \circ d\mathcal{N}|_{b_0} a_0. 
\end{equation}
\end{lemma}
\begin{proof}

Define 
\begin{equation}
h(t) = d\mathcal{N}|_{b(t)} \circ \Phi_b(t-1)  a_0 =  d\mathcal{N}|_{b(t)} a(t) = \frac{1}{2} e^{1/2 \int_{-\infty}^x b(t,y)dy} \int_{-\infty}^x a(t,y) dy,
\end{equation}
 which corresponds to ``moving down then right" on the above diagram. Then $h(t)$ solves the heat equation via direct calculation:

\begin{align} 
h_t &= \left( \int_{-\infty}^x a_t \right)\exp\left({\frac{1}{2}\int_{-\infty}^x b(t,y)dy} \right)+ \left( \int_{-\infty}^x a(t,y) dy \right) \exp \left( {\frac{1}{2}\int_{-\infty}^x b(t,y)dy} \right) \frac{1}{2}\int_{-\infty}^x b_t \\
h_x &= a(t,x) \exp\left( \frac{1}{2} \int_{-\infty}^x b(t,y)dy \right) +  \exp\left( \frac{1}{2} \int_{-\infty}^x b(t,y)dy \right) \frac{1}{2}b(t,x) \int_{-\infty}^x a(t,y)dy
\end{align}

\begin{align} \nonumber
 h_{xx} =& a_x \exp\left( \frac{1}{2} \int_{-\infty}^x b(t,y)dy \right) + \frac{1}{2}a\exp\left( \frac{1}{2} \int_{-\infty}^x b(t,y)dy \right) b \\ \nonumber
 &+ \exp \left( \frac{1}{2} \int_{-\infty}^x b(t,y)dy  \right) \frac{1}{2}b_x \left( \int_{-\infty}^x a(t,y)dy \right) + \exp\left( \frac{1}{2} \int_{-\infty}^x b(t,y)dy \right) \frac{1}{2}b(t,x) a(t,x) \\& + \frac{1}{4}b^2(t,x) \left( \int_{-\infty}^x a(t,y)dy \right) \exp \left( \frac{1}{2} \int_{-\infty}^x b(t,y)dy \right) 
\end{align}

Via $a_t = S_ba$ and $b_t = b_{xx} + \left(\frac{1}{2}b^2 \right)_x$ we confirm that $h_t = h_{xx}$.  Now define 
\begin{equation}
\tilde{h} = \phi_t^H \circ d\mathcal{N}|_{b_0} a_0,
\end{equation}

which on the above diagram corresponds to moving ``right, then down." By definition $\tilde{h}$ also solves the heat-equation, with the same initial data as $h$. Therefore $h = \tilde{h}$. 

\end{proof}

\begin{remark}[Notation] We set the following notational conventions, motivated by the above diagram and (\ref{rep.form}). $b$ will always denote a solution to Burgers flow, which takes place on the ``left-side" of the above diagram. $h$ will always denote a solution to the heat equation, so flows taking place on the ``right-side" of the above diagram. Due to the appearance of both $b_0$ and $b(t)$ in (\ref{rep.form}), we take care to denote this explicitly in the forthcoming calculations. 
\end{remark}

We now obtain the invertibility of the linear maps $d\mathcal{N}|_{b_0}, d\mathcal{N}|_{b(t)}^{-1}$ on the spaces shown in the above diagram: 
\begin{lemma} $d\mathcal{N}|_{b_0}$ is invertible as a map $L^1$ to $L^\infty \cap \partial_x L^1$, with inverse bounded independent of $t$: 
\begin{align}
||d\mathcal{N}|_{b_0} a_0 ||_{L^\infty \cap \partial_x L^1} \le C(\phi_d) ||a_0||_{L^1}, \hspace{3 mm} ||\Big(d \mathcal{N}|_{b(t)} \Big)^{-1} h||_{L^1}  \le C(\phi_d) ||h||_{L^\infty \cap \partial_x L^1}. 
\end{align}
\end{lemma}
Here $C(\phi_d)$ is exponential in $||b_0||_{L^1}$, and independent of $t \ge 1$. 
\begin{proof}

By the assumption that $b(t,x) \ge 0$, we have: 
\begin{align} \label{cons.l1}
||b(t,\cdot)||_{L^1} = \int b(t,x) \ud x = \int b_0(x) \ud x = \phi_d.
\end{align}

Consider a function $a_0 \in L^1$, and we have 
\begin{align} \nonumber
||d\mathcal{N}|_{b_0}a_0||_{L^\infty} &= ||e^{1/2 \int_{-\infty}^x b_0(y)dy} \int_{-\infty}^x a_0(y) dy ||_{L^\infty} \le ||e^{1/2 \int_{-\infty}^x b_0(y)dy} ||_{L^\infty} ||a_0||_{L^1} \\
&\le e^{||b(t)||_{L^1}} ||a_0||_{L^1} \le e^{||b_0||_{L^1}} ||a_0||_{L^1} = C(\phi_d) ||a_0||_{L^1}.
\end{align}

Similarly, we may compute 
\begin{align} \nonumber
||\partial_x d\mathcal{N}|_{b_0}a_0||_{L^1} &\le ||b_0 e^{\int_{-\infty}^x b_0} \int_{-\infty}^x a_0 ||_{L^1} + ||e^{\int_{-\infty}^x b_0}a_0||_{L^1} \\ 
& \le ||b_0||_{L^1} e^{||b_0||_{L^1}} ||a_0||_{L^1} = C(\phi_d)||a_0||_{L^1}.
\end{align}

In converse direction, using the $L^1$ conservation law in (\ref{cons.l1}), we have: 
\begin{align} \nonumber
||d\mathcal{N}|_{b(t)}^{-1}h||_{L^1} &\le ||e^{-\int_{-\infty}^x b(t)} h_x||_{L^1} + ||b(t)e^{-\int_{-\infty}^x b(t)} h||_{L^1} \\
&\le C(\phi_d)\Big( ||h_x||_{L^1} + ||h||_{L^\infty} \Big).
\end{align}

\end{proof}

In fact, the above diagrams may be generalized in the following way, for any $1 \le p, q \le \infty$: 

\hspace{50 mm} \begin{tikzpicture}
  \node (A) {$L^p$};
  \node (B) [below of=A] {$L^q$};
   \node (C) [right of=A] {$ \partial_x L^p$};
  \node (D) [right of=B] {$ \partial_x L^q$};
  \draw[->] (A) to node [swap]{ $\phi_t^B$} (B);
  \draw[->] (A) to node {$\mathcal{N}$} (C);
  \draw[->] (B) to node [swap] {$\mathcal{N}$} (D);
   \draw[->] (C) to node {$\phi_t^H$} (D);
   
\end{tikzpicture}

\begin{lemma} For any $p \in [1,\infty]$: 
\begin{equation} \label{est.dn.3}
||d\mathcal{N}|_{b_0} a_0||_{\partial_x L^p} \le C(\phi_d) ||a_0||_{L^p \cap L^1} \text{ and } ||\Big(d\mathcal{N}|_{b(t)} \Big)^{-1} h||_{L^p} \le C(\phi_d) ||h||_{\partial_x L^p \cap L^\infty}.
\end{equation} 
\end{lemma}
\begin{proof}
The proof follows directly from the expressions in (\ref{dN}) - (\ref{dN.inv}).
\end{proof}

Our aim is to study the operator $\Phi_b(L^2-1) $ on the function space $H^2(2)$. According to our representation formula in (\ref{rep.form}), 
\begin{equation}
\Phi_b(L^2-1)  g = \Big(d\mathcal{N}|_{b(L^2)} \Big)^{-1} \circ e^{\Delta(L^2-1)} \circ d\mathcal{N}|_{b_0} g.
\end{equation}

The linearizations of $\mathcal{N}$ behave in the following way in $H^2(2)$:
\begin{lemma} Suppose the criteria in (\ref{criteria}) are met. Then for any functions $h(x), w(x)$, such that $\int w \ud x = 0$: 
\begin{align} \label{est.dn.1}
||\Big(d\mathcal{N}|_{R_Lb(L^2,\cdot)} \Big)^{-1}h||_{H^2(2)} &\le C(\phi_d) ||h||_{\partial_x H^2(2)}, \\ \label{est.dn.2}
|| d\mathcal{N}|_{b_0}  w||_{\partial_x H^2(2)} &\le C(\phi_d)||w||_{H^2(2)}.
\end{align}
\end{lemma}
\begin{proof}

Via direct computation, 
\begin{align} \nonumber
||\Big(d\mathcal{N}|_{R_Lb(L^2,\cdot)} \Big)^{-1}h||_{H^2(2)} &\le ||e^{\int_{-\infty}^x R_Lb(L^2,\cdot)} h_x||_{H^2(2)} + ||R_Lb(L^2,\cdot) e^{\int_{-\infty}^x R_Lb(L^2,\cdot)} h||_{H^2(2)} \\ 
& \le C(\phi_d) \Big( ||h_x||_{H^2(2)} \Big).
\end{align}

The delicate term above, based on the absorption capabilities in (\ref{criteria}) is when no derivatives are present together with a weight of $z^2$. For this we exhibit the calculation:
\begin{align} \nonumber
||x^2 R_Lb(L^2,\cdot) e^{-\int_{-\infty}^x R_Lb(L^2,\cdot)} h(x)||_{L^2} &\le ||x R_L b(L^2,x)||_{L^\infty} ||e^{-\int_{-\infty}^x R_Lb(L^2,\cdot)}||_{L^\infty} ||xh||_{L^2} \\
&\le C(\phi_d) ||h_x x^2||_{L^2}. 
\end{align}

We have used Hardy for rapidly decaying functions at $|x| \rightarrow \infty$: 
\begin{align}
||xh||_{L^2}^2 = ||xh1_{x \ge 0}||_{L^2}^2 + ||xh1_{x \le 0}||_{L^2}^2 \lesssim \int_0^\infty x \Big( \int_x^\infty h_x(\theta) d\theta\Big)^2 dx \lesssim ||x^2 h_x||_{L^2}^2. 
\end{align}

In the opposite direction: 
\begin{align} \nonumber
|| d\mathcal{N}|_{b_0}  w||_{\partial_x H^2(2)} &\le ||w e^{\int_{-\infty}^x b_0}||_{H^2(2)} + ||b_0 \int_{-\infty}^x w \ud y \cdot e^{\int_{-\infty}^x b_0}||_{H^2(2)} \\ \label{del.1} 
&\le C(\phi_d) \Big( ||w||_{H^2(2)} \Big).
\end{align}

For the estimate in (\ref{del.1}) the most delicate case occurs for $L^2(2)$. In this event, we split the integration: 
\begin{align} \nonumber
||x^2 b_0 e^{\int_{-\infty}^{x} b_0} &\int_{-\infty}^x w \ud y||_{L^2} \\
& \le ||x^2 b_0 e^{\int_{-\infty}^{x} b_0} \int_{-\infty}^x w \ud y||_{L^2(x \le 0)} + ||x^2 b_0 e^{\int_{-\infty}^{x} b_0} \int_{-\infty}^x w \ud y||_{L^2(x \ge 0)} \\
& \lesssim ||xb_0||_{L^\infty} \Big( ||x \int_{-\infty}^x w \ud y||_{L^2(x \le 0)} + ||x \int_{-\infty}^x w \ud y||_{L^2(x \ge 0)} \Big) \\ \label{first.ab}
& \lesssim ||xb_0||_{L^\infty} \Big(  ||x \int_{-\infty}^x w \ud y||_{L^2(x \le 0)} + ||x \int_{x}^\infty w \ud y||_{L^2(x \ge 0)} \Big) \\ \label{first.ab.1}
& \lesssim ||xb_0||_{L^\infty} \Big(  ||x^2 w||_{L^2} \Big).
\end{align}

Note crucially we have used the mean zero assumption in (\ref{first.ab}). The estimate (\ref{first.ab.1}) then follows via Hardy's inequality. The remaining components of the $H^2(2)$ norm are handled similarly (they are simpler), so we omit those details. 

\end{proof}

First we establish the following relations, which should be compared with (\ref{commb1}):
\begin{lemma}[Commutation of $R_L$] For any profile, $b$,
\begin{align} \label{contr.1}
R_L d\mathcal{N}|_b w = L d\mathcal{N}|_{R_L b} R_L w, \hspace{3 mm} R_L (d\mathcal{N}|_b)^{-1} h = \frac{1}{L} \Big( d\mathcal{N}|_{R_Lb} \Big)^{-1} R_L h
\end{align}
\end{lemma}
\begin{proof}

The second relation in (\ref{contr.1}) follows from the first, through invertibility of $d\mathcal{N}|_b$. The first follows from the calculation:  
\begin{align} \nonumber
R_L d\mathcal{N}|_b w &= R_L e^{-\int_{-\infty}^x b(t,y) dy} \int_{-\infty}^x w(t,y) dy = L e^{-\int_{-\infty}^{Lx} b(t, y) dy} \int_{-\infty}^{Lx} w(t, y) dy \\
&= L e^{-\int_{-\infty}^x R_Lb(t,y)dy} \int_{-\infty}^x Lw(t,Ly)dy = L d\mathcal{N}|_{R_Lb} R_L w. 
\end{align}

\end{proof}

We are now ready to prove Proposition \ref{prop.Burg.1}:

\begin{proof}[Proof of Proposition \ref{prop.Burg.1}]

We start with:
\begin{align} \nonumber
||R_L \Phi_b(L^2-1) &g ||_{H^2(2)} \\ \nonumber
& \stackrel{(\ref{rep.form})}{=} ||R_L \Big\{ \Big[ (d\mathcal{N}|_{b(L^2,\cdot)})^{-1} \circ e^{\Delta(L^2-1)} \circ d\mathcal{N}|_{b_0}  \Big] g\Big\}  ||_{H^2(2)} \\ \nonumber
&\stackrel{(\ref{contr.1})}{=} \frac{1}{L}  || \Big\{ \Big[ (d\mathcal{N}|_{R_L b(L^2,\cdot)})^{-1} \circ R_L e^{\Delta(L^2-1)} \circ d\mathcal{N}|_{b_0}  \Big] g\Big\}  ||_{H^2(2)} \\  \nonumber
&\stackrel{(\ref{est.dn.1})}{\le} \frac{1}{L}C(\phi_d) \Big( || R_L e^{\Delta(L^2-1)} \circ d\mathcal{N}|_{b_0} g ||_{\partial_x H^2(2)}  \Big) \\ \nonumber
&= \frac{1}{L}C(\phi_d) \Big( ||\partial_x R_L e^{\Delta(L^2-1)} d\mathcal{N}|_{b_0} g||_{H^2(2)}  \Big)\\ \label{seq.2}
&\stackrel{(\ref{commb1})}{=} C(\phi_d) ||R_L e^{\Delta(L^2-1)} \partial_x d\mathcal{N}|_{b_0} g||_{H^2(2)} .
\end{align}

Observe that since $g$ is mean-zero, $d\mathcal{N}|_b g$ approaches $0$ as $x \rightarrow \pm \infty$ by (\ref{dN}), and so 
\begin{equation}
\int_{-\infty}^\infty \partial_x d\mathcal{N}|_{b_0} g = 0. 
\end{equation}

Thus, we may use the contraction of the heat semigroup to proceed: 
\begin{align*} 
||R_L  \Phi_b(L^2-1) g ||_{H^2(2)} &\stackrel{(\ref{contr.1.1})}{\le} \frac{C(\phi_d)}{L} ||\partial_x d\mathcal{N}|_{b_0} g||_{H^2(2)}  \\ 
 &= \frac{C(\phi_d)}{L} ||d\mathcal{N}|_{b_0} g||_{\partial_x H^2(2)} \\
 & \stackrel{(\ref{est.dn.2})}{\le} \frac{C(\phi_d)}{L} ||g||_{H^2(2)}. 
\end{align*} 

For the application of estimate (\ref{est.dn.2}), the mean zero feature of $g$ is crucial. This concludes the proof of Proposition \ref{prop.Burg.1}. 

\end{proof}

\subsection{The Operator $\Phi_b(t-1) $: Smoothing Effects}

The purpose of the present subsection is to obtain smoothing estimates for the operator $\Phi_b(t-1)$ near $t = 1$. The precise estimate that we need is presented now: 
\begin{proposition}[$L^2$ Smoothing Estimate] \label{prop.smooth} For all $t > 1$, 
\begin{align} \label{Reg.1}
||\partial_x^k \Phi_b(t-1)  a_0||_{L^2} \le C(\phi_d) (t-1)^{-\frac{k}{2}}  ||a_0||_{L^1 \cap L^2}.
\end{align} 
\end{proposition}

This proposition should be compared with the known smoothing estimates for the heat kernel. Given any $h_0 \in L^p(\mathbb{R})$, we have by Young's inequality for convolution: 
\begin{align} \label{heat.smooth.1}
||\partial_x^k e^{\Delta t} h_0||_{L^p} &\lesssim || \partial_x^k \Gamma_t ||_{L^1} ||h_0 ||_{L^p} \lesssim (t-1)^{-\frac{k}{2}} ||h_0||_{L^p}.
\end{align}

We will now build the relevant machinery to prove our smoothing estimate. The first task is to obtain commutator relations between differential operators $\partial_x$ and the semigroup $\Phi_b(t-1) $.  
\begin{lemma}[{Commutators Relations}]
\begin{align} \label{comm.1}
[\partial_x, d\mathcal{N}|_{b_0}] a_0 &= \frac{1}{2}b(t,y)d\mathcal{N}|_b a_0, \\ \label{comm.2}
\Big[\partial_x, \Big(d\mathcal{N}|_{b(t)} \Big)^{-1}\Big]h &= e^{-\frac{1}{2}\int_{-\infty}^x b(t,y)dy} \left( - 2h_xb(t) - hb(t)_x + \frac{hb(t)^2}{2} \right), \\ \label{comm.3}
[\partial_x, \Phi_b(t-1) ] a_0 &= \mathcal{S} a_0, 
\end{align}
where $\mathcal{S} a_0$ contains no derivatives of $a_0$, and whose exact expression is given below in (\ref{exact.S.1}).
\end{lemma}
\begin{proof}

The starting point are the expressions in (\ref{dN}) - (\ref{dN.inv}). First,
\begin{align} \nonumber
 \partial_x(d\mathcal{N}|_{b_0}a_0) &= \partial_x \left( \exp(\frac{1}{2}\int_{-\infty}^x b(t,y) dy ) \int_{-\infty}^x a_0(y)dy \right) \\ \nonumber 
 &= a_0(x)  \exp(\frac{1}{2}\int_{-\infty}^x b(t,y) dy ) + \frac{1}{2}b(t,y)\left( \int_{-\infty}^x a_0(y)dy \right) \exp(\frac{1}{2}\int_{-\infty}^x b(t,y) dy) \\
 & = d\mathcal{N}|_b \partial_x a_0 + \frac{1}{2}b(t,y)d\mathcal{N}|_b a_0.
\end{align}

Next, we have:
\begin{align} \nonumber
\partial_x d\mathcal{N}|_{b(t)}^{-1}h &=e^{-\frac{1}{2}\int_{-\infty}^x b(t,y)dy} \left( 2h_{xx} - 2h_xb(t) - hb_x(t) + \frac{hb(t)^2}{2} \right) \\
&= \Big(d\mathcal{N}|_{b(t)} \Big)^{-1}  \partial_x h + e^{-\frac{1}{2}\int_{-\infty}^x b(t,y)dy} \left( - 2h_xb(t) - hb(t)_x + \frac{hb(t)^2}{2} \right).
\end{align}

We now turn to the representation formula for our semigroup in (\ref{rep.form}), from which we may start the calculation: 
\begin{align} \nonumber
\partial_x \Phi_b(t-1)  a_0 &= \partial_x \Big\{ \Big(d\mathcal{N}|_{b(t)} \Big)^{-1} e^{\Delta(t-1)} \mathcal{N}|_{b_0} a_0 \Big\} \\ \nonumber
&=  \Phi_b(t-1)  \partial_x a_0 - \Big(d\mathcal{N}|_{b(t)}\Big)^{-1} e^{\Delta(t-1)} [\partial_x, d\mathcal{N}|_{b_0}] a_0 - [\partial_x, \Big(d\mathcal{N}|_{b(t)}\Big)^{-1}] e^{\Delta(t-1)} d\mathcal{N}|_{b_0}a_0 \\
&= \Phi_b(t-1) \partial_x a_0 + \mathcal{S}a_0,
\end{align}

where $\mathcal{S}a_0$ is smoother, i.e. contains no derivatives of $a_0$:
\begin{align} \nonumber
\mathcal{S}a_0 &:= - \Big(d\mathcal{N}|_{b(t)}\Big)^{-1} e^{\Delta(t-1)} [\partial_x, d\mathcal{N}|_{b_0}] a_0 - [\partial_x, \Big(d\mathcal{N}|_{b(t)}\Big)^{-1}] e^{\Delta(t-1)} d\mathcal{N}|_{b_0}a_0 \\ \nonumber
&= -\frac{1}{2}\Big(d\mathcal{N}|_{b(t)}\Big)^{-1} e^{\Delta(t-1)} b_0 d\mathcal{N}|_{b_0} a_0 - e^{-\frac{1}{2}\int_{-\infty}^x b(t,y) dy} \Big[ 2b(t) e^{\Delta(t-1)}b_0 d\mathcal{N}|_{b_0}a_0 
\\ \label{exact.S.1} & \hspace{4 mm} + b_x(t) e^{\Delta(t-1)} d\mathcal{N}|_{b_0} a_0 + b(t)^2 e^{\Delta(t-1)} d\mathcal{N}|_{b_0} a_0 -2b(t) e^{\Delta(t-1)} e^{-\int_{-\infty}^x b(t,y) dy} a_0 \Big].
\end{align}

\end{proof}

We now quantify the property that $\mathcal{S}$ gains one derivative of regularity: 
\begin{lemma} Let $\mathcal{S}$ be as in (\ref{exact.S.1}). Then 
\begin{equation} \label{S.smoothing}
||Sa_0||_{L^2} \le C(\phi_d)||a_0||_{L^2 \cap L^1}, \text{ and } ||\partial_x Sa_0||_{L^2} \le C(\phi_d) (t-1)^{-\frac{1}{2}}||a_0||_{L^2 \cap L^1}. 
\end{equation}
\end{lemma}
\begin{proof}

This follows via a direct calculation using (\ref{exact.S.1}). 

\end{proof}

We now give the proof of Proposition \ref{prop.smooth}:
\begin{proof}[Proof of Proposition \ref{prop.smooth}]
We shall compute this directly via: 
\begin{align} \nonumber
||\partial_x \Phi_b(t-1)  a_0||_{L^2} &\stackrel{(\ref{rep.form})}{=} ||\partial_x \Big(d\mathcal{N}|_{b(t)} \Big)^{-1} e^{\Delta(t-1)} d\mathcal{N}|_{b_0} a_0||_{L^2} \\ \nonumber
&\le ||\Big(d \mathcal{N}|_{b(t)} \Big)^{-1} \partial_x e^{\Delta(t-1)} d\mathcal{N}|_{b_0} a_0||_{L^2} \\ \nonumber
& \hspace{20 mm} + || \Big[\partial_x, \Big(d\mathcal{N}|_{b(t)} \Big)^{-1} \Big] e^{\Delta(t-1)} d\mathcal{N}|_{b_0} a_0||_{L^2} \\ \label{this}
& = (\ref{this}.1) + (\ref{this}.2). 
\end{align}

First, 
\begin{align} \nonumber
(\ref{this}.1) &\stackrel{(\ref{est.dn.3})}{\le} C(\phi_d) ||\partial_x e^{\Delta(t-1)} d\mathcal{N}|_{b_0} a_0||_{\partial_x L^2 \cap L^\infty} \\
& \stackrel{(\ref{heat.smooth.1})}{\le} C(\phi_d) (t-1)^{-\frac{1}{2}}|| d\mathcal{N}|_{b_0} a_0||_{\partial_x L^2 \cap L^\infty} \\
&\stackrel{(\ref{est.dn.3})}{\le} C(\phi_d) (t-1)^{-\frac{1}{2}}||a_0||_{L^1 \cap L^2}. 
\end{align}

Next, for the commutator estimate:
\begin{align} \nonumber
(\ref{this}.2) &\stackrel{(\ref{comm.2})}{\le} C(\phi_d) \Big[ ||\partial_x \{ e^{\Delta(t-1)} d\mathcal{N}|_{b_0} a_0 \} b(t)||_{L^2} +  \\ \nonumber
& \hspace{30 mm} ||b_x(t) e^{\Delta(t-1)} d\mathcal{N}|_{b_0} a_0||_{L^2} + ||b(t)^2 e^{\Delta(t-1)} d\mathcal{N}|_{b_0} a_0||_{L^2} \Big] \\
&\stackrel{(\ref{est.dn.3})}{\le} C(\phi_d) (t-1)^{-\frac{1}{2}}||a_0||_{L^1 \cap L^2}. 
\end{align}

The claim is proven for $k = 1$. The general case follows in the same manner.

\end{proof}

We shall now give some supplementary lemmas which will be used to treat the Burger's nonlinearity (see estimate (\ref{eps.select})).

\begin{lemma}[$L^1$ Bounds for $\Phi_{\bar{b}}(t-1)$] 
\begin{equation}
||\partial_x^j \Phi_{\bar{b}}(t-1)\partial_x^k \alpha_0||_{L^1} \le C(\phi_d) (t-1)^{-\frac{j}{2}-\frac{k}{2}} ||\alpha_0||_{L^1}. 
\end{equation}
\end{lemma} 
\begin{proof}
Via the representation formula (\ref{rep.form}),  
\begin{align} \nonumber
||\Phi_{\bar{b}}\alpha_0||_{L^1} &= ||d\mathcal{N}|_{b(t)}^{-1} \circ e^{\Delta(t-1)} \circ d\mathcal{N}|_{b_0} \alpha_0||_{L^1} \\  \nonumber &\le ||d\mathcal{N}|_{b(t)}^{-1}||_{op} ||e^{\Delta(t-1)} d\mathcal{N}|_{b_0} \alpha_0||_{L^\infty} + ||d\mathcal{N}|_{b(t)}^{-1}||_{op} ||e^{\Delta(t-1)}\partial_x d\mathcal{N}|_{b_0}\alpha_0||_{L^1}\\
& \le C||d\mathcal{N}|_{b_0}\alpha_0||_{L^\infty} + C||\partial_x d\mathcal{N}|_{b_0}\alpha_0||_{L^1} \le C(\phi_d)||\alpha_0||_{L^1}.
\end{align}

We now demonstrate the $j = 0, k = 1$ case. Using the above commutator expressions, we have:
\begin{align}  \nonumber
||\Phi_{\bar{b}} \partial_x \alpha_0||_{L^1} &= || d\mathcal{N}|_{b(t)}^{-1} \circ e^{\Delta(t-1)} \circ d\mathcal{N}|_{b_0} \partial_x \alpha_0||_{L^1} \\ \label{dn.1}
&\le ||d\mathcal{N}|_{b(t)}^{-1}||_{op} \Big( ||e^{\Delta(t-1)} d\mathcal{N}|_{b_0} \partial_x \alpha_0||_{L^\infty} + ||e^{\Delta(t-1)}\partial_x d\mathcal{N}|_{b_0} \partial_x \alpha_0||_{L^1} \Big).
\end{align}
For the first term in (\ref{dn.1}) 
\begin{align} \nonumber
||d\mathcal{N}|_{b(t)}^{-1}||_{op}& ||e^{\Delta(t-1)} d\mathcal{N}|_{b_0} \partial_x \alpha_0||_{L^\infty} \\ \nonumber
&\le ||d\mathcal{N}|_{b(t)}^{-1}||_{op} \left( || e^{\Delta(t-1)}\partial_x(d\mathcal{N}|_{b_0} \alpha_0) ||_{L^\infty} + \frac{1}{2}||e^{\Delta(t-1)} d\mathcal{N}|_{b_0} \alpha_0(\cdot) b_0(\cdot)  ||_{L^\infty} \right) \\ 
\nonumber &\le ||d\mathcal{N}|_{b(t)}^{-1}||_{op} \left( Ct^{-1/2}||d\mathcal{N}|_{b_0}\alpha_0||_{L^\infty}  + t^{-\frac{1}{2}} ||b_0||_{L^1} ||d\mathcal{N}|_{b_0}\alpha_0||_{L^\infty}\right) \\
&\le C(\phi_d) t^{-1/2} ||\alpha_0||_{L^1}.
\end{align}

For the second term in (\ref{dn.1}), 
\begin{align} \nonumber
||e^{\Delta(t-1)} \partial_x d\mathcal{N}|_{b_0} \partial_x \alpha_0||_{L^1} \le t^{-\frac{1}{2}} ||d\mathcal{N}|_{b_0} \partial_x \alpha_0||_{L^1} \le C(\phi_d) t^{-\frac{1}{2}} ||\alpha_0||_{L^1}.
\end{align}

We can iterate the above process for general $j,k$. 

\end{proof}

Via standard interpolation, this then gives:
\begin{corollary}
\begin{equation} \label{l2.new}
||\Phi_{\bar{b}}(t-1)\partial_x \alpha_0||_{L^2} \le C(\phi_d) (t-1)^{-\frac{3}{4}}||\alpha_0||_{L^1}. 
\end{equation}
\end{corollary}

We record the following observation, which we note is not delicate in that we are allowed to pay any factor of $t$ and $\phi_d$ that we want:
\begin{lemma} Fix any $T > 0$. The $H^2(2)$ norm commutes with the linearized flow $\Phi_b$ in the following manner: 
\begin{align} \label{check}
||\Phi_b(T-1)  g||_{H^2(2)} \le C(\phi_d) \Big( \sup_{t \in [1,T]} ||xb(t,x)||_{L^\infty} \Big) J(T) ||g||_{H^2(2)},
\end{align}
where $C(\phi_d), J(T)$ are some factors which depend poorly on $\phi_d, T$ respectively. 
\end{lemma}
\begin{proof}
The commutator of $[\partial_x, \Phi_b(t-1) ]$ is already understood from (\ref{comm.3}) and so we must therefore address the weights. For this, we record the observation that: 
\begin{align} \nonumber
||x e^{\lambda(\partial_x) (t-1)} g||_{L^2} &= ||\partial_\xi \{ e^{-\lambda(\xi) t} \hat{g} \}||_{L^2} \le ||-\xi t e^{-\lambda(\xi) t} \hat{g}||_{L^2} + ||e^{-\lambda(\xi) t} \partial_\xi \hat{g}||_{L^2} \\
&= ||\sqrt{t}\xi e^{-\lambda(\xi) t} \hat{g}\sqrt{t}||_{L^2} + ||e^{-\lambda(\xi) t} \partial_\xi \hat{g}||_{L^2} \le \sqrt{t}||g||_{L^2} + ||xg||_{L^2}. 
\end{align}

Put $w = \Phi_b(t-1) g$. We may then use the definition of $S_b$ and Duhamel to give bounds: 
\begin{align} \nonumber
||xw||_{L^2} &= ||x \Phi_b(t-1)  g||_{L^2} \le ||x e^{\lambda(\partial_x)(t-1)} g||_{L^2} + \int_1^{t} || x e^{\lambda(\partial_x) (t-s)} b w ||_{L^2} \\
&\le ||\Big(x+\sqrt{t}\Big)g||_{L^2}  + \int_1^{t} ||\Big(x + \sqrt{t-s}\Big) bw||_{L^2} ds.
\end{align}

We may now absorb the $x$ into the Burgers term $b$, and pay factors of $t$ to obtain: 
\begin{equation} \label{weight.1}
||xw||_{L^2} \le J(t)||xg||_{L^2} + C(\phi_d) J(t) \sup_{1 \le s \le t}||xb||_{L^\infty} \sup_{1 \le s \le t} ||w||_{L^2} \le J(t) C(\phi_d) ||xg||_{L^2}. 
\end{equation}

For the weight of $x^2$, we simply repeat the calculation, absorbing again one weight of $x$ into $b$ and the other weight of $x$ into $w$ in (\ref{weight.1}), which has been controlled inductively. Repeating this calculation for higher derivatives gives the result.

\end{proof}

Pairing with (\ref{criteria.2}), we can remove the $\sup_{t \in [1,T]} ||xb(t,x)||_{L^\infty}$ term from above, to obtain: 
\begin{corollary} Let $\bar{b}^{(n)}$ be as in (\ref{explicit.bn}). Fix any $T > 0$. Then for any $g \in H^2(2)$,
\begin{equation} \label{check.2}
||\Phi_{\bar{b}^{(n)}}(T-1)  g||_{H^2(2)} \le C(\phi_d) J(T) ||g||_{H^2(2)}.
\end{equation}
\end{corollary}

\section{Step 7: Asymptotic Convergence} \label{step.final}

Our point of view for this step is starting at iterate $n = N_0 + 1$, where $N_0$ is defined as in Step 6. We cannot continue the procedure from Steps 1-6 due to Remark \ref{remk.tr}. Due to (\ref{N0conc}), we have as starting data for this step: 
\begin{equation} \label{ind.1}
||g^{n}||_{H^2(2)} + ||\tilde{g}^{(n,s)}||_{B_{L^n}(2,2,2)} \le \epsilon. 
\end{equation}

The bound (\ref{ind.1}) is the inductive hypothesis for this step. Let us also mention that this step will be the use for the machinery developed in Section \ref{s.ff}. The goal of this section will be to bootstrap the smallness in (\ref{ind.1}) to obtain the asymptotics:
\begin{align} \label{imply.1}
||g^{n}||_{H^2(2)} + ||\tilde{g}^{(n,s)}||_{B_{L^n}(2,2,2)} \lesssim L^{-n(1-\sigma)}, \text{ for some } \sigma > 0. 
\end{align}

Upon using definition of $g^{(n)}, \tilde{g}^{(n,s)}$ found in (\ref{eqn.stable.1}), estimate (\ref{imply.1}) would imply: 
\begin{align}
||L^na(L^{2n}, L^nz)||_{H^2(2)} + ||L^{n(1-p)}\tilde{u}^s(L^{2n}, \frac{\xi}{L^n}, \nu)||_{B_{L^n}(2,2,2)} \le L^{-n(1-\sigma)}.
\end{align}

Define now the following terms: 
\begin{align} \label{defn.c}
\mathcal{C} := p_{mf}^s(\frac{\xi}{L^n}) i \beta \xi \Big( \hat{\bar}{b}^{c,n} \ast \hat{a}^{(n)} \Big) + p_{mf}^c(\frac{\xi}{L^n}) i \beta \xi \Big( \hat{\bar}{b}^{s,n} \ast \hat{a}^{(n)} \Big) = \mathcal{C}_1 + \mathcal{C}_2. 
\end{align}

By rewriting slightly the system (\ref{eqn.c.alpha}) - (\ref{eqn.stable.1}), we have:
\begin{align} \nonumber
a^{(n)}(L^2, z) &= \Phi_{\bar{b}^{(n)}}(L^2-1)  g^{(n)}  + \int_1^{L^2} \Phi_{\bar{b}^{(n)}}(L^2-s)  a^{(n)} a^{(n)}_z \ud s  \\ \nonumber
& + \int_1^{L^2} \Phi_{\bar{b}^{(n)}}(L^2-s)  N^{(n)}(u^{n,c}, u^{n,s}) \ud s  + \int_1^{L^2} \Phi_{\bar{b}^{(n)}}(L^2-s)  \mathcal{C} \ud s \\ \nonumber
&+  \int_1^{L^2} \Phi_{\bar{b}^{(n)}}(L^2-s)  \mathcal{F}^{-1} \{L^{-n}o(\xi^3) \hat{\bar{b}}^{(n)} \} \ud s \\ \label{hard.2}
& + \int_1^{L^2} \Phi_{\bar{b}^{(n)}}(L^2-s)  \mathcal{F}^{-1}\{p_{mf}^c i \beta \xi N^c_b \} \ud s.
\end{align}

\begin{remark} It is our goal to estimate each term appearing above. The rough plan is as follows: for the linear term, first term on the right-hand side of (\ref{hard.2}), we will use estimate (\ref{contraction.Sb}). For the quadratic terms, we will use the smallness in (\ref{ind.1}). Note carefully our treatment of the linear terms in (\ref{hard.2}). According to (\ref{eqn.c.alpha}), displaying just the linear terms: 
\begin{align}
\partial_t \hat{a}^{(n)} - \lambda^{(n)} \hat{a}^{(n)} - p_{mf}^c(\xi) i \beta \xi \Big( \hat{a}^{(n)} \ast \hat{\bar{b}}^{n,c} \Big) + \text{ Quadratic Nonlinearities}.
\end{align}

If we were to write instead: 
\begin{align} \label{instead.1}
\hat{a}^{(n)} = e^{\lambda^{(n)}t} \hat{a}^{(n)}(1,\xi) + \int_1^{L^2} e^{\lambda^{(n)}(t-s)} \Big( \hat{a}^{(n)} \ast \hat{\bar{b}}^{n,c} \Big) + \text{ Quadratic Nonlinearities}.
\end{align}
we would have no way to estimate the second term on the right-hand side of (\ref{instead.1}), because $\bar{b}^{n,c}$ is large (see estimate \ref{transient}).
\end{remark}

The terms in $N^{(n)}$ from (\ref{hard.2}), despite being irrelevant to the asymptotics (carrying extra factors of $L^{-n}$) may contain a loss of derivative, as they are quasilinear. For this reason, we now need technical lemmas which extracts the maximal regularity properties for $a^{(n)}$:
\begin{lemma}[Maximal Regularity Bounds] Let $\phi_d < \infty$ be prescribed. There exist universal functions $C_{MR}(\phi_d), J_{MR}(L)$ such that if:  
\begin{align} \label{crit.MR}
L^{-N_0} C_{MR}(\phi_d) J_{MR}(L) \ll 1, 
\end{align}

then for $n \ge N_0$,  the solution to equation (\ref{eqn.c.alpha}) satisfies: 
\begin{align} \label{max.reg.1}
\int_1^{L^2} ||a^{(n)}_{zzz}||_{L^2(2)}^2 \lesssim J(L) \Big[ C(\phi_d) ||g^{(n)}||^2_{H^2(2)} + ||\tilde{g}^{(n,s)}||^2_{B_{L^n}(2,0,2)}+  L^{-n} C(\phi_d) \Big],
\end{align}
\end{lemma}
where $J(L), C(\phi_d)$ are universal functions of $L, \phi_d$. 
\begin{proof}

We differentiate equation (\ref{eqn.c.alpha}) twice in $z$, which corresponds to a multiplication of $-|\xi|^2$ in Fourier, yielding: 
\begin{align}
\Big( \partial_t - \lambda^{(n)} \Big) \Big(-|\xi|^2 \Big) \hat{a}^{(n)} =  \Big(-|\xi|^2 \Big) \Big[\mathcal{M}^{(n)} +  N_1^{(n)} + N_2^{(n)} + N_3^{(n)} + \mathcal{B}^{(n)} \Big].
\end{align}

We now apply the multiplier $\Big( -|\xi|^2 \Big)\hat{a}^{(n)}$ to both sides of the above equation and integrate by parts. This captures control over the required quantities: 
\begin{align} \label{hard.LHS}
\int \Big( \partial_t - \lambda^{(n)} \Big) \Big(-|\xi|^2 \Big) \hat{a}^{(n)} \cdot \Big(-|\xi|^2 \hat{a}^{(n)} \Big) \ge \partial_t \int \Big| a^{(n)} \Big|^2 + \int \Big| a^{(n)}_{zzz} \Big|^2.
\end{align}

Next, for the marginal nonlinearities, we first have via an integration by parts: 
\begin{align} \nonumber
\Big| \int o(\xi^3) \Big( \hat{a}^{(n)} \ast \hat{a}^{(n)} \Big) \cdot \overline{\hat{a}^{(n)} |\xi|^2} \Big| &\le || |\xi|^2 \Big(\hat{a}^{(n)} \ast \hat{a}^{(n)}\Big)||_{L^2} ||a^{(n)}_{zzz}||_{L^2} \\ \nonumber
&\lesssim ||a^{(n)}||_{H^2}^2 ||a^{(n)}||_{H^2}^2 + \frac{1}{100}||a^{(n)}_{zzz}||_{L^2}^2 \\ &\lesssim J(L)C(\phi_d) ||g^{(n)}||_{H^2}^2. 
\end{align}

We next treat the second component of the marginal contribution, for which we again integrate by parts:
\begin{align}
\Big|\int o(\xi^3) \Big( \hat{\bar{b}}^{(n,c)} \ast \hat{a}^{(n)} \Big) \cdot \overline{ |\xi|^2 \hat{a}^{(n)} } \Big| \lesssim ||\bar{b}^{n,c}||_{H^2}^2 ||a^{(n)}||_{H^2}^2 + \frac{1}{100} ||a^{(n)}_{zzz}||_{L^2}^2. 
\end{align}

Summarizing the marginal contributions: 
\begin{align}
\Big| \int \Big( - |\xi|^2 \Big) \mathcal{M}^{(n)} \cdot \overline{ \Big( - |\xi|^2  \Big) \hat{a}^{(n)} } \Big| \le C(\phi_d) J(L) ||g^{(n)}||_{H^2}^2 .  
\end{align}

Next, we control the higher-order quadratic nonlinearity. This term is the most delicate as there is a high-regularity term ($o(\xi^3)$) that we must contend with:
\begin{align} \nonumber
&\Big| \int L^{-n}o(\xi^4) \Big( \{ \hat{a}^{(n)} + \hat{\bar{b}}^{(n,c)} \} \ast \{ \hat{a}^{(n)} + \hat{\bar{b}}^{(n,c)} \} \Big) \cdot \overline{ |\xi|^2 \hat{a}^{(n)} } \Big| \\ \nonumber
&\hspace{10 mm} =\Big| \int L^{-n}o(\xi^3) \Big( \{ \hat{a}^{(n)} + \hat{\bar{b}}^{(n,c)} \} \ast \{ \hat{a}^{(n)} + \hat{\bar{b}}^{(n,c)} \} \Big) \cdot \overline{ o(\xi^3) \hat{a}^{(n)} } \Big| \\
& \hspace{10 mm} \le L^{-n} ||\hat{a}^{(n)}||_{L^1_\xi} ||a^{(n)}||_{H^3}^2 + L^{-n}C(\phi_d).
\end{align}

Now, upon making the observation that $||\hat{a}^{(n)}||_{L^1_\xi} \le ||a^{(n)}||_{H^2} \le C(\phi_d)$, and that $n \ge N_0$, coupled with (\ref{crit.MR}), we may absorb the highest (third) order derivative term from the above estimate to the left-hand side of (\ref{hard.LHS}). Next, for the stable-critical quadratic interaction: 
\begin{align} \nonumber
&L^{-n(1-p)}\Big|\int o(\xi^3) \Big(\tilde{u}^{n,s} \ast \hat{u}^{n,c}\Big) \cdot \overline{|\xi|^2 \hat{a}^{(n)}} \Big|  = L^{-n(1-p)}\Big|\int o(\xi^2) \Big( \tilde{u}^{n,s} \ast \hat{u}^{n,c}\Big) \cdot \overline{o(\xi^3)\hat{a}^{(n)}} \Big| \\ \nonumber
& \hspace{10 mm} \le L^{-n(1-p)} \Big[ ||u^{n,c}||_{H^2}||\tilde{u}^{n,s}||_{L^1_\xi(1)} ||a^{(n)}_{zzz}||_{L^2} + ||\hat{u}^{n,c}||_{L^1_\xi} ||\tilde{u}^{n,s}||_{B_{L^n}(2,0,0)} ||a^{(n)}_{zzz}||_{L^2} \Big] \\
& \hspace{10 mm} \lesssim \frac{1}{100} ||a^{(n)}_{zzz}||_{L^2}^2 +  L^{-2n(1-p)} \Big[ C(\phi_d) + ||\tilde{u}^{n,s}||_{B_{L^n}(2,0,0)}^2 + J(L)||g^{(n)}||_{H^2}^2 \Big]. 
\end{align}

The cubic nonlinearity in $N_3$ can be handled similar to above, by noting that each additional power in the nonlinearity produces a factor of $L^{-n}$, which exactly cancels the largeness of the profiles. The rapidly decaying Burgers terms in $\mathcal{B}^{(n)}$ are accompanied by factors of $L^{-n}$, and so their estimate is immediate. The weighted estimates follow similarly by applying weighted multipliers. 

\end{proof}

With maximal regularity in hand, our aim is to give estimates on the irrelevant nonlinearities. 
\begin{lemma}[Irrelevant Nonlinearity Bounds] \label{nl.n} Suppose the inductive hypothesis in (\ref{ind.1}), the criteria (\ref{crit.MR}) for $N_0$, and that $n \ge N_0$. Then, the irrelevant nonlinearities $N^{(n)}_i, i = 1,...4$ can be controlled as follows: 
\begin{align} \nonumber
||\int_1^{L^2} \Phi_{\bar{b}^{(n)}}(L^2-s)  N^{(n)}(u^{n,c}, u^{n,s}) ||_{H^2(2)} \lesssim & J(L) C(\phi_d) \epsilon ||g^{(n)}||_{H^2(2)} + J(L) L^{-n} C(\phi_d) \\ 
& + J(L) 2\phi_d L^{-n(1-p)} ||\tilde{g}^{(n,s)}||_{B_{L^n}(2,2,2)}, 
\end{align}
where $C(\phi_d)$ and $J(L)$ are universal functions of their arguments. 
\end{lemma}
\begin{proof}

Recall the breakdown of the irrelevant nonlinearity into $N^{(n)} = N^{(n)}_1 + N^{(n)}_2 + N^{(n)}_3$ in (\ref{nl.1}) - (\ref{nl.4}). The estimate for $N_1^{(n)}$ is the most delicate as there is a potential loss of a derivative due to the quasilinear nature of the nonlinearity. In detail, 
\begin{align}
N_1^{(n)} = L^{-n} o(\xi^2) \Big[ \Big( \hat{a}^{(n)} \ast \hat{a}^{(n)} \Big) + \Big( \hat{a}^{(n)} \ast \hat{\bar{b}}^{(n,c)}\Big) + \Big(\hat{\bar{b}}^{(n,c)} \ast \hat{\bar{b}}^{(n,c)} \Big)  \Big]. 
\end{align}

For the purely nonlinear component, we may trade a factor of $L^{-n}$ for one derivative, $\xi$: 
\begin{align} \nonumber
&|| \int_1^{L^2} \Phi_{\bar{b}^{(n)}}(L^2-s)  L^{-n} o(\xi^2) \Big(\hat{a}^{(n)} \ast \hat{a}^{(n)}\Big) ||_{H^2(2)} \le J(L) C(\phi_d) \epsilon ||g^{(n)}||_{H^2(2)},
\end{align}

where we have applied the inductive hypothesis in (\ref{ind.1}).  For the purely Burgers contribution above, we may give the trivial estimate: 
\begin{align}
|| \int_1^{L^2} \Phi_{\bar{b}^{(n)}}(L^2-s)  L^{-n} o(\xi^2) \Big(\hat{\bar{b}}^{(c,n)} \ast \hat{\bar{b}}^{(c,n)}\Big) ||_{H^2(2)} \le J(L) C(\phi_d) L^{-n}. 
\end{align}

We must now control the cross-term, for which we require the maximal regularity lemma from above as well as the regularization estimate in (\ref{Reg.1}). It suffices to treat the highest-order (two derivatives coming from the norm) term,  which is the most delicate. Moreover, we can remove the weight of $z^2$ in this estimate, as the $z^2$ estimate follows similarly, after noting that the weight does not affect the regularization near $t = 0$, and we are free to pay factors of $J(L)$ and $C(\phi_d)$. We fix a $p > 0$ and now turn to the estimate: 
\begin{align} \nonumber
&|| \int_1^{L^2} \Phi_{\bar{b}^{(n)}}(L^2-s)  L^{-n}o(\xi^2) \Big(\hat{\bar{b}}^{(n,c)} \ast \hat{a}^{(n)}\Big) ds||_{H^2(2)} \\ & \hspace{10 mm} =  \int_1^{L^2} || o(\xi^2) \Phi_{\bar{b}^{(n)}}(L^2-s)  L^{-n} o(\xi^2) \Big( \hat{\bar{b}}^{(n,c)} \ast \hat{a}^{(n)}\Big) ||_{L^2} + \text{l.o.t.$(\xi)$} \\ \label{n1.1}
& \hspace{10  mm} \le L^{-n(1-p)} \int_1^{L^2} ||o(\xi^{2-p}) \Phi_{\bar{b}^{(n)}}(L^2-s)  o(\xi^2) \Big(\hat{\bar{b}}^{(n,c)} \ast \hat{a}^{(n)}\Big) ||_{L^2}  + \text{l.o.t($\xi$)} \\   \nonumber
& \hspace{10 mm} \le L^{-n(1-p)} \int_1^{L^2} ||o(\xi^{1-p}) \Phi_{\bar{b}^{(n)}}(L^2-s)  o(\xi^3)  \Big(\hat{\bar{b}}^{(n,c)} \ast \hat{a}^{(n)}\Big) ||_{L^2}  \\ \label{n1.2}
& \hspace{20 mm} + L^{-n(1-p)} \int_1^{L^2} ||o({\xi^{1-p}}) \mathcal{S} o(\xi^2) \Big(\hat{\bar{b}}^{(n,c)} \ast \hat{a}^{(n)}\Big)||_{L^2}+ \text{l.o.t($\xi$)}. 
\end{align}

In (\ref{n1.1}), we have traded a factor of $L^{np}\xi^{p}$ for an order one constant, providing necessary additional regularity. In (\ref{n1.2}), we have used the commutator expansion given in (\ref{comm.3}). For the first term in (\ref{n1.2}), we must interpolate via: 
\begin{align} \nonumber
||o(\xi^{1-p}) &\Phi_{\bar{b}^{(n)}}(L^2-s)  o(\xi^3) \Big(\hat{\bar{b}}^{(n,c)} \ast \hat{a}^{(n)}\Big) ||_{L^2} \\ \nonumber
&\le ||o(\xi) \Phi_{\bar{b}^{(n)}}(L^2-s)  o(\xi^3) \Big(\hat{\bar{b}}^{(n,c)} \ast \hat{a}^{(n)}\Big)||_{L^2}^{\theta(p)} || \Phi_{\bar{b}^{(n)}}(L^2-s)  o(\xi^3)\Big( \hat{\bar{b}}^{(n,c)} \ast \hat{a}^{(n)}\Big)||_{L^2}^{1-\theta(p)} \\ \label{maj.1}
&\le  (L^2-s)^{-\frac{\theta(p)}{2}} ||o(\xi^3) \Big(\hat{\bar{b}}^{(n,c)} \ast \hat{a}^{(n)}\Big)||_{L^2 \cap L^1} \le  (L^2-s)^{-\frac{\theta(p)}{2}} ||\bar{b}^{(n,c)} a^{(n)}||_{H^3}. 
\end{align} 

where $\theta(p) \rightarrow 1$ as $p \rightarrow 0$, but $\theta(p) < 1$ so long as $p >0$. We have used our regularization estimates from (\ref{Reg.1}): 
\begin{align}
 ||o(\xi) \Phi_{\bar{b}^{(n)}}(L^2-s) o(\xi^3) \Big(\hat{\bar{b}}^{(n,c)} \ast \hat{a}^{(n)}\Big)||_{L^2}^{\theta(p)} \le  (L^2-s)^{-\frac{\theta(p)}{2}} ||o(\xi^3)\Big( \hat{\bar{b}}^{(n,c)} \ast \hat{a}^{(n)}\Big)||_{L^2 \cap L^1}^{\theta(p)},
\end{align}

and also the Sobolev embedding $H^2 \hookrightarrow \dot{W}^{1,1}$. The smoothing term in (\ref{n1.2}) can be estimated in a similar manner, after noting the bound (\ref{S.smoothing}). Inserting (\ref{maj.1}) into (\ref{n1.2}), we arrive at: 
\begin{align} \nonumber
&|| \int_1^{L^2} \Phi_{\bar{b}^{(n)}}(L^2-s) L^{-n}o(\xi^2) \Big( \hat{\bar{b}}^{(n,c)} \ast \hat{a}^{(n)}\Big) ds||_{H^2(2)} \\  \nonumber & \hspace{10 mm} \le C(\phi_d) L^{-n(1-p)} \int_1^{L^2} (L^2-s)^{-\frac{\theta(p)}{2}} ||\bar{b}^{(n,c)} a^{(n)}||_{H^3} \\ \nonumber
&  \hspace{10 mm} \le C(\phi_d) L^{-n(1-p)} \int_1^{L^2} (L^2-s)^{-\theta(p)} + L^{-n(1-p)} \int_1^{L^2} ||a^{(n)}||_{H^3}^2 ds \\ \label{l.step}
& \hspace{10 mm} \le C(\phi_d) J(L) L^{-n(1-p)} + L^{-n(1-p)} J(L) C(\phi_d) \Big[ ||g^{(n)}||_{H^2(2)}^2 + ||\tilde{g}^{(n,s)}||_{B_{L^2}(2,0,2)}^2 \Big]. 
\end{align}

The interpolation using $L^{np}$ was necessary to force $\theta(p) < 1$, thereby obtaining a convergent integral above. In the last step in (\ref{l.step}), we have used the maximal regularity estimate provided in (\ref{max.reg.1}). Summarizing the $N_1^{(n)}$ estimate: 
\begin{align} \nonumber
\int_1^{L^2}|| &\Phi_{\bar{b}^{(n)}}(L^2-s) N^{(n)}_1(\hat{u}^{n,c}, \hat{u}^{n,c})||_{H^2(2)} \le J(L)L^{-n}C(\phi_d) + \epsilon J(L) C(\phi_d) ||g^{(n)}||_{H^2(2)}.
\end{align}

Next, 
\begin{align} \nonumber
\int_1^{L^2} ||\Phi_{\bar{b}^{(n)}}(L^2-s)  N^{(n)}_2(\hat{u}^{n,c}, \tilde{u}^{n,s})||_{H^2(2)} &\le J(L) L^{-n(1-p)}||\bar{b}^{(n)}(1,\cdot) + g^{(n)}||_{H^2(2)}||\tilde{g}^{(n,s)}||_{B_{L^n}(2,2,2)} \\
&\le J(L) 2\phi_d L^{-n(1-p)}||\tilde{g}^{(n,s)}||_{B_{L^n}(2,2,2)}.
\end{align}

The third irrelevant nonlinearity can be bounded similar to the previous three, by using the extra factor of $L^{-n}$ which arises for each additional power in the nonlinearity. 

\end{proof}

We need one more lemma: 
\begin{lemma} Recall the definition (\ref{defn.c}). Then we have exponential decay: 
\begin{equation}
\int_1^{L^2} ||\Phi_{\bar{b}^{(n)}}(L^2-s)  \mathcal{C} ds||_{H^2(2)} \le J(L) C(\phi_d) L^{-2n} \Big[ ||\tilde{g}^{n,s}||_{B_{L^n}(2,2,2)}  + \epsilon ||g^{(n)}||_{H^2(2)} + C(\phi_d) \Big].
\end{equation}
\end{lemma}
\begin{proof}

Both terms in $\mathcal{C}$ are supported in $ (-\infty, -\frac{l_1}{8}L^{-n})$, and so are expected to decay exponentially. For the second term, we may directly use the exponential Burgers decay in $\bar{b}^{s,n}$ to estimate: 
\begin{align} \label{c2estimate}
\int_1^{L^2} ||\Phi_{\bar{b}^{(n)}}(L^2-s) \mathcal{C}_2 ||_{H^2(2)} \le \epsilon J(L)C(\phi_d) \int_1^{L^2} e^{-L^ns} ds \le \epsilon J(L) C(\phi_d) L^{-n}. 
\end{align}

For the first term, if the multiplier $p_{mf}^s$ falls on the $\hat{\bar{b}}^{c,n}$ term, we may simply repeat the above calculation. If $p_{mf}^s$ falls on the $\hat{a}^{(n)}$, this is then controlled by the stable propagation, $\tilde{u}^{n,s}$. To capture this, we must use an iterated integration: 
\begin{align} \nonumber
\int_1^{L^2}& ||\Phi_{\bar{b}^{(n)}}(L^2-s)  i \beta \xi \hat{\bar{b}}^{c,n} \ast \Big( p_{mf}^s \hat{a}^{(n)}\Big)||_{H^2(2)} \le J(L)C(\phi_d) \int_1^{L^2} ||p^s_{mf} \hat{a}^{(n)}||_{H^2(2)} ds \\ \nonumber
&\le J(L) C(\phi_d) \int_1^{L^2} ||\tilde{u}^{n,s}||_{B_{L^n}(2,2,2)} ds \\ \nonumber
& \le J(L) C(\phi_d) \int_1^{L^2} \Big\{e^{-c_0 L^{2n} (s-1)}||\tilde{g}^{n,s}||_{B_{L^n}(2,2,2)} + \int_1^{s+1} e^{-c_0 L^{2n}(s-1)} ||\tilde{N}^{n,s}||_{B_{L^n}(2,2,2)} \Big\}ds \\
& \le J(L) C(\phi_d) L^{-2n} \Big[ ||\tilde{g}^{n,s}||_{B_{L^n}(2,2,2)}  + ||g^{(n)}||_{H^2(2)}^2 + C(\phi_d) \Big] \\\
& \le J(L) C(\phi_d) L^{-2n} \Big[ ||\tilde{g}^{n,s}||_{B_{L^n}(2,2,2)}  + \epsilon ||g^{(n)}||_{H^2(2)} + C(\phi_d) \Big].
\end{align}

\end{proof}

We may now come to the renormalization iteration: 
\begin{proposition} \label{Prop.FIN} Let $\phi_d, \rho_\ast$ be prescribed. There exists universal functions $J_3(L), J_{E,4}(L), C_4(\phi_d), C_{E,4}(\phi_d), \tilde{C}(\phi_d, L)$ such that if:
\begin{align} \label{crit.2.1}
&||g^{(n)}||_{H^2(2)} + ||\tilde{g}^{(n,s)}||_{B_{L^n}(2,2,2)} \le \epsilon,\\ \label{crit.2.2}
&\epsilon J_3(L) C_3(\phi_d) \ll 1, \\ \label{crit.2.3}
&\frac{\tilde{C}(\phi_d, L)}{L^{N_0 - p - 1}} < \frac{\epsilon}{2}, \\ \label{crit.2.4}
&L^{-N_0} C_{MR}(\phi_d) J_{MR}(L) \ll 1, 
\end{align} 
then:
\begin{align} \label{add.1}
&||g^{(n+1)}||_{H^2(2)} \le \frac{C_{E,4}(\phi_d)}{L} ||g^{(n)}||_{H^2(2)} + \frac{1}{L} ||\tilde{g}^{(n,s)}||_{B_{L^n}(2,2,2)} + \frac{C_{E,4}(\phi_d)}{L^n}, \\ \label{add.2}
&||g^{(n+1,s)}||_{B_{L^n}(2,2,2)} \le \frac{C}{L}||\tilde{g}^{(n,s)}||_{B_{L^n}(2,2,2)} + \frac{1}{L}||g^{(n)}||_{H^2(2)} + J_{E,4}(L)C_{E,4}(\phi_d)L^{-n}.
\end{align}
\end{proposition}
\begin{proof}

We begin again with (\ref{hard.2})
\begin{align} \nonumber
a^{(n)}(L^2, z) &= \Phi_{\bar{b}^{(n)}}(L^2-1)  g^{(n)}  + \int_1^{L^2} \Phi_{\bar{b}^{(n)}}(L^2-s)  a^{(n)} a^{(n)}_z   \\ \nonumber
& + \int_1^{L^2} \Phi_{\bar{b}^{(n)}}(L^2-s)  N^{(n)}(u^{n,c}, u^{n,s})  + \int_1^{L^2} \Phi_{\bar{b}^{(n)}}(L^2-s)  \mathcal{C} ds \\ \nonumber
&+  \int_1^{L^2} \Phi_{\bar{b}^{(n)}}(L^2-s)  \mathcal{F}^{-1} \{L^{-n}o(\xi^3) \hat{\bar{b}}^{(n)} \} \\ \label{hard.1}
& + \int_1^{L^2} \Phi_{\bar{b}^{(n)}}(L^2-s)  \mathcal{F}^{-1}\{p_{mf}^c i \beta \xi N^c_b \}.
\end{align}

Then, applying the map $R_L$ and taking the $H^2(2)$ norm yields first:
\begin{equation}
|| R_L \Phi_{\bar{b}^{(n)}}(L^2-1)  g^{(n)}||_{H^2(2)} \le \frac{C(\phi_d)}{L} ||g^{(n)}||_{H^2(2)},
\end{equation}

via the contractive semigroup estimate in (\ref{contraction.Sb}). The profiles $\bar{b}^{(n)}$ satisfy the criteria in (\ref{criteria}) via (\ref{criteria.2}). For the Burgers nonlinearity, we use (\ref{l2.new}) to give the bound: 
\begin{align} \label{eps.select} 
\int_1^{L^2} ||\Phi_{\bar{b}^{(n)}}(L^2-1)   a^{(n)} a^{(n)}_z ||_{H^2(2)}  \le C(\phi_d) \epsilon J(L) ||g^{(n)}||_{H^2(2)}, 
\end{align}

via the inductive hypothesis in (\ref{ind.1}). The $N^{(n)}$ term can be estimated via Lemma \ref{nl.n}, and the $\bar{b}^{(n)}$ from the final term in (\ref{hard.1}) contributes a factor of $L^{-n}C(\phi_d)$ in essentially the same manner as estimate (\ref{c2estimate}). This all gives:  
\begin{align} \nonumber
||g^{(n+1)}||_{H^2(2)} &\le \frac{C(\phi_d)}{L} ||g^{(n)}||_{H^2(2)} + 2C(\phi_d)J(L)\epsilon ||g^{(n)}||_{H^2(2)} + L^{-n} C(\phi_d) \\ \nonumber 
& \hspace{30 mm} + J(L) C(\phi_d) L^{-n(1-p)}||\tilde{g}^{(n,s)}||_{B_{L^n}(2,2,2)} \\ \nonumber
&\le \frac{C(\phi_d)}{L} ||g^{(n)}||_{H^2(2)} + \frac{C(\phi_d)}{L^n} +  J(L) C(\phi_d) L^{-N_0(1- p)} ||\tilde{g}^{(n,s)}||_{B_{L^n}(2,2,2)} \\ \nonumber
&\le \frac{C(\phi_d)}{L} ||g^{(n)}||_{H^2(2)} + \frac{\epsilon}{2} ||\tilde{g}^{(n,s)}||_{B_{L^n}(2,2,2)} + \frac{C(\phi_d)}{L^n} \\ \label{it.10} 
&\le \frac{C(\phi_d)}{L} ||g^{(n)}||_{H^2(2)} + \frac{1}{L} ||\tilde{g}^{(n,s)}||_{B_{L^n}(2,2,2)} + \frac{C(\phi_d)}{L^n}.
\end{align}

The next step is to treat the stable iteration:  
\begin{align} \nonumber
||g^{(n+1,s)}&||_{B_{L^n}(2,2,2)} \le \frac{C}{L}||\tilde{g}^{(n,s)}||_{B_{L^n}(2,2,2)} + J(L) L^{-n} \Big( ||g^{(n)}||^2_{H^2(2)} + ||\tilde{g}^{(n,s)}||^2_{B_{L^n}(2,2,2)} + C(\phi_d) \Big) \\ \nonumber
&\le  \frac{C}{L}||\tilde{g}^{(n,s)}||_{B_{L^n}(2,2,2)} + J(L) L^{-n} 8\phi_d \Big( ||g^{(n)}||_{H^2(2)} + ||\tilde{g}^{(n,s)}||_{B_{L^n}(2,2,2)}+ C(\phi_d) \Big) \\ \nonumber
& \le  \frac{C}{L}||\tilde{g}^{(n,s)}||_{B_{L^n}(2,2,2)} + J(L) L^{-N_0} 8\phi_d \Big( ||g^{(n)}||_{H^2(2)} + ||\tilde{g}^{(n,s)}||_{B_{L^n}(2,2,2)} \Big) \\ \nonumber
& \hspace{20 mm} + J(L) L^{-n}C(\phi_d)\\ \label{it.11}
&\le  \frac{C}{L}||\tilde{g}^{(n,s)}||_{B_{L^n}(2,2,2)} + \frac{1}{L}||g^{(n)}||_{H^2(2)} + J(L)L^{-n}C(\phi_d). 
\end{align}

The desired estimate follows by appropriately selecting the universal functions. 

\end{proof}

\begin{proposition} \label{Prop.sel} Given any $0 < \phi_d, \rho_\ast < \infty$ arbitrarily large, and $0 < \sigma < 1$, there exist choices of parameters $L, \epsilon, N_0, \delta_w, \delta$ which guarantee that:
\begin{equation}
||L^n a({L^{2n}, L^n z})||_{H^2(2)} \le L^{-n(1-\sigma)}  C(\phi_d, \rho_\ast) \text{ as } n \rightarrow \infty.
\end{equation}
\end{proposition}
\begin{proof}

We first define the universal function:
\begin{align}
C_\ast(\phi_d) = \max \{C_1(\phi_d), C_2(\phi_d), C_3(\phi_d), C_{E,4}(\phi_d), c_0\}, 
\end{align}

and then pick $L$ large enough based only on $\phi_d$, satisfying:
\begin{align}
\frac{C_\ast(\phi_d)}{L^\sigma} \ll 1, \\
L^{\frac{5}{2}+1}e^{-c_1L^2} \ll 1.
\end{align}

It is clear that, given a $\phi_d < \infty$, one can simultaneously satisfy the criteria above by taking $L$ sufficiently large. Next, define the universal function: 
\begin{align}
J_\ast(L) = \max \{J_1(L), J_2(L), J_3(L) \}. 
\end{align}

Select then $\delta_w, \epsilon$ via:
\begin{align}
\delta_w J_\ast(L) C_\ast(\phi_d) \ll 1, \\
\epsilon J_\ast(L) C_\ast(\phi_d) \ll 1. 
\end{align}

At this point, one checks that criteria (\ref{crit.0}), (\ref{crit.1}) - (\ref{crit.2}), and (\ref{crit.2.2}) are all satisfied. 

Select next $N_0$ to simultaneously satisfy the following criteria:  
\begin{align}
&\frac{C_{E,3}(\phi_d) J_{E,3}(L) + \rho_\ast}{L^{N_0(1-\sigma)}} \le \epsilon,  \\
&\frac{\tilde{C}(\phi_d, L)}{L^{N_0 - p - 1}} < \frac{\epsilon}{2}, \\ \label{sel.N0.3}
& \frac{C_{E_4}(\phi_d)}{L^{N_0}}J_{E,4}(L) < \frac{\epsilon}{2}, \\
& L^{-N_0} C_{MR}(\phi_d) J_{MR}(L) \ll 1.
\end{align}

This can easily be done by taking $N_0$ sufficiently large, since we have already fixed the remaining parameters in the relations above. This satisfies all criteria on $N_0$, namely (\ref{eps.1}), (\ref{crit.MR}), and (\ref{crit.2.3}). Finally, select $\delta$ according to: 
\begin{align} \label{delta.deltaw}
\delta = L^{-2N_0} \delta_w. 
\end{align}

According to calculation (\ref{continue}), an application of Theorem \ref{thm.LE} then ensures that for $0 \le k \le N_0$, we satisfy the criteria on the initial data, namely criteria (\ref{crit.ID}). We may apply Corollary \ref{cor.N0} to conclude estimate (\ref{N0conc}). For $n = N_0 + 1$, the inductive hypothesis in (\ref{ind.1}) is satisfied due to our choice of $N_0$: indeed, from (\ref{add.1}) - (\ref{add.2}), and (\ref{sel.N0.3}) we see:  
\begin{align}
||g^{(n+1)}||_{H^2(2)} &\le \frac{C_{E,4}(\phi_d)}{L} \epsilon + \frac{1}{L} \epsilon + \frac{C_{E,4}(\phi_d)}{L^{N_0}} \le \epsilon. 
\end{align}

Similarly, for the inductive hypothesis on $u^{n,s}$, we have from (\ref{it.11}):
\begin{align}
||\tilde{g}^{(n,s)}||_{B_{L^n}(2,2,2)} \le \frac{C}{L}\epsilon + \frac{1}{L}\epsilon + C_{E,4}(\phi_d)J_{E,4}(L) L^{-N_0} \le \epsilon. 
\end{align}

This then enables us to iterate Proposition \ref{Prop.FIN}, which upon adding together (\ref{add.1}) - (\ref{add.2}), and recalling the definition of $\rho^{(n)}$ from (\ref{defn.rho.n})
\begin{align}
\rho^{(n+1)} \le \frac{C_{E,4}(\phi_d)}{L} \rho^{(n)} + \frac{C_{E,4}(\phi_d) J_{E,4}(L)}{L^n}, \text{ for } n \ge N_0.
\end{align}

Iterating this then yields:
\begin{equation}
||L^n a({L^{2n}, L^n z})||_{H^2(2)} \le L^{-n(1-\sigma)}  C(\phi_d, \rho_\ast),
\end{equation}

\end{proof}

This immediately proves estimate (\ref{goal}) upon putting $t = L^n$, and consequently Theorem \ref{thm.main.2}, and subsequently our main result, Theorem \ref{thm.main}. 

\vspace{3 mm}

\section{Discussion and Outlook}

In this paper, we considered spectrally stable wave-train solutions of the form $u(t,x)=u_0(k_0x-\omega_0t)$ of systems of reaction-diffusion equations. We then considered initial data of the form
\begin{equation}\label{d1}
u(0,x) = u_0(k_0x+\phi_0(x))+v_0(x),
\end{equation}
which correspond to phase modulations of the underlying wave train. Our goal was to show that the resulting solution can be written as
\[
u(t,x) = u_0(k_0x+\phi(t,x))+v(t,x)
\]
and to extract the leading-order dynamics of the phase modulation $\phi(t,x)$. More precisely, if the phase modulation  $\phi_0(x)$ has limits $\phi_\pm$ as $x\to\pm\infty$ and if the derivative $\partial_x\phi_0(x)$ is sufficiently small in an appropriate sense, then we proved that $\phi(t,x)$ converges to an explicit front profile, obtained from a self-similar solution of the viscous Burgers equation, that travels with the linear group velocity of the wave train. In contrast to earlier work \cite{Sandstede,Zumbrun3,Zumbrun1,Zumbrun2}, we do not need to assume that the phase offset $\phi_d:=\phi_+-\phi_-$ is small: our stability result is therefore valid for arbitrary phase offsets as long as the phase modulation varies sufficiently slowly. Our proof relies on renormalization group techniques, estimates for the linearization about the front profile of the Burgers equation, and energy estimates.

The remaining open problem is whether these results can be extended to perturbations of the wavenumber. Wave trains generically come in one-parameter families and can therefore be written as $u(t,x)=u_0(kx-\omega(k)t;k)$ for a certain nonlinear function $\omega(k)t$ that relates the spatial wavenumber $k$ and the temporal frequency $\omega$; the profiles may depend on $k$, which is accounted for through the parameter $k$ in the expression $u_0(\cdot;k)$. Instead of selecting the initial condition as in (\ref{d1}), we focus on data of the form
\begin{equation}\label{d2}
u(0,x) = u_0(k_0x+\phi_0(x);k_0+\partial\phi_0(x)) + v_0(x), \qquad \partial_x\phi_0(x)\to k_\pm \mbox{ as } x\to\pm\infty,
\end{equation}
which correspond to wavenumber modulations of the underlying wave train. We will assume that the asymptotic wavenumber offsets $|k_\pm|\ll1$ are small compared to $k_0$. It was then shown in \cite[Theorem~4.8]{DS} that we can write the solution associated with the initial condition (\ref{d2}) in the form
\[
u(x,t) = u_0(k_0x+\phi_0(t,x);k_0+\partial\phi(t,x)) + v(t,x),
\]
where the wavenumber $q(t,x):=\partial_x\phi(t,x)$ satisfies the viscous Burgers equation
\[
q_t = \alpha q_{xx} + \beta \partial_x(q^2).
\]
Furthermore, $v$ stays small on large bounded time intervals: the key question is then whether this estimate can be shown to hold for all times $t$. Our approach will not immediately work as $\partial_x\phi$ will not be bounded in $L^1$. One potential way to address this question is to subtract the anticipated wavenumber profile of the viscous Burgers equation from the wavenumber and focus on the wavenumber offset from this profile, which should be integrable. We leave this for future work.

\paragraph{Acknowledgements:} The first author wishes to thank Yan Guo for many stimulating discussions regarding the energy method in Section \ref{RGA.Section}

\end{document}